\documentclass[a4paper,regno,11pt]{amsart}

\usepackage[margin=1in]{geometry}

\usepackage[english]{babel}
\usepackage{graphicx} %figure, includegraphics
\usepackage{amsthm} %theoremstyle
\usepackage{amsmath}
\usepackage{amssymb}
\usepackage{amsthm} %proof
\usepackage{mathrsfs} %mathscr
\usepackage{bbm} %mathbbm
\usepackage{calligra} %calligra
\usepackage{enumerate} %(i)
\usepackage[shortlabels]{enumitem}
\usepackage{xcolor}
\usepackage{multirow} %\begin{tabular}
\usepackage{subfigure}

\usepackage{hyperref}

\newcommand{\vertiii}[1]{{\left\vert\kern-0.25ex\left\vert\kern-0.25ex\left\vert #1 \right\vert\kern-0.25ex\right\vert\kern-0.25ex\right\vert}}

\newcommand{\N}{\mathbb{N}}
\newcommand{\Z}{\mathbb{Z}}

\newcommand{\R}{\mathbb{R}}
\newcommand{\C}{\mathbb{C}}
\newcommand{\T}{\mathbb{T}}

\newcommand{\norma}[2]{\lVert #1\rVert_{#2}}
\newcommand{\Norma}[2]{|\!|\!|{#1}|\!|\!|_{#2}}
\newcommand{\snorma}[2]{\lvert #1\rvert_{#2}}

\newcommand{\velocity}{v}
\newcommand{\zcurve}{\mathbf{z}}
\newcommand{\vectV}{\mathbf{v}}
\newcommand{\vectW}{\mathbf{v}^{(1)}}
\newcommand{\map}{\mathbf{x}}
\newcommand{\A}{a}
\newcommand{\B}{b}
\newcommand{\F}{f}
\newcommand{\q}{q}
\newcommand{\BSO}{\mathbf{B}}
\newcommand{\BRO}{\mathscr{B}}
\newcommand{\normal}[2]{\mathbf{n}_{#1}^{#2}}
\newcommand{\Turzone}{\Omega_{\mathrm{tur}}}
\newcommand{\TT}[1]{\mathbf{T}_{#1}}

\newcommand{\M}[2]{[\![#1]\!]_{#2}}
\newcommand{\J}[2]{[#1]_{#2}}
\newcommand{\mean}[1]{\langle #1\rangle}
\newcommand{\dev}[1]{\{#1\}}
\newcommand{\CA}[1]{\mathcal{C}(#1)}

\newcommand{\PV}{\mathrm{pv}}
\newcommand{\dif}{\,\mathrm{d}}
\newcommand{\tr}{\mathrm{tr}}
\newcommand{\Log}{\mathrm{Log}}

\newcommand{\Div}{\mathrm{div}}
\newcommand{\Curl}{\mathrm{curl}}
\newcommand{\sop}{\mathrm{supp}}
\newcommand{\car}[1]{\mathbbm{1}_{#1}}
\newcommand{\sgn}{\mathrm{sgn}}

\newcommand{\Ind}{\mathrm{I}}

\newcommand{\circulation}{\mathscr{C}}

\def\Xint#1{\mathchoice
	{\XXint\displaystyle\textstyle{#1}}%
	{\XXint\textstyle\scriptstyle{#1}}%
	{\XXint\scriptstyle\scriptscriptstyle{#1}}%
	{\XXint\scriptscriptstyle\scriptscriptstyle{#1}}%
	\!\int}
\def\XXint#1#2#3{{\setbox0=\hbox{$#1{#2#3}{\int}$}
		\vcenter{\hbox{$#2#3$}}\kern-.5\wd0}}

\def\dashint{\Xint-}  %mean integral

\theoremstyle{theorem}
\newtheorem{thm}{Theorem}[section]
\newtheorem{prop}{Proposition}[section]
\newtheorem{cor}{Corollary}[section]
\newtheorem{lemma}{Lemma}[section]
\theoremstyle{definition}
\newtheorem{defi}{Definition}[section]
\theoremstyle{remark}
\newtheorem{Rem}{Remark}[section]
\newtheorem{Ex}{Example}[section]

\numberwithin{equation}{section}

\title[Evolution of vortex sheet]{Dissipative Euler flows for vortex sheet initial data without distinguished sign}

\subjclass[2020]{35Q35, 76B47, 76E30, 76F10}

\keywords{Hydrodynamics, evolution of unstable interfaces, vortex sheet, convex integration}

\author{Francisco Mengual}
\address{Departamento de Matem\'aticas, Universidad Aut\'onoma de Madrid, E-28049 Madrid, Spain; Instituto de Ciencias Matem\'aticas (CSIC-UAM-UC3M-UCM), E-28049  Madrid, Spain.}
	\email{francisco.mengual@uam.es}
	\thanks{F.~M.~was partially supported by the Spanish Ministry of Economy through the ICMAT Severo Ochoa project SEV-2015-0554,  the grant MTM2017-85934-C3-2-P (Spain) and the ERC grant  307179-GFTIPFD, ERC grant 834728-QUAMAP}
	
\author{L\'{a}szl\'{o} Sz\'{e}kelyhidi, Jr.}
	\address{Institut f\"{u}r Mathematik, Universit\"{a}t Leipzig, Augustusplatz 10, D-04109, Leipzig, Germany.}
	\email{laszlo.szekelyhidi@math.uni-leipzig.de}
	\thanks{L.~Sz.~was partially supported by the European Research Council (ERC) under the European Union's Horizon 2020 research and innovation programme (grant agreement No.724298-DIFFINCL). Part of this work was completed in the Hausdorff Research Institute (HIM) in Bonn during the Trimester Programme Evolution of Interfaces. The authors gratefully acknowledge the warm hospitality of HIM during this time.}

\date{\today}

\begin{document}
	
\begin{abstract}We construct infinitely many admissible weak solutions to the 2D incompressible Euler equations for vortex sheet initial data. Our initial datum has vorticity concentrated on a simple closed curve in a suitable H\"older space and the vorticity may not have a distinguished sign. Our solutions are obtained by means of convex integration; they are smooth outside a ``turbulence'' zone which grows linearly in time around the vortex sheet. %and where the kinetic energy is dissipated. 
As a byproduct, this approach shows how the growth of the turbulence zone is controlled by the local energy inequality and measures the maximal initial dissipation rate in terms of the vortex sheet strength.
\end{abstract}
\maketitle

%\tableofcontents

\pagenumbering{arabic} 
\setcounter{page}{1}
\section{Introduction and main results}\label{sec:Introduccion}

The motion of a 2D ideal incompressible fluid is described by its velocity field $\velocity(t,x)$, which satisfies the \textbf{incompressible Euler equations} (IE) for some scalar pressure $p(t,x)$
\begin{align}
\partial_t\velocity+\Div(\velocity\otimes\velocity)+\nabla p & = 0, \label{Euler:1}\tag{IEa}\\
\Div\velocity & = 0,\label{Euler:2}\tag{IEb}
\intertext{in $\R_+\times\R^2$, evolving from an initial datum $\velocity_0(x)$ satisfying (IEb)}
\velocity |_{t=0} & = \velocity_0.\label{Euler:3}\tag{IEc}
\end{align}
In this work we are interested in the dynamics of vortex sheets, that is, the system (IE) when the initial vorticity $\omega_0:=\Curl\velocity_0$ is concentrated on a curve.  Here we consider
\begin{subequations}
\label{initial}
\begin{align}
\textrm{curves:}& &\zcurve_0&\in C^{k_0+1,\alpha}(\T;\R^2)
\quad\textrm{closed, simple and regular}\label{z0},\\
\textrm{vortex sheet strength:}& &\varpi_0&\in C^{k_0,\alpha}(\T;\R),\label{varpi0}
\end{align}
\end{subequations}
for some $k_0\geq 0$ and $\alpha>0$ to be determined. We may assume w.l.o.g.~that $\zcurve_0$ is the positively oriented ($\circlearrowleft$) arc-length ($|\partial_s\zcurve_0|=1$) parametrization, and so
$\T:=\R/\ell\Z$ with $\ell\equiv\mathrm{length}(\zcurve_0)$. Then, $\omega_0$ is the pushforward measure
\begin{equation}\label{omega0}
\omega_0=\zcurve_0^\sharp(\varpi_0\dif s).
\end{equation}
By identifying $\R^2\simeq\C$ ($*$ $\equiv$ complex conjugate) $\velocity_0$ is recovered from $\omega_0$ by the Biot-Savart law 
\begin{equation}\label{BSlaw}
\velocity_0(x)^*
%=\BSO(\omega_0)(x)^*
=\frac{1}{2\pi i}\int_{\T}\frac{\varpi_0(s)}{x-\zcurve_0(s)}\dif s,\quad x\notin\zcurve_0(\T).
\end{equation}
This $\velocity_0$ is bounded, anti-holomorphic outside $\Gamma_0:=\zcurve_0(\T)$ but with tangential discontinuities across $\Gamma_0$. Due to this lack of regularity we must interpret (IE) in its weak formulation.
\begin{defi}\label{Euler:weak} Let us denote
$$L^\infty_{\sigma}(\R^2):=\left\lbrace\velocity\in L^\infty(\R^2;\R^2)\,:\,\int_{\R^2}\velocity\cdot\nabla\psi\dif x=0
\quad\forall\psi\in C_c^1(\R^2)\right\rbrace.$$
A pair $(\velocity,p)\in L^\infty(0,T;L^\infty_{\sigma}(\R^2)\times L^\infty(\R^2))\cap C([0,T];L^\infty_{w*}(\R^2))$ is a \textbf{weak solution} to $\mathrm{(IE)}$ if 
$$\int_0^t\int_{\R^2}(\velocity\cdot\partial_t\Psi
+\velocity\otimes\velocity :\nabla\Psi+p\,\Div\Psi)\dif x\dif\tau=\int_{\R^2}\velocity(t)\cdot\Psi(t)\dif x-\int_{\R^2}\velocity_0\cdot\Psi_0\dif x$$
holds for every test function $\Psi\in C_c^1(\R^3;\R^2)$ and $0\leq t\leq T$, with $\Psi_0\equiv\Psi(0,\cdot)$. In addition, $(\velocity,p)$ is a \textbf{strong solution}  to $\mathrm{(IE)}$ if it is continuous and piecewise $C^1$ on $]0,T]\times\R^2$.
\end{defi}

\subsection{Brief background} \hfill

The Cauchy problem (IE) for the vortex sheet initial data \eqref{BSlaw} serves as a simplified model of many physical phenomena observed in the atmosphere
and oceans related to turbulence, such as mixing layers, jets and wakes (see \cite[\S 9]{MB02} and \cite{Fluidmotion}). By neglecting the effects of  surface tension and viscosity, this predicts the evolution of two incompressible and irrotational fluids (with the same constant densities, e.g.~two masses of water) when they come into contact with different motions at $\Gamma_0$ (\cite{Bir62}). The (tangential) discontinuity in the velocity induces vorticity at the interface $\Gamma_0$ \eqref{omega0}. Experimentally, this instability triggers a laminar-turbulent transition in a neighbourhood of the vortex sheet (\cite[\S 14]{GeoFD}). Mathematically, this Cauchy problem has been tackled from two different approaches.

One approach begins with the celebrated paper \cite{Delort}. Via a compensated-compactness type argument Delort proved that, provided $\omega_0$ belongs to the class $D^+(\R^2):=\mathcal{M}^+(\R^2)\cap H^{-1}(\R^2)$, (IE) admits a global weak solution $(\velocity,p)\in L_{\mathrm{loc}}^{\infty}(\R;L_{\sigma,\mathrm{loc}}^{2}(\R^2)\times\mathscr{S}'(\R^2))$ whose vorticity $\omega(t):=\Curl\velocity(t)$ remains in $D^+$ for all times, where $\mathcal{M}$ (resp. $\mathcal{M}^+$) denotes the space of (non-negative) Radon measures. The question of uniqueness as well as how the singularity spreads is not explicit in this construction.
Delort's result has been proved in different ways: by examining the concentration-cancellation effect \cite{EM94,Sch95} and via vanishing viscosity \cite{Maj93,Sch95} and vortex \cite{LX95,Sch96} methods. In all of them the fixed sign hypothesis seems crucial (cf.~\cite[\S 5]{EM94}).
The case of mixed sign vortex sheets has its own interest, both for its practical applications in aerodynamics and for the complex structures created by the intertwining between regions of positive and negative vorticity (\cite{Kra87}). Despite this, Delort's result has been only extended to $L^p+D^+$ \cite{Delort,Sch95,VW93} and to the case of vortex sheets with reflection symmetry by Lopes, Nussenzveig and Xin \cite{LNX01}.

The other approach aims to capture the structure of these solutions. Under the assumption that the vorticity (equivalently the discontinuity) remains concentrated on a moveable interface
\begin{equation}\label{freebdryassump}
\omega(t)=\zcurve(t)^\sharp(\varpi(t)\dif s),
\end{equation}
it is classical (cf.~\cite{SSBF81,MB02,LNS07,CCG12}) that the corresponding velocity field
\begin{equation}\label{velocity:1}
\velocity(t,x)^*
=\frac{1}{2\pi i}\int_{\T}\frac{\varpi(t,s)}{x-\zcurve(t,s)}\dif s,\quad x\notin\zcurve(t,\T),
\end{equation}
(with $p$ determined by the Bernoulli's law) is a weak solution to (IE) if and only if 
\begin{equation}\label{varpi:1}
\varpi=\varpi_0+\partial_s\tilde{\varpi},
\end{equation}
and $(\zcurve,\tilde{\varpi})$ solves the Birkhoff-Rott integrodifferential equations (BR)
\begin{subequations}
\label{BReq}
\begin{align}
\partial_t\zcurve&=\BRO+r\partial_s\zcurve,
&\partial_t\tilde{\varpi}&=r\varpi, \label{BReq:1}\\
\zcurve|_{t=0}&=\zcurve_0,&
\tilde{\varpi}|_{t=0}&=0, \label{BReq:2}
\end{align}
\end{subequations}
where $r(t,s)$ represents the re-parametrization freedom (cf.~\cite{CCG12}) and $\BRO\equiv\BRO(\zcurve,\varpi)$ is the Birkhoff-Rott operator ($\PV$ $\equiv$ Cauchy principal value)
$$
\BRO(t,s)^*:=
\frac{1}{2\pi i}\PV\!\!\int_{\T}\frac{\varpi(t,s')}{\zcurve(t,s)-\zcurve(t,s')}\dif s'.
$$

The linear stability analysis of \eqref{BReq} w.r.t.~the (steady) planar vortex sheet turns out to be an ill-posed Cauchy problem in terms of the Hilbert transform (cf.~\cite[\S 9.3]{MB02}). By desingularizing $\BRO$ it can be observed that the interface tends to roll-up into spiral vortices (cf.~\S\ref{sec:simulations}). This phenomenon is known as the \textbf{Kelvin-Helmholtz instability}.
The first result of local in time well-posedness of \eqref{BReq} in the class of analytic $(\zcurve_0,\varpi_0)$ was given in \cite{SSBF81} by Sulem, Sulem, Bardos and Frisch. Global in time results are only known for (not necessarily analytic) curves which are close enough to the planar vortex sheet, due to Duchon and Robert \cite{DR88}, by taking $\varpi_0$ somehow well-prepared in terms of $\zcurve_0$. Conversely, Caflish and Orellana exhibited in \cite{CO89} examples of breakdown of analyticity in finite time (see also \cite{Moo79}). In addition, they proved ill-posedness in $H^{s}$ for $s>3/2$. In \cite{Wu06} Wu proved instantaneous analyticity in a very weak class of solutions to \eqref{BReq} as well as some compatibility conditions required in $(\zcurve_0,\varpi_0)$.\\
 
\indent In this way, many of the data in \eqref{initial} are not within these approaches, both for the regularity and freedom required and the possible change of sign in $\varpi_{0}$.

\subsection{The main results} \hfill

Our aim is to present a third approach which is not only complementary to the previous two approaches but also remedies the shortcomings alluded to above. More precisely, our goal in the present paper is to develop a robust existence theory for (IE) for the vortex sheet initial data \eqref{BSlaw} in a large class of non-analytic initial data \eqref{initial}, which at the same time is able to keep track of the geometric evolution of the vortex sheet. In our approach, following \cite{Mixing,Piecewise}, the sharp interface  \eqref{freebdryassump} is replaced by an opening and evolving strip, called \textbf{turbulence zone}, $\Turzone$ around the initial interface. Outside the turbulence zone the solution is analytic, but inside $\Turzone$ the weak solution is obtained using convex integration and is highly irregular. The technique we use for obtaining weak solutions follows \cite{Eulerinclusion,Onadmissibility} and is based on the existence of a suitable \textbf{subsolution} (cf.~Definition~\ref{defi:subsolution}).
The key object of study in our work is then the subsolution, rather than the irregular weak solutions obtained by convex integration. As argued in \cite{HPfluid,HdPDE}, the subsolution relates to the macroscopic information usually studied in connection with hydrodynamical instabilities, such as growth of the turbulent zone, geometric evolution of the instability surface, macroscopic energy transfer, etc. 
The main advantage of this approach is that we circumvent the classical ill-posedness and are in this way able to identify the main contributions for the evolution of the macroscopic interface.

%\begin{thm}\label{thm:weak} Let \eqref{initial} with $k_0=3$ and $\alpha>0$. Then, there exist infinitely many weak solutions to $\mathrm{(IE)}$ for the vortex sheet initial datum \eqref{BSlaw}. 
%\end{thm}

In discussing weak solutions to (IE) it is important to specify some form of \textbf{admissibility} \cite{Onadmissibility,DR00,Lions}, in order to be able to rule out unphysical energy-creating solutions. A typical choice is monotonicity of the \textbf{global kinetic energy}
$$
E(t):=\frac12\int_{\R^2}|\velocity(t,x)|^2\dif x.
$$ 
This condition already suffices to guarantee weak-strong uniqueness \cite{BDS11,Lions}.
However, as in \cite{Delort}, for initial data as in \eqref{omega0} we expect the decay (cf.~Proposition~\ref{prop:BSlaw})
\begin{equation}\label{vdecay}
\velocity(t,x)^*
=\frac{1}{2\pi ix}\left(\int_{\T}\varpi_0\dif s
+\mathcal{O}(|x|^{-1})\right),
\quad |x|\gg 1,
\end{equation}
so that the total energy is not defined, unless $\varpi_{0}$ has zero mean. However, one may proceed as follows. First, we recall the dissipation measure associated with a weak solution.

\begin{defi}[Duchon, Robert \cite{DR00}]\label{defi:dissipation}
Given a weak solution $(\velocity,p)$ to (IE), the \textbf{dissipation} $D$ is defined as the distribution given by
$$\partial_te+\Div((e+p)\velocity)+D=0,$$
where $e:=\tfrac{1}{2}|\velocity|^2$ represents the kinetic energy density. Moreover, for any $0\leq t\leq T$, we denote by $D(t)$ the distribution defined as
\begin{equation}\label{D(t)}
\langle D(t),\psi\rangle
:=\int_{0}^{t}\int_{\R^2}(e\partial_t\psi+(e+p)\velocity\cdot\nabla\psi)\dif x\dif\tau
-\int_{\R^2}e\psi\dif x\Big|_{\tau=0}^{\tau=t},
\end{equation}
for every test function $\psi\in C_c^1(\R^3)$.
\end{defi}
Under some mild conditions (cf.~Lemma~\ref{De})  in the limiting case $\psi\uparrow\car{}$ \eqref{D(t)} reads as
$$
\int_{\R^2}(e(t)-e(0))\dif x=-\langle D(t),\car{}\rangle,
$$
which agrees with the expression
$$E(t)-E(0)=-\langle D(t),\car{}\rangle,$$
for finite-energy solutions. In this way the dissipation $D$ allows us to formulate admissibility. 
\begin{defi}\label{def:globalEI} A weak solution to (IE) whose dissipation $D$ is compactly supported is  \textbf{admissible} (or globally dissipative) if for all $0\leq t\leq T$
$$\langle D(t),\car{}\rangle\geq 0.$$
\end{defi}
For completeness, in Lemma~\ref{lemma:WSU} we will extend the classical weak-strong uniqueness statement to weak solutions as in Definition \ref{def:globalEI} with possibly infinite energy.

With these preparations, our main result is as follows:
\begin{thm}\label{thm:global} Consider initial data \eqref{initial} with $k_0=4$, $\alpha>0$ and $\varpi_0$ not identically vanishing. Then, there exist infinitely many admissible weak solutions to $\mathrm{(IE)}$ for the vortex sheet initial datum \eqref{BSlaw}.
\end{thm}

By \cite[\S B]{Onadmissibility}, these are ``dissipative solutions'' in the sense of Lions \cite{Lions}.

\smallskip

\subsection{Evolution of the vortex sheet}\hfill

By multiplying formally the momentum balance equation \eqref{Euler:1} by $\velocity$, it is straightforward to check that $D=0$ wherever $\velocity$ is smooth enough. Indeed, the critical smoothness in 3D is the subject of the recently resolved Onsager's conjecture. We refer to \cite{BDSV19,CET94,Eyi94,Ise18} as well as the surveys \cite{HPfluid,HdPDE}. In particular, in \cite{BDSV19,Onadmissibility,DLSz12} infinitely many admissible weak solutions with non-vanishing $D$ were constructed. Moreover, in \cite{DanSz,DanRSz,WSz} it was shown that the set of initial data admitting infinitely many admissible weak solutions (called `wild initial data') is $L^2$-dense.

Concrete examples of wild initial data were first found in \cite{Sze11}, where the second author constructed weak solution to (IE) with decreasing $E$ for the planar interface $\zcurve_0(s)=(s,0)$ with vortex sheet strength $\varpi_0= 2$, and observed that the maximal dissipation rate $\frac{\dif E}{\dif t}=-\frac{1}{6}$ determines uniquely the rate of expansion $c=\tfrac{1}{2}$ of the turbulence zone (cf.~Example~\ref{Ex:flat} below). This example then served as the basis for a large number of explicit non-uniqueness examples, see e.g.~\cite{Mixing,Degraded,ChDLK,Piecewise,gksz:2020,MK18,M20,Sze12}. A common property of these explicit wild initial data is the fact that they are associated with an unstable interface of discontinuity. 

One particularly well-studied example in recent years is concerning the unstable Muskat problem for the 2D incompressible porous media equations (IPM). This system models the evolution of two incompressible fluids with different constant densities, moving under the action of gravity through a 2D porous medium, once the heavier fluid meets the lighter one at some $\Gamma_0=\{(s,f_0(s))\,:\,s\in\R\}$ from above. As a result, the vorticity is concentrated on $\Gamma_0$ as in \eqref{omega0} with $\varpi_0=\partial_sf_0$. For the case of a flat initial interface $f_0=0$ the second author in \cite{Sze12} constructed weak solutions to (IPM), which recover the mixing behaviour expected at the interface: inside a linearly growing mixing zone $\Omega_{\mathrm{mix}}$ around the flat interface $f_0=0$ the density is a fine mixture of the two constant densities. As a byproduct, a promising connection between the corresponding subsolution and the Lagrangian relaxed solution obtained by the gradient flow approach of Otto \cite{Ott99} was established (cf.~\cite{M20}).
Subsequently, Castro, C\'ordoba and Faraco extended \cite{Sze12} to the case $f_0\in H^5(\R)$ \cite{Mixing} (see also \cite{Degraded}) by introducing a two-scale dynamics: in order to describe $\Omega_{\mathrm{mix}}$ one needs, in addition to the growth of the mixing zone, a second dynamics, relating to the evolution of the pseudo-interface $f(t,s)$. In \cite{Piecewise}, F\"{o}rster and the second author simplified the approach of \cite{Mixing}, obtaining the result for $f_0\in C_*^{3,\alpha}(\R)$.

Our construction for (IE) follows the approach in \cite{Mixing,Piecewise}. Let us describe the geometric setup.
At each time $t>0$, the turbulence zone $\Turzone(t)$ is defined as the annular region in $\R^2$ whose boundary is
\begin{equation}\label{def:turzone}
\partial\Turzone(t)=\Gamma_-(t)\cup\Gamma_+(t),
\end{equation}
where $\Gamma_{\pm}(t):=\map_{\pm}(t,\T)$ are the new evolving vortex sheets parametrized by the map
\begin{equation}\label{eq:map}
\map_{\lambda}(t,s)
:=\zcurve(t,s)+\lambda\vectV(t,s),
\end{equation}
where $\zcurve$ is an evolution of $\zcurve_0$ to be determined and $\vectV$ is an opening non-tangential vector field to this curve. As our result is local in time, it seems suitable to consider
\begin{equation}\label{ansatz:vectV}
\vectV(t,s):=t\vectW(s)
\quad\textrm{with}\quad
\vectW(s):=c(s)\partial_s\zcurve_0(s)^\perp,
\end{equation}
where $c\in C^{k_0,\alpha}(\T;\R_+)$ represents the (local) rate of expansion of the turbulence zone, to be determined.

We define also the open sets $\Omega_{\pm}(t)$ such that $\partial\Omega_{\pm}(t)=\Gamma_{\pm}(t)$; observe that, since $z_0$ and $z(t)$ are assumed to be positively oriented (cf.~\eqref{initial}), $\Omega_+(t)$ is the bounded open set \emph{inside} the Jordan curve $\Gamma_+(t)$, whereas $\Omega_-(t)$ is the unbounded open set \emph{outside} the Jordan curve $\Gamma_-(t)$. For each region $r\in\{-,+,\mathrm{tur}\}$, we define the space-time open sets $\Omega_{r}:=\{(t,x)\,:\,x\in\Omega_{r}(t),\,t\in[0,T]\}$. Finally, we set $\Gamma(t):=\Gamma_-(t)\cup\Gamma_+(t)$, $\Omega(t):=\Omega_-(t)\cup\Omega_+(t)$ and respectively $\Gamma$ and $\Omega$. 

\begin{figure}[h!]
	\centering
	\subfigure[Vortex sheet initial datum.]{\includegraphics[width=0.49\textwidth]{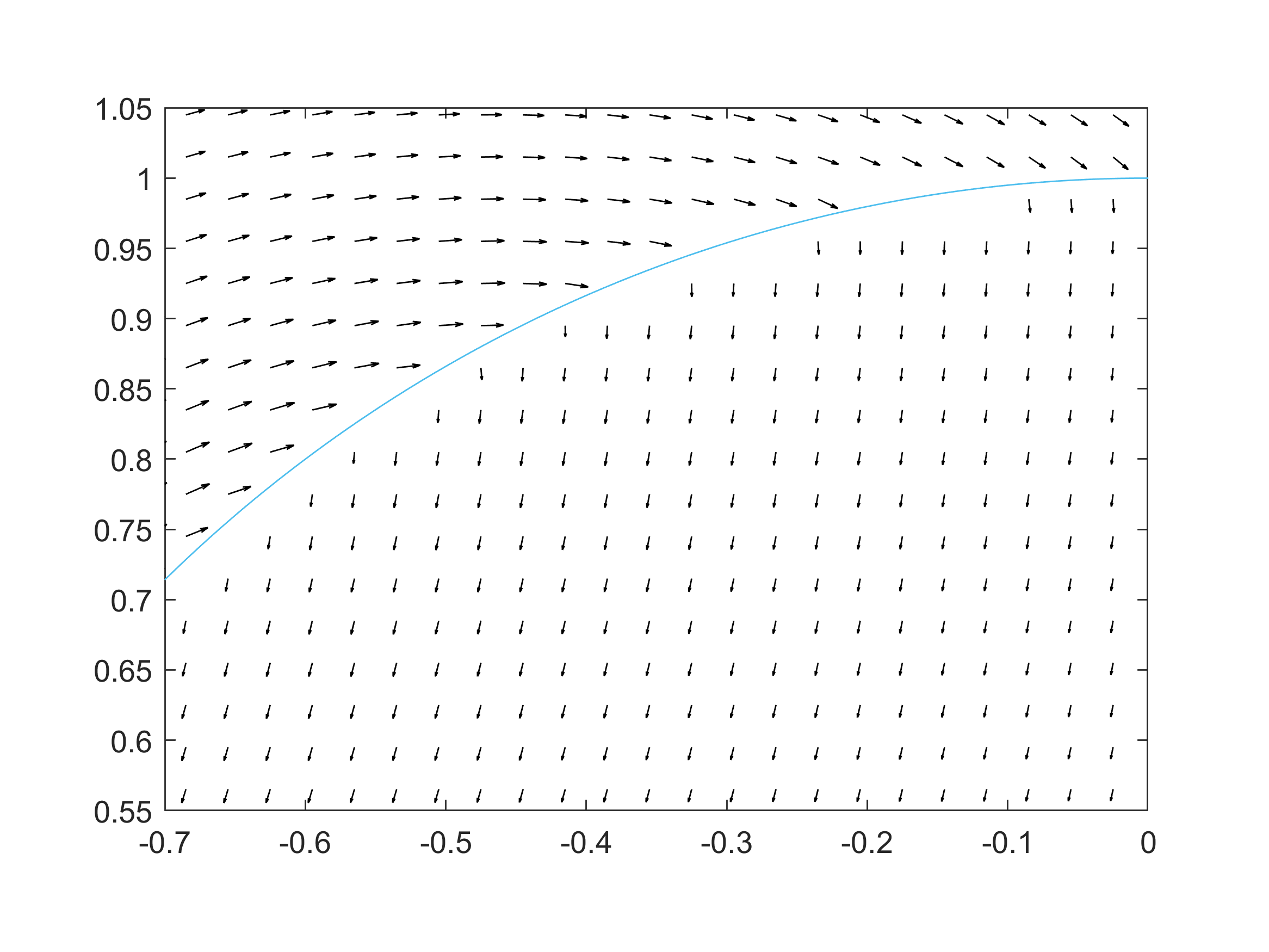}}\label{fig:initial}  
	\subfigure[The turbulence zone.]{\includegraphics[width=0.49\textwidth]{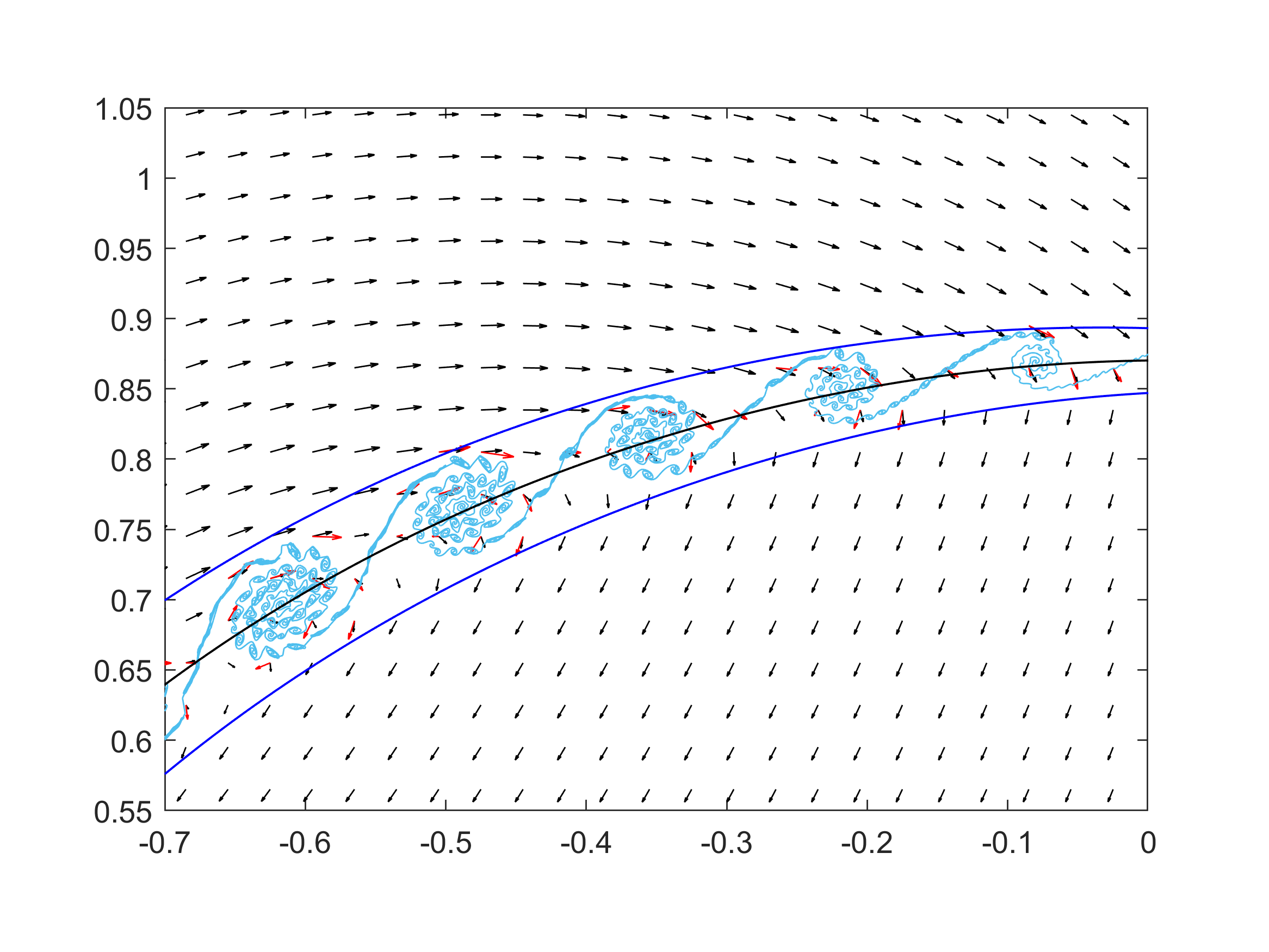}}   
	\caption{(a) The incompressible velocity field $\velocity_0$ with vortex sheet strength $\varpi_0$ along the interface $\zcurve_0$. (b) At some $t>0$, the macroscopic velocity field $\bar{\velocity}(t)$, the boundary of the turbulence zone $\map_{\pm}(t)=\zcurve(t)\pm ct\partial_s\zcurve_0^\perp$ (dark blue) for some $\zcurve(t)$ (black) and $c(s)\propto|\varpi_0(s)|$. Inside $\Turzone(t)$ we plot the velocity field
	$\velocity(t)$ (red) w.r.t.~the Kelvin-Helmholtz curve $\zcurve_{\mathrm{per}}(t)$ (light blue) which starts from a tiny perturbation of $\zcurve_0$ (cf.~\S\ref{sec:simulations}, $\delta=0.001$).}              
\end{figure}

\subsection{Energy dissipation rate}\hfill

A new feature of our approach is that we are able to link the energy dissipation rate at the vortex sheet with the growth of the turbulence zone. This is relevant in view of the search for selection criteria among infinitely many admissible weak solutions - one that has been intensively studied in recent years is motivated by the \emph{entropy rate admissibility criterion} introduced by C.~Dafermos \cite{D73}, and amounts, in a nutshell, to identifying those weak solutions which maximize the energy dissipation rate in some sense, see \cite{Fei14,ChK}. To explain this point, let us recall the construction of \cite{Sze11}. 

\begin{Ex}\label{Ex:flat}
Let $(\bar{v},\bar{p},\bar{R})$ be defined as
$$\bar{\velocity}=(\alpha,0),
\quad\quad\bar{p}=\tfrac{1}{2}\alpha^2-e,
\quad\quad\bar{R}=\left(\begin{array}{cc}
e-\tfrac{1}{2}\alpha^2&\gamma\\ \gamma&e-\tfrac{1}{2}\alpha^2
\end{array}\right),$$
for some (scalar) functions $\alpha,\gamma,e$ of $(t,x_2)$, with 
\begin{equation}\label{alpha0}
\alpha(0,x_2)=\alpha_0(x_2)=\sgn\, x_2.
\end{equation}
Observe that the initial datum $\velocity_0(x)=\bar{\velocity}(0,x)=(\alpha_0(x_2),0)$ is a shear flow whose vorticity is concentrated on the $x_1$-axis $\Gamma_0$ with density $\varpi_0=2$. It is straightforward to check that $\Div\bar{v}=0$ and 
$\partial_t\bar{v}+\Div(\bar{v}\otimes\bar{v}+\bar{R})+\nabla\bar{p}=0$ holds with $\bar{R}$ positive semidefinite if and only if
\begin{equation}\label{eq:flatcase}
\partial_t\alpha+\partial_{x_2}\gamma=0,\quad e\geq \tfrac{1}{2}\alpha^2+|\gamma|.
\end{equation}
As noted in \cite{Sze11} (based on \cite{Onadmissibility}; see also Theorem \ref{thm:hprinciple} below), under this condition $(\bar{v},\bar{p},\bar{R})$ is a subsolution. Moreover, if there exists a space-time open set $\Turzone$ such that $e>\tfrac{1}{2}\alpha^2+|\gamma|$ inside $\Turzone$ and $e=\tfrac12\alpha^2$ outside $\Turzone$, then for any $T>0$ there exist infinitely many weak solutions $(v,p)$ to (IE) on $[0,T]$ with $\velocity(0,x)=\velocity_0(x)$ such that, in $C_tL_{\textrm{loc}}^1$, 
$$
\frac{1}{2}|v|^2=e,\quad p=\bar{p}.
$$

Let us construct some explicit examples. First of all, 
fix $c>0$, let $\Turzone=\{|x_2|<ct\}$ (so that $\map_{\lambda}(s,t)=(s,\lambda ct)$, c.f.~\eqref{def:turzone}-\eqref{eq:map}), and set 
$$
\alpha=\pm 1,\quad\gamma=0,\quad e=\tfrac12\quad\textrm{ on }\Omega_{\pm}.
$$
It is clear that many choices exist inside $\Turzone$ which satisfy \eqref{eq:flatcase}. One simple choice is given by
$$
\alpha=0,\quad \gamma=-c,\quad e=c+e'\quad\textrm{ in }\Turzone,
$$
for any $e'>0$. Observe that $(\alpha,\gamma)$ is piecewise constant and the choice $\gamma=-c$ in $\Turzone$ ensures that the jump condition across $\Gamma$ arising from \eqref{eq:flatcase} holds. With these choices a quick calculation shows that in the limiting case $e'\downarrow0$ the dissipation (c.f.~Definition \ref{defi:dissipation}) of the weak solutions $(v,p)$ so obtained satisfies 
\begin{equation}
\langle D(t),\psi\rangle=c(\tfrac12-c)\sum_{\lambda=\pm 1}\int_0^t\int \psi(\tau,\map_\lambda(\tau,s))\dif s\dif\tau,
\end{equation}
so that in particular
$$
\frac{d}{dt}\bigg\vert_{t=0}\langle D(t),\psi\rangle =2c(\tfrac12-c)\int\psi(0,s,0)\dif s.
$$

Observe that in this example the vorticity of the corresponding $\bar{\velocity}(t)$ is concentrated on the two lines $\{x_2=\pm ct\}$. Analogusly, if we define $\alpha,\gamma$ inside $\Turzone$ in such a way that the vorticity of $\bar{\velocity}$ is concentrated on the $2N$ lines $\{x_2=\lambda_jct\}$ with $\lambda_{\pm j}=\pm\tfrac{2|j|-1}{2N-1}$ for $j=1,\ldots,N$, we obtain the piecewise constant solution
$$
\alpha(t,x_2)
=\left\lbrace\begin{array}{rl}
\pm\tfrac{j}{N}, & \lambda_jct<\pm x_2<\lambda_{j+1}ct, \quad j=1,\ldots,N-1,\\[0.1cm]
0, & |x_2|<\lambda_1ct,
\end{array}\right.
$$ 
with piecewise constant $\gamma$ determined uniquely by the jump conditions arising from \eqref{eq:flatcase}, namely $\gamma=-\tfrac{N}{2N-1}c(1-\alpha^2)$, and $e=\tfrac12\alpha^2+|\gamma|+e'$ for some $e'>0$. We note in passing that such construction was used in the context of the compressible Euler system under the name ``fan subsolution'', see \cite{ChDLK}. In this case it can be checked that, with $[\cdot]_{\lambda}$ denoting the jump across $\{x_2=\lambda ct\}$, as $e'\downarrow 0$ we have
$$
\langle D(t),\psi\rangle
=\tfrac{N}{2N-1}c\left(\tfrac{2N-1}{2N}-c\right)\sum_{1\leq|j|\leq N}\int_0^t\int\lambda_j[\alpha^2]_{\lambda_j}\psi(\tau,\map_{\lambda_j}(\tau,s))\dif s\dif\tau.
$$
Indeed, by the definition of $\lambda_j$ and $\alpha$,
$$\lambda_j[\alpha^2]_{\lambda_j}
%=\frac{2j-1}{2N-1}\left(\left(\frac{j}{N}\right)^2-\left(\frac{j-1}{N}\right)^2\right)
=\tfrac{(2j-1)^2}{(2N-1)N^2}>0,
\quad 1\leq|j|\leq N.$$
Hence, by applying Faulhaber's formula we get
$$\sum_{1\leq |j|\leq N}\lambda_j[\alpha^2]_{\lambda_j}
=\tfrac{2}{3}\tfrac{2N+1}{N},$$
which allows to compute the associated energy dissipation rate at the initial time
\begin{equation}\label{DN:flat}
\frac{\dif}{\dif t}\bigg\vert_{t=0}\langle D,\psi\rangle
=2\bar{a}_Nc(2\bar{c}_N-c)\int\psi(0,s,0)\dif s,
\end{equation}
where 
\begin{equation}\label{acN}
\bar{a}_N:=\frac{1}{3}\frac{2N+1}{2N-1}, \quad\quad\bar{c}_N:=\frac{2N-1}{4N}.
\end{equation}
Finally, the limiting case $N\uparrow\infty$ can be understood as distributing the vorticity on the whole turbulence zone. This corresponds to the rarefaction wave solution (c.f.~\cite{Sze11})
$$
\alpha(t,x_2)=\frac{x_2}{ct},
\quad |x_2|<ct,
$$
with $\gamma=-\tfrac{1}{2}c(1-\alpha^2)$. Observe that $\alpha,\gamma$ are continuous now.
In this case, as $e'\downarrow 0$ we have
$$\partial_te
=\tfrac{1}{2}(1-c)\partial_t(\alpha^2)
,$$
and hence the corresponding dissipation satisfies
$$
\langle D(t),\psi\rangle
=\tfrac{1}{2}(1-c)\int_0^t\int\int_{\{|x_2|<c\tau\}}
[\partial_{t}(\alpha^2)
-\tfrac{1}{2}\velocity\cdot\nabla(\alpha^2)]\psi\dif x_2\dif x_1\dif\tau,
$$
which allows to compute the associated energy dissipation rate at the initial time
$$
\frac{\dif}{\dif t}\bigg\vert_{t=0}\langle D,\psi\rangle
=\tfrac{2}{3}c(1-c)\int\psi(0,s,0)\dif s,
$$
which is indeed the limit of \eqref{DN:flat} as $N\uparrow\infty$. 
In particular, the maximum $\tfrac{\dif E}{\dif t}=-\tfrac{1}{6}$ is achieved at $c=\tfrac{1}{2}$.
\end{Ex}

In this paper we are considering that both $\varpi_0$ and $c$ could depend on $s$. Inspired by \eqref{DN:flat} and recalling that in that example $\varpi_0(s)= 2$, we introduce the following functional.
\begin{defi}
For given $\varpi_0,c\in C(\T;\R)$, any $N\in\N$ and any interval $I\subset\T$ we define the energy dissipation functional 
\begin{equation}\label{e:DF}
W^{(N)}_{I}(c):=\bar{a}_N\int_{I}c(s)|\varpi_0(s)|(\bar{c}_N|\varpi_0(s)|-c(s))\dif s,
\end{equation}
with $\bar{a}_N$, $\bar{c}_N$ given in \eqref{acN}.
\end{defi}

It turns out that global admissibility, as in Definition \ref{def:globalEI} leads to the following relation:

\begin{thm}\label{thm:globalrate}
For any $N\in\N$ there exist admissible weak solutions as
in Theorem \ref{thm:global}, such that the rate of dissipation and expansion of $\Turzone$ are related via
\begin{equation}\label{global:1}
\int_{\R^2}\frac{e(t_2)-e(t_1)}{t_2-t_1}\dif x
=-W_{\T}^{(N)}(c)+\mathcal{O}(t_2),
\end{equation}
for every $0\leq t_1<t_2\leq T$.
\end{thm}

In particular, for zero-mean $\varpi_0$'s, \eqref{global:1} yields
$$
\frac{\dif E}{\dif t}=-W_{\T}^{(N)}(c)+\mathcal{O}(t).
$$
Note that the functional $W_{\T}^{(N)}$ is strictly concave and has a unique global maximum (cf.~Lemma \ref{lemma:Wmax}) at $c_{\max}^{(N)}:=\frac{1}{2}\bar{c}_N|\varpi_{0}|$, with  
$$
\max W_{\T}^{(N)}=W_{\T}^{(N)}(c_{\max}^{(N)})=\frac{1}{48}\left(1-\frac{1}{(2N)^2}\right)\int_{\T}|\varpi_{0}|^3\dif s.
$$
Thus, by considering $N\uparrow\infty$ we reach the maximal initial dissipation rate
$\frac{\dif E}{\dif t}\big|_{t=0}=-\frac{1}{48}\norma{\varpi_0}{3}^3$ obtainable 
by our method, which agrees with \cite{Sze11}. 

\smallskip

The relationship \eqref{global:1} shows that the growth rate of the turbulence zone $c$ cannot be arbitrarily large, but does not give precise information about its local growth. For this reason, we test $D$ with a larger class of $\psi$'s, instead of just $\psi=\car{}$. 
As we shall see in Section \ref{sec:Admissibility} (cf.~Propositions~\ref{prop:dissipation}-\ref{prop:D0}), the dissipation $D$ from Theorem~\ref{thm:globalrate} belongs to $\mathcal{M}_c([0,T]\times\R^2)$ with $\sop\, D\subset\bar{\Omega}_{\mathrm{tur}}$, so we can extend the space of test functions $\psi$ to indicator functions. With this notion we obtain a local version of Theorem~\ref{thm:globalrate}.

\begin{thm}\label{thm:local} Given $0<\varepsilon\leq\ell$, let \eqref{initial} with  $k_0=4$, $\alpha>0$ and $|\varpi_0|>0$ a.e. Then, for any $N\in\N$ there exist infinitely many dissipative solutions to $\mathrm{(IE)}$ for the vortex sheet initial datum \eqref{BSlaw} so that the local rate of dissipation and expansion of $\Turzone$ are related via
\begin{equation}\label{local:1}
\left\langle\frac{D(t_2)-D(t_1)}{t_2-t_1},\psi_I\right\rangle
=W_{I}^{(N)}(c)+\mathcal{O}(t_2),
\end{equation}
where $\psi_I(t,x):=\car{\map(t;I\times[-1,1])}(x)$, for every interval $I\subset\T$ with $|I|\geq\varepsilon$ and $0\leq t_1<t_2\leq T$.
\end{thm}

Since $\psi_{\T}=\car{\bar{\Omega}_{\mathrm{tur}}}$, Theorem~\ref{thm:local} generalizes Theorem~\ref{thm:globalrate} when $\varepsilon=\ell$. For small times, \eqref{local:1} can be viewed as testing $D$ with (space-time) cylinders  $[0,T]\times B_r(z_0)$ for every $z_0\in\Gamma_0$ and $r\gtrsim\varepsilon$, thus preventing local creation of kinetic energy (cf.~\cite{DR00}) across $\Gamma_0$ on length scales $\gtrsim\varepsilon$.

\smallskip

One possible choice of $c$ in Theorem~\ref{thm:local} is
\begin{equation}\label{c}
c=\delta\bar{c}_N|\varpi_0|\ast\eta_\epsilon,
\end{equation}
for any $0<\delta<1$ and some $\epsilon(\varpi_0,\varepsilon,\delta)\geq 0$,
where $(\eta_\epsilon)$ is a standard mollifier (cf.~Lemma \ref{lemma:T3}). In particular, the dissipation rate is maximazed at $\delta=\tfrac{1}{2}$ as $\epsilon\downarrow 0$.

\smallskip

It would be interesting to explore the question whether the bounds obtained in Theorems \ref{thm:globalrate}-\ref{thm:local} for the energy dissipation rate are optimal. Also, a natural and very interesting question is whether one can show convergence of the vortex blob approximation \cite{caflischlowengrub} in a suitable weak sense to an admissible subsolution as in Example \ref{Ex:flat}, see also Definition \ref{defi:subsolution}.

%Thus, if $|\varpi_0|\gg 0$ we can take $c=c_{\max}\in C^{k_0,\alpha}$ ($\delta=\tfrac{1}{2}$, $\epsilon=0$). Otherwise, we can consider $c=c_{\max}*\eta_{\epsilon}$ ($\delta=\tfrac{1}{2}$, $\epsilon>0$) for which $W_{\T}(c)=W_{\T}(c_{\max})+\mathcal{O}(\epsilon)$.

%Seen in another way, the rate of expansion $c_{\max}$ leads (initially) the most stable evolution in the sense that both $c\downarrow 0$ and $c\uparrow\bar{c}|\varpi_0|$ exhaust the times for which we can guarantee \eqref{global:1}\eqref{local:1}.\\[0.1cm]

\smallskip

These results admit some immediate improvements and generalizations.
\begin{itemize}
	\item The weak solutions from Theorems \ref{thm:global}-\ref{thm:local} belong to the stronger class $C_tL_{\mathrm{loc}}^q$ for all $1< q<\infty$. In particular, for zero-mean $\varpi_0$'s, also $\velocity\in C_tL^q$.
	\item We have chosen the case of closed curves ($\zcurve_0(s+\ell)=\zcurve_0(s)$) for simplicity.
	The periodic ($\zcurve_0(s+n)=\zcurve_0(s)+n$, $n\in\ell\Z$) and the asymptotically flat ($\zcurve_{0}(s)\sim s$ as $|s|\rightarrow\infty$) cases require the analysis of a Birkhoff-Rott type operator similar to the one \eqref{BRO:0} we study in Section \ref{sec:BRO}. For the first one, the Biot-Savart kernel is $\tfrac{1}{2\pi i}\cot x$ instead of $\tfrac{1}{2\pi i}x^{-1}$. For the second one, the domain is $\R$ instead of $\T$, thus increasing the computations to control the tails. However, we believe the same results should be proved similarly. Conversely, in our case we have to deal with the fact that $\Turzone(t)$ is not simply connected (cf.~Section \ref{sec:Helmholtz} and Lemma~\ref{R:solution}).
	\item The regularity required in Theorems~\ref{thm:global}-\ref{thm:local} is used to control, with relatively simple estimates, $\norma{\zcurve}{2,\alpha}$ and $\norma{\varpi}{1,\alpha}$. Thus, a finer analysis of Section \ref{sec:BRO} may reduce $k_0$.
\end{itemize}

At the same time there are several shortcomings. 
\begin{itemize}
	\item Our results are local in time. The time of existence $T$ in Theorems \ref{thm:global}-\ref{thm:globalrate} depends on  $\norma{\zcurve_0}{k_0+1,\alpha}$, $\norma{\varpi_0}{k_0,\alpha}$ and the chord-arc constant $\CA{\zcurve_{0}}$ (see \eqref{CAC}). The time $T_{\varepsilon}\leq T$ for which we can guarantee \eqref{local:1} satisfies (see Lemma~\ref{lemma:T3}) $$T_{\varepsilon}\gtrsim\min_{|I|\geq\varepsilon}\dashint_I|\varpi_0|^3\dif s.$$ 
	In addition, $T_{\varepsilon}$ depends on $1/c\sim 1/(|\varpi_0|*\eta_{\epsilon(\varepsilon)})$.
	Thus, $T_{\varepsilon}\gg 0$ independently of $\varepsilon$ provided $|\varpi_0|\gg0$. Otherwise, more terms should be controlled in Section \ref{sec:Admissibility} to avoid $T_{\varepsilon}\downarrow 0$ as $\varepsilon\downarrow 0$.
	\item Although $\omega_0\in D^+$ when $\varpi_0\geq 0$, in general for $t>0$ we only know that $\omega(t)$ is a distribution, although we can guarantee at least the vorticity balance
	$$\langle\omega(t),\car{}\rangle
	=\int\omega_0.$$
\end{itemize}

\subsection{Organization of the paper}\hfill 

We start Section \ref{sec:subsolution} recalling the concept of subsolution as well as the subsolution criterion for (IE). After that, we establish the conditions under which a subsolution exists. In addition, we determine in Section \ref{sec:Admissibility} the dissipation $D$ of the weak solutions obtained via convex integration applied to the subsolution. In Section \ref{sec:BRO} we analyse the corresponding Birkhoff-Rott type operator. Finally, we prove in Section \ref{sec:proof} the Theorems~\ref{thm:global}-\ref{thm:local}. The parameter $N$ in Theorems \ref{thm:globalrate}-\ref{thm:local} relates to the following ansatz: the initial vortex sheet \eqref{omega0} is split at time $t=0$ into a sum of $2N$ vortex sheets, which separate at linear speeds for $t>0$. The outmost curves form the boundaries of the turbulence zone $\Turzone$. In Sections \ref{sec:subsolution}-\ref{sec:proof} we will concentrate on the case $N=1$ for simplicity, and will show how to extend these considerations to general $N$ in Section \ref{sec:piecewise}. Finally, we provide some pictures of how these solutions may look like in the Appendix \ref{sec:simulations}.

\smallskip

\subsection{Notation}\label{sec:Notation}

\begin{enumerate}[(i)]
	%\item $\mathcal{M}_c:=$ (signed) compactly supported Radon measures.
	\item\textit{Lebesgue spaces}. For $1\leq p\leq\infty$ let $L^p$ be the space of $p$-integrable (or essentially bounded for $p=\infty$) functions with norm $\norma{\cdot}{p}$.
	\item\textit{H\"{o}lder spaces}. For $k\geq 0$ and 	$\alpha>0$ let $C^{k,\alpha}$ be the space of $k$-times differentiable functions with norm
	$$\norma{f}{k,\alpha}:=\sup_{j\leq k}\norma{\partial^jf}{\infty}
	+\snorma{\partial^kf}{\alpha}
	\quad\textrm{where}\quad
	\snorma{f}{\alpha}:=\sup_{s,\xi}\frac{|f(s+\xi)-f(s)|}{|\xi|^\alpha}.$$
	\item\label{Notation:matrix}\textit{Matrix sets}. Let
	$\mathcal{S}:=\{R\in\R^{2\times 2}\,:\,R^t=R\}$ and denote $R_{\mathrm{sym}}:=\mathrm{proj}_{\mathcal{S}}R=\tfrac{1}{2}(R+R^t)$.  We consider its subsets
	$\mathcal{S}_\pm:=\{R\in\mathcal{S}\,:\,\pm R\geq 0\}$ and $\mathcal{S}_0:=\{R\in\mathcal{S}\,:\,\tr R=0\}$, and denote $R^{\pm}:=\mathrm{proj}_{\mathcal{S}^{\pm}}R$ and $\mathring{R}:=\mathrm{proj}_{\mathcal{S}_0}R=R-\tfrac{1}{2}(\tr R)\mathrm{Id}$ with $|\mathring{R}|:=\lambda_{\max}(\mathring{R})=\max_{|\xi|=1}\mathring{R}\xi\cdot\xi$ by the operator norm of $\mathring{R}$ (the maximum eigenvalue).
	
	\item\label{Notation:complex}\textit{Complex coordinates}. We identify $\R^2$ with the complex plane $\C$ as usual $z=(z_1,z_2)=z_1+iz_2$ with $i=(0,1)$ the imaginary unit. Thus, without any ambiguity we shall denote $z^*=(z_1,-z_2)=z_1-iz_2$, $z\cdot w=\Re(zw^*)=z_1w_1+z_2w_2$ and $iz=z^\perp=(-z_2,z_1)$ interchangeably. Given $0\neq w\in\C$, we denote $\Re_{w},\Im_{w}:\C\rightarrow\R$ by the projections
	$$
	\Re_{w}z=\frac{z\cdot w}{|w|^2}=\Re\left(\frac{z}{w}\right),\quad\quad
	\Im_{w}z=\frac{z\cdot w^\perp}{|w|^2}=\Im\left(\frac{z}{w}\right),
	$$
	with which any $z\in\C$ can be expressed as
	$$z=(\Re_{w}z+i\Im_{w}z)w.$$
	We consider the isomorphism $\mathrm{M}:\C\rightarrow\mathcal{S}_0$ given by
	$$\mathrm{M}(z)
	=\left(\begin{array}{rr}
	-z_1 & z_2 \\
	z_2 & z_1
	\end{array}\right),$$
	which satisfies the identity
	$$
	\mathrm{M}(z)w
	=-(zw)^*.
	$$
	Note that $|\mathrm{M}(z)|=|z|$ (where $|\cdot|$ denotes the operator norm of the matrix on the l.h.s.~and the norm of the complex number on the r.h.s.).
	Given $\Omega\subset\C$ open, any $f\in C^1(\Omega;\R^2)$ satisfies ($\mathrm{Hol}$ $\equiv$ holomorphic functions)
	$$
	f\in\mathrm{Hol}(\Omega)\quad\Leftrightarrow\quad\Div(\mathrm{M}(f))=0\quad\textrm{on }\Omega.
	$$

	\item\label{Notation:traces}\textit{Traces}. We will repeatedly use the notation $\Omega_\lambda(t)$ and $\Gamma_\lambda(t)$, with $\lambda\in\{-1,1\}$, to denote the sets $\Omega_{\pm}(t)$ and $\Gamma_{\pm}(t)$. Thus, given $f$ continuous on $\bar{\Omega}_{\lambda}$ or $\bar{\Omega}_{\mathrm{tur}}$, we denote its traces on $\Gamma_{\lambda}$, $\lambda\in\{-1,1\}$ as
	$$f_{\lambda}^{\pm}(t,s)
	:=\lim_{\varepsilon\downarrow 0}f(t,\map_{\lambda}(t,s)\pm\varepsilon\partial_s\map_{\lambda}(t,s)^\perp),$$
	and the corresponding mean value and jump across $\Gamma_{\lambda}$ as
	$$
	\M{f}{\lambda}:=\tfrac{1}{2}(f_{\lambda}^++f_{\lambda}^-),\quad\quad
	\J{f}{\lambda}:=f_{\lambda}^+-f_{\lambda}^-. 
	$$
	In particular, if $\J{f}{\lambda}=0$ we shall write $f_{\lambda}(t,s)=f(t,\map_{\lambda}(t,s))$. In this case,
	complex path integrals along $\Gamma_\lambda$ can be written as
	$$
	\int_{\Gamma_\lambda} f(x)\dif x=\int_{\T} f_\lambda\partial_s\map_\lambda\dif s=\int f_\lambda\partial_s\map_\lambda,
  $$
  where we abbreviate $\int=\int_{\T}\dif s$.\\
Furthermore, given $b_{\pm}\in C(\Gamma_{\pm})$ we denote
	$$
	\mean{b}
	:=\tfrac{1}{2}(b_++b_-),\quad\quad
	\dev{b}:=\tfrac{1}{2}(b_+-b_-). 
	$$
\item\label{Notation:Taylor}\textit{Taylor series}. Let $I\subset\R$ be an interval. Given $f:I\rightarrow\C$ we denote
$$\triangle_hf(r):=\frac{f(r+h)-f(r)}{h},$$
for $r,r+h\in I$ with $h\neq 0$.
More generally, if $f$ is $n$-times differentiable at  $r$ we denote its Taylor polynomial of order $n$ and the corresponding reminder as
\begin{align*}
T_r^nf(h)&:=\sum_{j=0}^n\frac{\partial_r^jf(r)}{j!}h^j,\\
\triangle_h^{n+1}f(r)
&:=\frac{f(r+h)-T_r^nf(h)}{h^{n+1}}
=\frac{\triangle_h^nf(r)-\tfrac{1}{n!}\partial_s^nf(r)}{h}.
\end{align*}
In particular, if $f\in C^{n,\alpha}(I;\C)$ we have
$$\norma{\triangle_h^{n+1}f}{\infty}\leq\frac{\snorma{\partial_r^n f}{\alpha}}{\binom{n+\alpha}{n}n!}h^{\alpha-1},$$
and if $f\in C^{n+1}(I;\C)$ also
$$\triangle_0^{n+1}f(r):=\lim_{h\rightarrow 0}\triangle_h^{n+1}f(r)=\tfrac{1}{(n+1)!}\partial_r^{n+1}f(r).$$
On the one hand, we will apply this at ``$r=s$'' for ``$h=\xi$'' in Section \ref{sec:BRO}. In particular, for $f\in C^k(\T;\C)$, the following decomposition will be useful
$$
\triangle_\xi f(s)
=\sum_{j=1}^k\frac{\partial_s^jf(s)}{j!}\xi^{j-1}
+\xi^k\triangle_\xi^{k+1}f(s).
$$
On the other hand, we will apply this at ``$r=0$'' for ``$h=t$'' with the abbreviation
$$f^{(n)}\equiv\triangle_0^nf=\tfrac{1}{n!}\partial_t^nf(0),$$
and
$$
f^{n)}\equiv T_0^nf=\sum_{k=0}^nf^{(k)}t^k,\quad\quad
f^{(n+1}\equiv\Delta_t^{n+1}f
=\frac{f-f^{n)}}{t^{n+1}}
=\frac{f^{(n}-f^{(n)}}{t}.
$$
\end{enumerate}

\section{The subsolution}\label{sec:subsolution}

Our weak solutions to (IE) are obtained from a subsolution $(\bar{v},\bar{p},R)$ (cf.~Definition \ref{defi:subsolution} below) via convex integration and indeed, the key point in our approach is to construct the subsolution. This is in alignment with the classical approach to turbulent flows via the Reynolds decomposition, splitting the velocity field into $\velocity=\bar{\velocity}+\velocity'$ (and $p=\bar{p}+p'$) where $\bar{\velocity}$ represents a mean velocity and $\velocity'$ the corresponding fluctuation (cf.~\cite{HPfluid}). Then, formally $(\bar{\velocity},\bar{p},\bar{R})$ solves
\begin{subequations}
\begin{align}
\partial_t\bar{\velocity}+\Div(\bar{\velocity}\otimes\bar{\velocity}+\bar{R})+\nabla\bar{p}&=0,\label{ER}\\
\Div\bar{\velocity}&=0,
\end{align}
\end{subequations}
where
$$
\bar{R}=\overline{\velocity\otimes\velocity}-\bar{\velocity}\otimes\bar{\velocity}
=\overline{\velocity'\otimes\velocity'}
$$
is the \textbf{Reynolds stress tensor}, which satisfies $R\geq 0$ (positive semidefinite). Observe in particular that
$\tfrac12\tr\bar{R}=e-\tfrac{1}{2}|\bar\velocity|^2$, with $e=\tfrac{1}{2}|\velocity|^2$ being the kinetic energy density (cf.~Example~\ref{Ex:flat}). 
Observe that in \eqref{ER} the term $\tr\bar{R}$ may be absorbed in the pressure $\bar{p}$. A subsolution is then defined as follows.

\begin{defi}\label{defi:subsolution} A triple $(\bar{\velocity},\bar{p},R)\in L^\infty(0,T;L^\infty_{\sigma}(\R^2)\times L^\infty(\R^2)\times L^\infty(\R^2;\mathcal{S}))\cap C([0,T];L^\infty_{w^*}(\R^2))$ is a weak solution to the incompressible Euler-Reynolds equations (IER) if
\begin{equation}\label{subsolution:weak}
\int_0^t\int_{\R^2}(\bar{\velocity}\cdot\partial_t\Psi+\sigma:\nabla\Psi)\dif x\dif\tau
=\int_{\R^2}\bar{\velocity}(t)\cdot\Psi(t)\dif x-\int_{\R^2}\velocity_0\cdot\Psi_0\dif x
\end{equation}
holds for every test function $\Psi\in C_c^1(\R^3;\R^2)$  and $0\leq t\leq T$, being
$\sigma\equiv\bar{\velocity}\otimes\bar{\velocity}+ R+\bar{p}I$. \\
Let $\Turzone$ be an open subset of $[0,T]\times\R^2$ and $e\in C^0(\Turzone;\R_+)$. A triple $(\bar{\velocity},\bar{p},R)$ is a strict \textbf{subsolution} to (IE) w.r.t.~$e$ if $(\bar{\velocity},\bar{p},R)$ is a weak solution to (IER) with: 
\begin{enumerate}
\item[(i)] outside $\Turzone$ we have $R=0$,
\item[(ii)] in $\Turzone$ the functions $(\bar{\velocity},R)$ are continuous and
$$
\tfrac{1}{2}|\bar{\velocity}|^2+|\mathring{R}|<e.
$$ 
\end{enumerate}
Observe that (ii) is equivalent to $\bar{R}:=\mathring{R}+(e-\tfrac{1}{2}|\bar{\velocity}|^2)\mathrm{ Id}>0$ in $\Turzone$.
For simplicity of notation we extend $e$ as $\tfrac{1}{2}|\bar{\velocity}|^2$ outside $\Turzone$.
\end{defi}

We recall the following result from \cite{Onadmissibility}, guaranteeing the existence of (infinitely many) weak solutions based on a strict subsolution.

\begin{thm}[Subsolution Criterion \cite{Onadmissibility}]\label{thm:hprinciple} Suppose there exists a subsolution $(\bar{\velocity},\bar{p},R)$ to $\mathrm{(IE)}$ w.r.t.~some $e\in C^0(\Turzone;\R_+)$ with $\Turzone$ given by \eqref{def:turzone}. 
Then, there exist infinitely many weak solutions $(\velocity,p)$ to $\mathrm{(IE)}$ with $\velocity$ satisfying
\begin{align*}
\velocity=\bar{\velocity}\quad &\textrm{outside }\Turzone,\\
\tfrac{1}{2}|\velocity|^2=e\quad &\textrm{in }\Turzone,
\end{align*}
and $p$ given by
\begin{equation}\label{pbarp}
p=\bar{p}+\tfrac{1}{2}(|\bar{\velocity}|^2+\tr R)-e.
\end{equation}
%Moreover, there exists  a sequence $(v_k,p_k)$ of such weak solutions with 
%$v_k\rightharpoonup\bar{v}$ weakly in $L^2(\Turzone)$.
\end{thm}
%Moreover, by adapting the techniques in \cite{Onadmissibility} (see also \cite{Degraded}) to the case of space-time domains $\Turzone$ we may in addition ensure that for any $w\in L^2(\R^2)$
%$$
%\sup_{t}\left|\int_{\Turzone(t)}(v_k-v)(t,x)w(x)\,dx\right|\to 0\textrm{ as }k\to \infty.
%$$

By this result, the construction of weak solutions as in Theorem \ref{thm:global} is reduced to constructing suitable subsolutions $(\bar{\velocity},\bar{p},R)$ adapted to $\Turzone$. In the following Sections \ref{sec:velocity}-\ref{sec:R} we introduce our \emph{ansatz} for the velocity $\bar{\velocity}$, define the corresponding pressure $\bar{p}$ and derive conditions (see Proposition \ref{p:solvability}) under which a Reynolds stress leading to a strict subsolution exists. 

\subsection{The velocity}\label{sec:velocity}

Following \cite{Piecewise}, our central \emph{ansatz} is that the vorticity of the subsolution $\bar{\omega}(t):=\Curl\,\bar{\velocity}(t)$ is concentrated on the boundary of the turbulence zone, i.e.
\begin{equation}\label{ansatzomega}
\bar{\omega}(t):=\frac{1}{2}\sum_{\lambda=\pm1}\map_{\lambda}(t)^\sharp(\varpi_{\lambda}(t)\dif s),
\end{equation}
with the vortex sheet strengths $\varpi_{\lambda}$ to be determined. As in \eqref{varpi:1}, it is convenient (in fact necessary) to write it as
\begin{equation}\label{varpi}
\varpi_{\lambda}=\varpi_0+\partial_s\tilde{\varpi},
\end{equation}
with $\tilde{\varpi}$ to be determined. We note in passing that one could also consider different $\tilde{\varpi}_{\lambda}$ for $\lambda=\pm 1$, but for our purposes this additional freedom of choice is not needed. 

Once $\bar{\omega}$ is chosen, the velocity $\bar{\velocity}$ is determined by the Biot-Savart law as follows below.

\begin{prop}[Biot-Savart law]\label{prop:BSlaw} Given $\omega\in\mathcal{M}_c(\R^2)$, any distributional solution $v$ to
\begin{equation}\label{BS}
\Div\velocity=0,\quad
\Curl\velocity=\omega,
\end{equation}
satisfies $v^*\in\BSO(\omega)^*+\mathrm{Hol}(\C)$, where $\BSO:\mathcal{M}_c(\R^2)\rightarrow L_{\mathrm{loc}}^1(\R^2)$ is the Biot-Savart operator
$$\BSO(\omega)(x)^*:=(K*\omega)(x)
=\frac{1}{2\pi i}\int_{\R^2}\frac{\dif\omega(y)}{x-y},
\quad a.e.\,x\in\R^2,$$
and $K$ is the Cauchy kernel
$$K(x):=\frac{1}{2\pi i x}.$$
Furthermore, $\BSO(\omega)^*\in\mathrm{Hol}(\R^2\setminus\Gamma)$ where $\Gamma\equiv\sop\omega$ with decay 
\begin{equation}\label{BSO:decay}
\BSO(\omega)(x)^*=\frac{1}{2\pi i x}(\omega(\R^2)+\mathcal{O}(|x|^{-1})),
\quad |x|\gg 1.
\end{equation}
\end{prop}

Thus, for $t>0$ the velocity is given by $\bar{\velocity}(t):=\BSO(\bar{\omega}(t))$, 
\begin{equation}\label{u}
\bar{\velocity}(t,x)^*
=\frac{1}{2}\sum_{\lambda=\pm 1}\frac{1}{2\pi i}\int_{\T}\frac{\varpi_{\lambda}(t,s)}{x-\map_{\lambda}(t,s)}\dif s,\quad x\notin\Gamma(t),
\end{equation}
which is the unique distributional solution to \eqref{BS} for \eqref{ansatzomega} vanishing at infinity \eqref{BSO:decay}. As we shall see in Section \ref{sec:BRO},  $\bar{\velocity}(t)$ is bounded, anti-holomorphic outside $\Gamma(t)$ but with tangential discontinuities across $\Gamma(t)$. Indeed, by the classical Sokhotski-Plemelj theorem (cf.~\eqref{e:vpm} in Section \ref{sec:BRO}) these limits $\bar{\velocity}_{\lambda}^{\pm}(t,s)$ are (recall \S\ref{sec:Notation}\ref{Notation:traces})
\begin{equation}\label{velocity:traces}
\bar{\velocity}_{\lambda}^{\pm}
=\BRO_{\lambda}\mp\frac{1}{4}\frac{\varpi_{\lambda}}{\partial_s\map_{\lambda}^*},
\end{equation}
where $\BRO_{\lambda}\equiv\BRO_{\lambda}(\zcurve,\varpi)$ are the Birkhoff-Rott type operators
\begin{equation}\label{BRO:0}
\BRO_{\lambda}(t,s)^*
=\frac{1}{2}\sum_{\mu=\pm 1}\frac{1}{2\pi i}\PV\!\int_{\T}\frac{\varpi_{\mu}(t,s')}{\map_{\lambda}(t,s)-\map_{\mu}(t,s')}\dif s'.
\end{equation}
Notice that the $\PV$ is not necessary for $\mu\neq\lambda$ when $t>0$.
Therefore,
\begin{equation}\label{/u/}
\M{\bar{\velocity}}{\lambda}=\BRO_{\lambda},
\quad\quad
\J{\bar{\velocity}}{\lambda}=-\frac{1}{2}\frac{\varpi_{\lambda}}{\partial_s\map_{\lambda}^*}.
\end{equation}

\subsubsection{Helmholtz decomposition of $\bar{v}$}\label{sec:Helmholtz} It is well-known that an incompressible and irrotational vector field on a simply connected domain can be expressed as the gradient of an harmonic function. However, since $\Omega_{-}(t)$ and $\Turzone(t)$ are not simply connected, we must add the corresponding circulation (\eqref{u:harcir} below). This expression is indeed necessary to recover the pressure $\bar{p}$ outside $\Gamma$ (see \eqref{Blaw} below). 
\begin{defi}\label{defi:circulation}
Let $\Omega\subset\C$ be open and $f\in C(\Omega;\C)$. Given a closed, positively oriented, simple curve $\gamma\in C^1(\T;\C)$, the \textbf{circulation} of $f$ around $\gamma$ is defined as
$$\circulation_{\gamma}(f)
:=\int_{\gamma}f\cdot\dif x.$$
In particular, the \textbf{index} of $x_0\in\C\setminus\gamma(\T)$ w.r.t.~$\gamma$ is  defined as ($K_{x_0}\equiv K(\cdot - x_0)$)
$$\Ind_{\gamma}(x_0):=
\circulation_{\gamma}(K_{x_0}^*)
=\Re\int_\gamma K_{x_0}\dif x
=\frac{1}{2\pi i}\int_{\gamma}\frac{\dif x}{x-x_0}
=\left\lbrace
\begin{array}{rl}
1, & x_0\textrm{ ``inside'' }\gamma(\T),\\[0.1cm]
0, & x_0\textrm{ ``outside'' }\gamma(\T).
\end{array}\right.$$
\end{defi}
\begin{prop}\label{conscirc}
Let $\omega\in\mathcal{M}_c(\R^2)$ and $\Gamma\equiv\sop\omega$. Then, for every $\gamma$ as in Definition \ref{defi:circulation} with $\gamma(\T)\subset\C\setminus\Gamma$, we have
$$\circulation_{\gamma}(\BSO(\omega))
%=\int_{\gamma}\bar{\velocity}\dif x^*
=\omega(\Gamma_\gamma),$$
where $\Gamma_\gamma\equiv\{x\in\Gamma\,:\,\Ind_\gamma(x)=1\}$.
\end{prop}
\begin{proof}
On the one hand
$$\int_{\gamma}\BSO(\omega)\dif x^*
=\int_{\gamma}\BSO(\omega)\cdot\dif x+
i\int_{\gamma}\BSO(\omega)\cdot\dif x^\perp
=\circulation_\gamma(\BSO(\omega)),$$
where the last equality follows from
Gauss divergence theorem because $\Div\BSO(\omega)=0$. On the other hand, Fubini's theorem and Cauchy's integral formula yield
$$
\int_{\gamma}\BSO(\omega)\dif x^*
=\left(\frac{1}{2\pi i}\int_{\gamma}\int_{\Gamma}\frac{\dif\omega(y)}{x-y}\dif x\right)^*=\left(\frac{1}{2\pi i}\int_{\Gamma}\int_{\gamma}\frac{\dif x}{x-y}\dif\omega(y)\right)^*
=\int_{\Gamma_\gamma}\dif\omega(y),$$
as we wanted to prove.
\end{proof}

Given $x\in\R^2$ let us denote $L_x$ by the half-line $x+(-\infty,0]$.
Let us fix $x_0\in\Omega_+(0)$ so that $L_{x_0}\cap\Gamma_0$ consists of a single point and consequently for $0\leq t\leq T_0$ small enough $L_{x_0}\cap\Gamma_{\lambda}(t)=\{\map_{\lambda}(t,s_{t,\lambda})\}$ and $\Omega_{r}(t)\setminus L_{x_0}$ is simply connected for each region $r\in\{-,+,\mathrm{tur}\}$. Fix also some $x_{t,r}\in\Omega_{r}(t)\setminus L_{x_0}$.

By Proposition~\ref{conscirc}, the circulation $\circulation_\gamma$ of $\bar{v}=\BSO(\bar{\omega})$ 
around any $\gamma\subset\Omega_{r}(t)$ $\circlearrowleft$-surrounding $x_0$ is constant on $\Omega_{r}(t)$, for each region $r\in\{-,+,\mathrm{tur}\}$. Moreover, since $\int\varpi(t)=\int\varpi_0$ by \eqref{varpi}, it is in fact constant on $\Omega_{r}$ with
\begin{equation}\label{circulation}
\circulation_\gamma(\bar{\velocity})=\circulation_r
%\textrm{ for }\gamma\subset \Omega_{r}(t)
:=
\begin{cases}
0& r=+,\\
\tfrac{1}{2}\int\varpi_0 & 
r=\textrm{tur},\\
\int\varpi_{0} & r=-,
\end{cases}
\quad\textrm{for }\gamma\,\circlearrowleft\textrm{-surrounding }x_0\textrm{ on }\Omega_{r}(t).
\end{equation}
Thus, we deduce that there exists a (piecewise) harmonic function $\phi(t)$ so that
\begin{equation}\label{u:harcir}
\bar{\velocity}=\nabla\phi+\circulation K_{x_0}^*,
\end{equation}
where $\circulation=\circulation(t,x)$ is defined to be the step function taking the value $\circulation_r$ whenever $x\in \Omega_r(t)$ for $r\in \{-,+,\textrm{tur}\}$. Indeed, by this choice we obtain, for any $\gamma\in C^1(\T;\Omega_{r}(t))$ surrounding $x_0$
$$
\circulation_\gamma(\bar{\velocity}-\circulation K_{x_0}^*)=0,
$$ 
since $
\circulation_\gamma(K_{x_0}^*)=I_\gamma(x_0)=1.
$
Hence $\phi$ can be recovered via
\begin{align}
\phi(t,x)-\phi_{x_{t,r}}
&:=\int_{\gamma}(\bar{\velocity}-\circulation K_{x_0}^*)(t,y)\cdot\dif y\label{phi}\\
&=\frac{\Re}{2}\sum_{\lambda=\pm 1}\frac{1}{2\pi i}\int_{\T}\varpi_{\lambda}(t,s)\int_{\gamma}\left(\frac{1}{y-\map_{\lambda}(t,s)}
-\frac{1-\Ind_{\map_{\lambda}(t)}(x)}{y-x_0}
\right)\dif y\dif s,\nonumber
\end{align}
for any path $\gamma\in C^1([0,1];\Omega_{r}(t)\setminus L_{x_0})$ with $\gamma(0)=x_{t,r}$ and $\gamma(1)=x$, where $\phi(t,x_{t,r})=\phi_{x_{t,r}}$ may be chosen.
By definition \eqref{circulation}\eqref{phi}, $\phi(t,\cdot)$ is continuous on $\Omega_{r}(t)$ and verifies \eqref{u:harcir}. Hence, $\phi(t,\cdot)$ is harmonic.
On the one hand, 
$$\int_{\gamma}\frac{\dif y}{y-\map_{\lambda}(t,s)}
=\mathrm{L}_{\map_\lambda(t,s)}(y)\Big|_{y=x_{t,r}}^{y=x}$$
where $\mathrm{L}_{\map_\lambda(t,s)}$ is the branch of the logarithm given by the ray $l_{\map_\lambda(t,s)}:[s,\infty)\rightarrow\C$
$$
l_{\map_\lambda(t,s)}(s')=\left\lbrace
\begin{array}{rl}
\map_{\lambda}(t,s'), & s'\in[s,s_{t,\lambda}),\\[0.1cm]
\map_\lambda(t,s_{t,\lambda})+(s_{t,\lambda}-s'), & s'\in[s_{t,\lambda},\infty).
\end{array}\right.$$
On the other hand, 
$$
\int_{\gamma}\frac{\dif y}{y-x_0}
=\Log(y-x_0)\Big|_{y=x_{t,r}}^{y=x}
$$
where $\Log$ is the principal branch of the logarithm. Hence, \eqref{phi} reads as
$$
\phi(t,x)
=\frac{\Re}{2}\sum_{\lambda=\pm 1}\frac{1}{2\pi i}\int_{\T}\varpi_{\lambda}(t,s)(\mathrm{L}_{\map_{\lambda}(t,s)}(x)-(1-\Ind_{\map_{\lambda}(t)}(x))\Log(x-x_0))\dif s+O(t),
$$
where $O(t)$ is an arbitrary function of $t$.
In particular, by \eqref{varpi} we deduce
$$
\partial_t\phi(t,x)
=\frac{\Re}{2}\sum_{\lambda=\pm 1}\frac{1}{2\pi i}\int_{\T}\left(\partial_{st}\tilde{\varpi}(t,s)\mathrm{L}_{\map_{\lambda}(t,s)}(x)-\varpi_{\lambda}(t,s)\frac{\partial_t\map_{\lambda}(t,s)}{x-\map_{\lambda}(t,s)}\right)\dif s+\partial_tO(t).
$$
Hence, integrating by parts, we can choose $O(t)$
in such a way that
\begin{equation}\label{phit}
\partial_t\phi(t,x)
=\frac{\Re}{2}\sum_{\lambda=\pm 1}\frac{1}{2\pi i}\int_{\T}\frac{(\partial_{t}\tilde{\varpi}\partial_s\map_{\lambda}-\varpi\partial_t\map_{\lambda})(t,s)}{x-\map_{\lambda}(t,s)}\dif s.
\end{equation}
Then, using the Sokhotski-Plemelj theorem, we deduce (recall \S\ref{sec:Notation}\ref{Notation:complex})
\begin{equation}\label{eq:potjump}
\J{\partial_t\phi}{\lambda}
=-\frac{\Re}{2}\left(\frac{\partial_{t}\tilde{\varpi}\partial_s\map_{\lambda}-\varpi\partial_t\map_{\lambda}}{\partial_s\map_{\lambda}}\right)
=\tfrac{1}{2}(
\varpi\Re_{\partial_s\map_\lambda}(\partial_{t}\map_\lambda)-\partial_{t}\tilde{\varpi}).
\end{equation}

\subsection{The pressure}\label{sec:p}

We define $\bar{p}$ outside $\Gamma$ by means of the Bernoulli's law
\begin{equation}\label{Blaw}
\bar{p}:=-\partial_t\phi-\tfrac{1}{2}|\bar{\velocity}|^2
\quad\textrm{outside }\Gamma,
\end{equation}
with $\phi$ given in \eqref{phi}. 
%Notice $\bar{p}$ is bounded and vanishes at infinity.
Since $\bar{\velocity}$ is $\Div$-$\Curl$ free outside $\Gamma$, a simple computation yields
$$\Div(\bar{\velocity}\otimes\bar{\velocity})
%=\bar{\velocity}\cdot\nabla\bar{\velocity}
=\tfrac{1}{2}\nabla(|\bar{\velocity}|^2)
\quad\textrm{outside }\Gamma.$$
Thus, by applying $\nabla$ on \eqref{Blaw} we deduce
\begin{equation}\label{EqBernoulli:0}
\partial_t\bar{\velocity}+\Div(\bar{\velocity}\otimes\bar{\velocity})+\nabla \bar{p}=0
\quad\textrm{outside }\Gamma.
\end{equation}

\begin{prop}\label{p:[p]} The jump in $\bar{p}$ across $\Gamma_{\lambda}$ is
\begin{equation}\label{[p]}
[\bar{p}]_{\lambda}=\tfrac{1}{2}(\partial_t\tilde{\varpi}-\varpi\Re_{\partial_{s}\map_{\lambda}}(\partial_t\map_{\lambda}-\BRO_{\lambda})).
\end{equation}	
\end{prop}
\begin{proof}
By applying \eqref{eq:potjump} coupled with \eqref{/u/} and
\begin{equation}\label{jump|u|}
\tfrac{1}{2}\J{|\bar{\velocity}|^2}{\lambda}
=\M{\bar{\velocity}}{\lambda}\cdot\J{\bar{\velocity}}{\lambda}
=-\tfrac{1}{2}\varpi\Re_{\partial_{s}\map_{\lambda}}(\BRO_{\lambda}),
\end{equation}
Bernoulli's law \eqref{Blaw} implies \eqref{[p]}.
\end{proof}

\subsection{The Reynolds stress}\label{sec:R}

\begin{prop}\label{prop:consmom} Let $(\bar{\velocity},\bar{p})$ given by \eqref{u} and \eqref{Blaw} respectively. Then, $(\bar{\velocity},\bar{p},R)$ is a weak solution to $\mathrm{(IER)}$ if and only if, at each time slice $t>0$, $R=R(t,x)$ solves
\begin{subequations}
\label{R}
\begin{align}
\Div R&=0 \hspace{1.43cm} \textrm{in }\Omega_{\mathrm{tur}}, \label{R:1}\\
\pm(R\partial_s\map^\perp)_{\pm}&=i\B_{\pm} \hspace{1.0cm} \textrm{on }\Gamma_{\pm},\label{R:2}
\end{align}
\end{subequations}
where, for $\lambda=\pm 1$, $\B_{\lambda}$ are the boundary conditions
\begin{equation}\label{R:4}
\B_{\lambda}
=\tfrac{1}{2}(\partial_t\tilde{\varpi}\partial_s\map_\lambda
-\varpi(\partial_t\map_\lambda-\BRO_\lambda)).
\end{equation}
\end{prop}
\begin{proof} 
Let us parametrize $\Gamma_{\lambda}$ by the map
$$\mathrm{X}_{\lambda}(t,s):=(t,\map_{\lambda}(t,s)).$$
Then, since
$$\partial_t\mathrm{X}_{\lambda}\times\partial_s\mathrm{X}_{\lambda}
=(1,\partial_t\map_{\lambda})\times(0,\partial_s\map_{\lambda})
=(-\partial_t\map\cdot\partial_s\map^\perp,\partial_s\map^\perp)_{\lambda},$$
the outward (w.r.t.~$\Turzone$) unit normal vector to $\Gamma_{\lambda}$ is
$$\normal{\lambda}{}=\sgn\lambda\frac{(-\partial_t\map\cdot\partial_s\map^\perp,\partial_s\map^\perp)_{\lambda}}{|(-\partial_t\map\cdot\partial_s\map^\perp,\partial_s\map^\perp)_{\lambda}|}.$$
Thus, by splitting $[0,T]\times\R^2$ into each $\Omega_{r}$ for $r\in\{-,+,\mathrm{tur}\}$ and
integrating by parts, we deduce that, for every test function $\Psi\in C_c^1(\R^3;\R^2)$, we have
\begin{align*}
\int_0^t\int_{\R^2}(\bar{\velocity}\cdot\partial_t\Psi+\sigma:\nabla\Psi)\dif x\dif\tau
-\int_{\R^2}\bar{\velocity}(t)\cdot\Psi(t)\dif x+\int_{\R^2}\velocity_0\cdot\Psi_0\dif x\\
=\sum_{\lambda=\pm 1}\int_0^t\int_{\T}((\partial_t\map\cdot\partial_s\map^\perp)[\bar{\velocity}]-[\sigma]\partial_s\map^\perp)_{\lambda}\cdot\Psi_{\lambda}\dif s\dif\tau
-\int_0^t\int_{\Turzone(\tau)}(\Div R)\cdot\Psi\dif x\dif\tau,
\end{align*}
where we have applied $\bar{\velocity}|_{t=0}=\velocity_0$ and \eqref{EqBernoulli:0}.
Therefore, \eqref{subsolution:weak} is equivalent to \eqref{R:1} and
\begin{equation}\label{consmomGamma2}
(\partial_t\map\cdot\partial_s\map^\perp)_{\lambda}[\bar{\velocity}]_{\lambda}=[\sigma]_{\lambda}\partial_s\map_{\lambda}^\perp
\quad\quad\textrm{in }\Gamma_{\lambda}.
\end{equation}
On the one hand, by \eqref{velocity:traces} (recall \S\ref{sec:Notation}\ref{Notation:complex}),
$$(\partial_t\map\cdot\partial_s\map^\perp)_{\lambda}[\bar{\velocity}]_{\lambda}
=-\varpi\Im_{\partial_s\map_{\lambda}}(\partial_t\map_{\lambda})\partial_s\map_{\lambda}.$$
On the other hand, let us split
$$[\sigma]_{\lambda}=[\bar{\velocity}\otimes\bar{\velocity}]_{\lambda}+[\bar{p}]_{\lambda}+
[R]_{\lambda}.$$
Then, by applying the identity $(a\otimes b)c=a(b\cdot c)$ we get
$$[\bar{\velocity}\otimes\bar{\velocity}]_{\lambda}\partial_s\map_{\lambda}^\perp
=\bar{\velocity}_{\lambda}^+(\bar{\velocity}^+\cdot\partial_s\map^\perp)_{\lambda}
-\bar{\velocity}_{\lambda}^-(\bar{\velocity}^-\cdot\partial_s\map^\perp)_{\lambda}
=[\bar{\velocity}]_{\lambda}(\BRO\cdot\partial_s\map^\perp)_{\lambda}
=-\varpi\Im_{\partial_s\map_{\lambda}}(\BRO_{\lambda})\partial_s\map_{\lambda}.$$
Therefore, since $[R]_{\pm}=\mp R_{\pm}$, \eqref{consmomGamma2} reads as
$$\pm R_{\pm}\partial_s\map_{\pm}^\perp
=i\underbrace{([\bar{p}]_\pm-\tfrac{1}{2}i\varpi\Im_{\partial_s\map_\pm}(\partial_t\map_\pm-\BRO_\pm))\partial_s\map_{\pm}}_{b_\pm}\quad\quad\textrm{in }\Gamma_{\pm},$$
where we applied Proposition \ref{[p]}.
\end{proof}

\begin{Rem}\label{Rem:BReq}
Observe that in the case of a single sheet ($\Gamma_+=\Gamma_-$) the above analysis reduces to the derivation of the Birkhoff-Rott system \eqref{BReq} from the weak formulation of the incompressible Euler equations: 
$\B:=\partial_t\tilde{\varpi}\partial_s\zcurve-\varpi(\partial_t\zcurve-\BRO)=0$.
As a result, Proposition \ref{p:[p]} implies continuity of the pressure across the sheet: $[p]=\Re_{\partial_{s}\zcurve}\B=0$, thus generalizing the observation made in \cite{CCG12} that the continuity of the pressure can be deduced from the weak formulation and need not appear as an assumption as for instance in \cite{MB02}.\\
In the case of two sheets ($\Gamma_+\neq\Gamma_-$) one may set $R=0$ by imposing $\B_{+}=\B_{-}=0$, but this seems as problematic as the single sheet case. In other words, the Reynolds stress $R$ allows to relax the ill-posed equations $\B_{+}=\B_{-}=0$ (cf.~Section \ref{sec:proof}).
\end{Rem}

\subsection{Solvability of \eqref{R}}

Next we discuss necessariy and sufficient conditions for solvability of the boundary value problem \eqref{R}. Recall the isomorphism $\mathrm{M}$ and our notation $\langle b\rangle=\frac12(b_++b_-)$ and $\{b\}=\frac12(b_+-b_-)$ introduced in Section \ref{sec:Notation}\ref{Notation:complex} and \ref{Notation:traces} respectively. 

\begin{prop}\label{p:solvability}
The boundary value problem \eqref{R} admits a solution $R=R(t,x)$ uniformly bounded in $t>0$ if and only if 
\begin{enumerate}[(a)]
\item\label{Bq}$\mean{\B}=t\partial_s(q\partial_s\zcurve)$ 
\end{enumerate} 
for some $q=q_1+iq_2$ satisfying
\begin{enumerate}[(a)]
\addtocounter{enumi}{1}
\item\label{p:solvability:average} $\int(q_1|\partial_s\zcurve|^2-\dev{\B}\cdot\vectW)=0$,
\item\label{p:solvability:pointwise} $q_1^{(0)}=\dev{\B}^{(0)}\cdot\vectW$,
\end{enumerate} 
\end{prop}
The question of solvability involves two issues. First, since $R$ is divergence-free, Gauss divergence theorem leads to an integrability condition for the boundary values - this is represented by \ref{Bq}\ref{p:solvability:average} above. Secondly, the thickness of the domain $\Omega_{\mathrm{tur}}(t)$ is order $\sim t$, so that, in order to make sure that $R$ is uniformly bounded in $t$, there is a necessary matching condition at $t=0$ - represented by \ref{p:solvability:pointwise} above. 

We split the proof of Proposition \ref{p:solvability} in two parts.  
First we look at necessary and sufficient conditions on $\B_{\lambda}$ for solvability of \eqref{R} at any fixed time slice. In the following we will use the notation $G:=g\circ \map=g^\sharp$ to denote the change of variables adapted to $\Turzone$ as
\begin{equation}\label{eq:defG}
G(t,s,\lambda)=g^\sharp(t,s,\lambda)=g(t,\map_\lambda(t,s)).
\end{equation}
In the following lemma we fix $t>0$ and, for ease of notation, supress dependence on $t$.

\begin{lemma}\label{R:solution} 
Given $\B_{\lambda}\in C(\T;\R^2)$, there exists $R\in C^1(\Omega_{\mathrm{tur}};\mathcal{S})\cap C(\bar{\Omega}_{\mathrm{tur}};\mathcal{S})$ solving \eqref{R} if and only if the following compatibility conditions hold
\begin{subequations}
\label{bcond}
\begin{align}
\int\mean{\B}&=0,\label{bcond:1}\\
\int\mean{\B\cdot\map}&=0.\label{bcond:2}
\end{align}
\end{subequations}
In this case, %any solution $R$ is of the form
$$R=\left(\begin{array}{rr}
\partial_{22}g & -\partial_{12}g \\
-\partial_{12}g & \partial_{11}g
\end{array}\right)
+\mathrm{M}\F,$$
where
$$\F:=\alpha^*K_{x_0}+\beta K_{x_0}'
\quad\textrm{with}\quad
\begin{array}{l}
\displaystyle\alpha:=\int\dev{\B},\\[0.1cm]
\displaystyle\beta:=\alpha\cdot x_0-\int\dev{\B\cdot\map},
\end{array}$$
for some $g\in C^3(\Omega_{\mathrm{tur}})\cap C^2(\bar{\Omega}_{\mathrm{tur}})$ satisfying
\begin{equation}
\partial_s(\nabla g)_{\pm}
=\pm\B_{\pm}-(\F\partial_s\map)_{\pm}^*.
\end{equation}
\end{lemma}
\begin{proof} Let us start assuming that $R$ is a solution to \eqref{R}. Then, Gauss divergence theorem implies \eqref{bcond:1}
$$
0=\int_{\Omega_{\mathrm{tur}}}\Div R
=2\int\dev{R\partial_s\map^\perp}
=2i\int\mean{\B}.
$$
But then, for any simple closed curve $\gamma\in C^1(\T;\Turzone)$ $\circlearrowleft$-surrounding $x_0$ we have $\int_{\gamma}R\dif y^\perp=i\int\dev{\B}$.
Consequently, there exists $\Psi=(\psi_1,\psi_2)\in C^2(\Omega_{\mathrm{tur}};\R^2)\cap C^1(\bar{\Omega}_{\mathrm{tur}};\R^2)$ so that
\begin{equation}\label{eqR:1}
R
=\left(\begin{array}{cc}
-\partial_2\psi_1 & \partial_1\psi_1\\
-\partial_2\psi_2 & \partial_1\psi_2
\end{array}\right)+\mathrm{M}(\alpha^* K_{x_0}).
\end{equation}
Indeed, for any simple closed curve $\gamma\in C^1(\T;\Turzone)$ $\circlearrowleft$-surrounding $x_0$, using $\mathrm{M}(z)w=-(zw)^*$, we have
$$
\int_{\gamma}(R-\mathrm{M}(\alpha^*K_{x_0}))\dif x^\perp=\int_{\gamma}R\dif x^\perp-i\alpha=i\int\dev{\B}-i\alpha=0.
$$
Therefore $\Psi$ can be recovered via
$$\Psi(x)-\Psi_{x_{\mathrm{tur}}}
=\int_{\gamma}R\dif y^\perp-i\alpha\left(\int_{\gamma}K_{x_0}\dif y\right)^*,$$
for any path $\gamma\in C^1([0,1];\Turzone\setminus L_{x_0})$ with  $\gamma(0)=x_{\mathrm{tur}}$ and $\gamma(1)=x$, where $\Psi(x_{\mathrm{tur}})=\Psi_{x_{\mathrm{tur}}}$ is an arbitrary constant vector.

Now, since $R$ is symmetric, necessarily $\Div\Psi=0$. Hence, since
$$\partial_s\Psi_{\pm}
=((\nabla\Psi)\partial_s\map)_{\pm}
=((R-\mathrm{M}(\alpha^*K_{x_0}))\partial_s\map^\perp)_{\pm}
=i(\pm\B-\alpha(K_{x_0}\partial_s\map)^*)_{\pm}
,$$
with
$$\int(\alpha^* K_{x_0}\partial_s\map)_{\lambda}^*\cdot\map_{\lambda}
=\alpha\cdot\int_{\map_{\lambda}}xK_{x_0}\dif x
=\alpha\cdot x_0,$$
Gauss divergence theorem implies \eqref{bcond:2}
$$0=\int_{\Omega_{\mathrm{tur}}}\Div\Psi
=2\int\dev{\Psi\cdot\partial_s\map^\perp}
=2\int\dev{\partial_s\Psi^\perp\cdot\map}
=-2\int\mean{\B\cdot\map}.$$
Therefore, there is some $g\in C^3(\Omega_{\mathrm{tur}})\cap C^2(\bar{\Omega}_{\mathrm{tur}})$ so that
\begin{equation}\label{eqR:3}
\Psi=\nabla^\perp g+i\beta K_{x_0}^*.
\end{equation}
Indeed, analogously to above, $g$ can be recovered via
$$g(x)-g_{x_{\mathrm{tur}}}
=\int_{\gamma}\Psi\cdot\dif y^\perp-\beta\Re\int_{\gamma}K_{x_0}\dif y,$$
for any path $\gamma\in C^1([0,1];\Turzone\setminus L_{x_0})$ with  $\gamma(0)=x_{\mathrm{tur}}$ and $\gamma(1)=x$, where $g(x_{\mathrm{tur}})=g_{x_{\mathrm{tur}}}$ may be chosen and $\beta\in\C$ is given by the circulation to guarantee the continuity of $g$ on $\Turzone$
$$0=\int_{\map_{\lambda}}\Psi\cdot\dif x^\perp-\beta\Re\int_{\map_{\lambda}}K_{x_0}\dif x
=\left(\alpha\cdot x_0-\int\dev{\B\cdot\map}\right)-\beta.$$
Now, \eqref{eqR:3} and the Cauchy-Riemann equations yield
$$\left(\begin{array}{cc}
-\partial_2\psi_1 & \partial_1\psi_1\\
-\partial_2\psi_2 & \partial_1\psi_2
\end{array}\right)
=\left(\begin{array}{rr}
\partial_{22}g & -\partial_{12}g\\
-\partial_{12}g & \partial_{11}g
\end{array}\right)
+\mathrm{M}(\beta K_{x_0}').$$
\indent Finally, $g$ must satisfy the boundary conditions
\begin{equation}\label{eqR:4}
\partial_s(\nabla g)_{\pm}
=(\nabla^2g\partial_s\map)_{\pm}
=-i((R-\mathrm{M}\F)\partial_s\map^\perp)_{\pm}
=\underbrace{\pm\B_{\pm}-(\F\partial_s\map)_{\pm}^*}_{\equiv\A_{\pm}}.
\end{equation}
%where we have abbreviated
%$$\A_{\pm}\equiv\pm\B_{\pm}-(\F\partial_s\map)_{\pm}^*.$$

Notice that
$$\int\A_{\pm}
=\pm\int\B_{\pm}-\left(\int_{\map_{\pm}}\F\dif x\right)^*=\int\dev{\B}-\alpha=0.$$
Then, \eqref{eqR:4} reads as
\begin{equation}\label{consmomGamma5}
(\nabla g)_{\pm}=\int_0^s\A_{\pm}\dif s_1+o_{\pm},
\end{equation}
for some constant vectors $o_{\lambda}\in\R^2$. Then, since 
$\nabla G=\nabla\map(\nabla g)^\sharp$, 
\eqref{consmomGamma5} is equivalent to
\begin{subequations}
\label{Gcond:1}
\begin{align}
(\partial_s G)_{\pm}&=\left(\int_0^s\A_{\pm}\dif s_1+o_{\pm}\right)\cdot\partial_s\map_{\pm},\label{Gpm:1}\\
(\partial_{\lambda} G)_{\pm}&=\left(\int_0^s\A_{\pm}\dif s_1+o_{\pm}\right)\cdot\vectV.\label{Glambdapm:1}
\end{align}
\end{subequations}
But then notice that
\begin{align*}
\int\left(\int_0^s\A_{\pm}\dif s_1+o_{\pm}\right)\cdot\partial_s\map_{\pm}
&=-\int (\A\cdot\map)_{\pm}
=\Re\int_{\map_{\pm}}x\F\dif x\mp\int (\B\cdot\map)_{\pm}\\
&=(\alpha\cdot x_0-\beta)-\int\dev{\B\cdot\map}=0.
\end{align*}
Hence, \eqref{Gpm:1} reads as
\begin{equation}\label{Gpm}
G_{\pm}=\int_0^s\left(\int_0^{s_1}\A_{\pm}\dif s_2+o_{\pm}\right)\cdot\partial_s\map_{\pm}\dif s_1+ d_{\pm},
\end{equation}
for some constants $d_{\lambda}\in\R$.

\smallskip

Conversely, the easiest way to define $G$ in the interior from the boundary conditions \eqref{Gpm} \eqref{Glambdapm:1} is by means of the Lagrange interpolation. Since there are four conditions, we consider the Lagrange polynomial of degree 3 on $\lambda$ 
\begin{equation}\label{ansatzL}
L(s,\lambda):=\sum_{k=0}^{3}l_k(s)\lambda^k,
\end{equation}
whose coefficients $l_k$ are determined by
\begin{equation*}
\left(\begin{array}{rrrr}
1 & 1 & 1 & 1 \\
1 & -1 & 1 & -1 \\
0 & 1 & 2 & 3 \\
0 & 1 & -2 & 3
\end{array}\right)
\left(\begin{array}{c}
l_0 \\ l_1 \\ l_2 \\ l_3
\end{array}\right)
=\left(\begin{array}{c}
G_+ \\ G_- \\ (\partial_{\lambda}G)_+ \\ (\partial_{\lambda}G)_-
\end{array}\right).
\end{equation*} 
%that is
%$$\left(\begin{array}{c}
%l_0 \\ l_1 \\ l_2 \\ l_3
%\end{array}\right)=\frac{1}{4}
%\left(\begin{array}{rrrr}
%2 & 2 & -1 & 1 \\
%3 & -3 & -1 & -1 \\
%0 & 0 & 1 & -1 \\
%-1 & 1 & 1 & 1
%\end{array}\right)
%\left(\begin{array}{c}
%G_+ \\ G_- \\ (\partial_{\lambda}G)_+ \\ %(\partial_{\lambda}G)_-
%\end{array}\right),$$
The solution of the above linear system is
\begin{align*}
l_0&=\mean{G}-\tfrac{1}{2}\dev{\partial_{\lambda}G},&
l_1&=\tfrac{1}{2}(3\dev{G}-\mean{\partial_{\lambda}G}),\\
l_2&=\tfrac{1}{2}\dev{\partial_{\lambda}G},&
l_3&=\tfrac{1}{2}(\mean{\partial_{\lambda}G}-\dev{G}).
\end{align*}
Thus, any $g$ has the form $g^\sharp=G=L+H$, for some solution $H$ to the homogeneous problem $(\partial_{\lambda}^kH)_{\pm}=0$ for $k=0,1$. 
This concludes the proof.
\end{proof}

Lemma~\ref{R:solution} above shows how $R(t)$ is at each time slice $0<t\leq T$. Next, we must guarantee that $R(t)$ remains uniformly bounded as $t\downarrow 0$.

%\begin{lemma}\label{lemma:b1} Assuming \eqref{eq:B1-}\eqref{bcondk}, the problem \eqref{R} admits $R$ uniformly bounded if and only if $\tilde{b}^{(0)}=0$. In other words, if and only if
%$$\mean{\B}^{(1)}=\partial_s((\dev{\B}^{(0)}\cdot\vectW+iq_2^{(0)})\partial_s\zcurve_0)$$
%holds for some $q_2^{(0)}\in C^1(\T)$.
%\end{lemma}
\begin{proof}[Proof of Proposition~\ref{p:solvability}]  
Recalling the definition of $G=g^\sharp$ from \eqref{eq:defG} we have
\begin{equation}\label{eq:derivG}
\begin{split}
\partial_sG&=(\nabla g)^\sharp\cdot\partial_s\map,\\
\partial_{\lambda}G&=t(\nabla g)^\sharp\cdot\vectW,
\end{split}
\end{equation}
%with
%$$(\nabla g)^\sharp=(D\map)^{-1}\nabla G
%=\frac{1}{t\partial_s\map^\perp\cdot\vectW}
%(\partial_{\lambda}G\partial_s\map-t\partial_sG\vectW)^\perp.$$
and
\begin{subequations}
\label{D2g}
\begin{align}
\partial_{ss}G&=\partial_s\map \cdot(\nabla^2g)^\sharp \partial_s\map+(\nabla g)^\sharp\cdot\partial_{ss}\map,\\
\partial_{s\lambda}G&=t\partial_s\map\cdot(\nabla^2g)^\sharp\vectW+t(\nabla g)^\sharp\cdot\partial_{s}\vectW,\\
\partial_{\lambda\lambda}G&=t^2\vectW\cdot (\nabla^2g)^\sharp\vectW.
\end{align}
\end{subequations}
Then, $\nabla^2g$ (equivalently $R$) is uniformly bounded if and only if, as $t\downarrow 0$,
\begin{equation}\label{Gorder}
\partial_s^k\partial_{\lambda}^nG=\mathcal O(t^n),\quad
0\leq k+n\leq 2.
\end{equation}
By considering the Taylor expansion of $\lambda\mapsto G(t,s,\lambda)$ on $[-1,1]$ we see that \eqref{Gorder} implies 
$$
 G_+-G_-=\mathcal O(t),\quad
\partial_\lambda G_+-\partial_\lambda G_-=\mathcal O(t^2),\quad 
(G_++\partial_\lambda G_+)-(G_--\partial_\lambda G_-)=\mathcal O(t^2),
$$
that is,
\begin{subequations}
\label{tcond}
\begin{align}
\mean{\partial_{\lambda}^n G}= \mathcal O(t^{n}),\quad
\dev{\partial_{\lambda}^n G}&=\mathcal O(t^{n+1}),\quad n=0,1,
\label{tcond:1}\\
\dev{G}-\mean{\partial_{\lambda}G}&=\mathcal O(t^2).\label{tcond:2}
\end{align}
\end{subequations}

Using \eqref{Gpm}\eqref{Glambdapm:1}, these terms are
\begin{align*}
\mean{G}
&=\int_0^s\left(\left(\int_0^{s_1}\mean{\A}\dif s_2+\mean{o}\right)\cdot\partial_s\zcurve+t\left(\int_0^{s_1}\dev{\A}\dif s_2+\dev{o}\right)\cdot\partial_s\vectW\right)\dif s_1+\mean{d},\\
\dev{G}
&=\int_0^s\left(\left(\int_0^{s_1}\dev{\A}\dif s_2+\dev{o}\right)\cdot\partial_s\zcurve+t\left(\int_0^{s_1}\mean{\A}\dif s_2+\mean{o}\right)\cdot\partial_s\vectW\right)\dif s_1+\dev{d},\\
\mean{\partial_{\lambda}G}
&=t\left(\int_0^s\mean{\A}\dif s_1
+\mean{o}\right)\cdot\vectW,\\
\dev{\partial_{\lambda}G}
&=t\left(\int_0^s\dev{\A}\dif s_1
+\dev{o}\right)\cdot\vectW.
\end{align*}
Observe that $\mean{\partial_{\lambda}^n G}= \mathcal O(t^{n})$ for $n=0,1$. The conditions $\dev{\partial_{\lambda}^n G}= \mathcal O(t^{n+1})$ for $n=0,1$ implies that $\dev{\A}^{(0)}=\dev{o}^{(0)}=\dev{d}^{(0)}=0$. In particular, since $\mean{\B}^{(0)}=\dev{\A}^{(0)}=0$, the zero-mean condition \eqref{bcond:1} reads as \ref{Bq}:
$$
\mean{\B}=t\partial_s(q\partial_s\zcurve),
$$
for some $q(t,s)=q_1(t,s)+iq_2(t,s)$. We may assume w.l.o.g.~that
\begin{equation}\label{qs=0}
(q\partial_s\zcurve)|_{s=0}=\dev{o}^{(1}+\dashint_{-1}^{1}\F_\lambda\dif\lambda\vectV|_{s=0}.
\end{equation}
By \ref{Bq}, the zero-mean condition \eqref{bcond:2} reads as \ref{p:solvability:average}:
$$0=\int\mean{\B\cdot\map}
=\int(\mean{\B}\cdot\zcurve+t\dev{\B}\cdot\vectW)
=t\int(\dev{\B}\cdot\vectW-q_1|\partial_s\zcurve|^2).$$
Coming back to \eqref{tcond:2}, an integration by parts yields
$$
t\int_0^s\left(\int_0^{s_1}\mean{\A}\dif s_2+\mean{o}\right)\cdot\partial_s\vectW\dif s_1
=\mean{\partial_{\lambda}G}-t\mean{o}\cdot\vectV^{(1)}(0)-t\int_0^s\mean{\A}\cdot\vectW\dif s_1,
$$
from which we deduce that
\begin{equation}\label{jumpslope:1}
\begin{split}
\dev{G}-\mean{\partial_{\lambda}G}\\
=t\bigg(\int_0^s\left(\left(\int_0^{s_1}\dev{\A}^{(1}
\dif s_2+\dev{o}^{(1}\right)\cdot\partial_s\zcurve-\mean{\A}\cdot\vectW\right)\dif s_1+\underbrace{\dev{d}^{(1}-\mean{o}\cdot\vectV^{(1)}(0)}_{\equiv\tilde{d}}\bigg).
\end{split}
\end{equation}
In particular, \eqref{tcond:2} implies that $\tilde{d}^{(0)}=0$. Let us split \eqref{jumpslope:1} in terms of $\B_{\lambda}$ and  $\F$.
On the one hand, 
$$\mean{\A}\cdot\vectW
=\dev{\B}\cdot\vectW-\Re(\mean{\F\partial_s\map}\vectW).$$
On the other hand, since $\F(t)\in\mathrm{Hol}(\C\setminus\{x_0\})$, Cauchy's integral theorem implies
$$\int_0^{s_1}\dev{\F\partial_s\map}\dif s_2
=\dashint_{-1}^{1}\F_\lambda\dif\lambda\vectV\Big|_0^{s_1},$$
and, by \eqref{qs=0},
\begin{equation}\label{b1:1}
\begin{split}
\int_0^{s_1}\dev{\A}^{(1}\dif s_2+\dev{o}^{(1}
&=\int_0^{s_1}\mean{\B}^{(1}\dif s_2+\dev{o}^{(1}-\left(\dashint_{-1}^{1}\F_\lambda\dif\lambda\vectW\Big|_0^{s_1}\right)^*\\
&=q\partial_s\zcurve-\left(\dashint_{-1}^{1}\F_\lambda\dif\lambda\vectW\right)^*.
\end{split}
\end{equation}
Therefore, \eqref{jumpslope:1} reads as
\begin{equation}
\label{jumpslope:2}
\dev{G}-\mean{\partial_{\lambda}G}
=t\int_0^s\bigg(q_1|\partial_s\zcurve|^2
-\dev{\B}\cdot\vectW
+t\tilde{d}^{(1}+\Re(F\vectW)\bigg)\dif s_1,
\end{equation}
where we have abbreviated
$$F\equiv
\mean{\F\partial_s\map}
-\dashint_{-1}^{1}\F_{\lambda}\dif\lambda\partial_s\zcurve.$$
It is straightforward to check that $F^{(0)}=F^{(1)}=0$, so $F=t^2F^{(2}$. 
Therefore, the condition \eqref{tcond:2} requires $q_1^{(0)}=\dev{\B}^{(0)}\cdot\vectW$, i.e. \ref{p:solvability:pointwise}.

Conversely, the Lagrange polynomial $L$ given in \eqref{ansatzL} satisfies \eqref{Gorder} if and only if 
$$\partial_s^k l_n,\ldots,\partial_s^k l_3=\mathcal{O}(t^n),\quad 0\leq k+n\leq 2,$$
which is indeed equivalent to \eqref{tcond}.
\end{proof}

\section{The dissipation}\label{sec:Admissibility}

Recall from Definition \ref{defi:subsolution} and Theorem \ref{thm:hprinciple} that our weak solutions $(v,p)$ to (IE) are obtained from a subsolution $(\bar{v},\bar{p},R)$ via convex integration, and satisfy 
\begin{equation}\label{D:1}
\tfrac{1}{2}|v|^2=e, \quad\quad
e+p=\tfrac{1}{2}(|\bar v|^2+\tr R)+\bar p,
\end{equation}
where
\begin{equation}\label{def:e}
e:=\tfrac{1}{2}|\bar{\velocity}|^2+|\mathring{R}|+e',
\end{equation}
for some error function $e'$ strictly positive on  $\Turzone$ while vanishing outside.
This allows us to calculate the associated dissipation measure:

\begin{prop}\label{prop:dissipation}
Let $(\bar{\velocity},\bar{p},R)$ be a strict subsolution w.r.t.~some $e\in C^0(\Turzone;\R_+)$ and assume that it is $C^1$ outside $\Gamma=\Gamma_-\cup\Gamma_+=\partial\Turzone$ with $\Gamma_{\pm}$ parametrized by $C^1$ curves $\map_{\pm}=\map_{\pm}(t,s)$. Then the dissipation measure $D(t)$ from Definition \ref{defi:dissipation} is supported in $\Turzone$ with
\begin{equation}\label{eq:dissipation}
\begin{split}
\langle D(t),\psi\rangle 
&=\sum_{\lambda=\pm 1}\int_0^t\int ([e]_\lambda\partial_t\map_\lambda+[(e+p)\bar{v}]_\lambda)\cdot\partial_s\map_\lambda^\perp\psi_\lambda\dif s\dif\tau\\
&-\int_0^t\int_{\Turzone(\tau)}\left(\partial_t(e-\tfrac{1}{2}|\bar v|^2)-\bar v\cdot\Div\mathring{R}+(v-\bar v)\cdot\nabla(e+p)\right)\psi\dif x\dif\tau.
\end{split}
\end{equation}
\end{prop}
\begin{proof} 
Since $(\bar{\velocity},\bar{p},R)$ is piecewise $C^1$ outside $\Gamma=\partial\Turzone$, by multiplying the (relaxed) momentum balance equation \eqref{ER} by $\bar{\velocity}$ we get
$$\tfrac{1}{2}\partial_t|\bar{\velocity}|^2
+\Div((\tfrac{1}{2}|\bar{\velocity}|^2+\bar{p})\bar{\velocity})+\bar{\velocity}\cdot\Div R=0
\quad\textrm{outside }\Gamma,$$
or, by \eqref{D:1}, equivalently
\begin{equation}\label{D:2}
\tfrac{1}{2}\partial_t|\bar{\velocity}|^2
+\Div((e+p)\bar{\velocity})+\bar{\velocity}\cdot\Div \mathring{R}=0
\quad\textrm{outside }\Gamma.
\end{equation}
Firstly, by adding and subtracting $\bar{\velocity}$, we split the dissipation into
\begin{subequations}
\label{Dsplit}
\begin{align}
\langle D(t),\psi\rangle
&=\int_{0}^{t}\int_{\R^2}(e\partial_t\psi+(e+p)\bar{\velocity}\cdot\nabla\psi)\dif x\dif\tau
-\int_{\R^2}e\psi\dif x\Big|_{\tau=0}^{\tau=t}\label{Dsubsolution}\\
&+\int_{0}^{t}\int_{\R^2}(e+p)(\velocity-\bar{\velocity})\cdot\nabla\psi\dif x\dif\tau,\label{Dfluctuation}
\end{align}
\end{subequations}
where the first term \eqref{Dsubsolution} only depends on the subsolution by \eqref{D:1}\eqref{def:e}, while the second term \eqref{Dfluctuation} is the corresponding fluctuation.
On the one hand, similarly to the proof of Proposition~\ref{prop:consmom}, an integrating by parts yields
\begin{subequations}
\begin{align}
\eqref{Dsubsolution}
&=\sum_{\lambda=\pm 1}\int_0^t\int ([e]_\lambda\partial_t\map_\lambda+[(e+p)\bar{v}]_\lambda)\cdot\partial_s\map_\lambda^\perp\psi_\lambda\dif s\dif\tau\\
&-\int_0^t\int_{\Turzone(\tau)}(\partial_te+\Div((e+p)\bar{\velocity}))\psi\dif x\dif\tau,\label{D:3}
\end{align}
\end{subequations}
where we have applied that $(\bar{\velocity},\bar{p},R)=(\velocity,p,0)$ outside $\Turzone$ and \eqref{D:2}. Indeed, by \eqref{D:2}, the term \eqref{D:3} reads as
$$\eqref{D:3}
=-\int_0^t\int_{\Turzone(\tau)}(\partial_t(e-\tfrac{1}{2}|\bar v|^2)-\bar v\cdot\Div\mathring{R})\psi\dif x\dif\tau.$$
On the other hand, using that $\velocity,\bar{\velocity}\in L_{\sigma}^\infty$, an integration by parts yields
$$\eqref{Dfluctuation}
=-\int_0^t\int_{\Turzone(\tau)}(v-\bar v)\cdot\nabla(e+p)\psi\dif x\dif\tau.$$
This concludes the proof.
\end{proof}

Note that there is some ambiguity in Definition \ref{defi:subsolution} because the trace part of $R$ may be absorbed into the pressure; in particular we have
$$
R+\bar p\textrm{ Id}=\bar{R}+p\textrm{ Id},
$$
where $\bar R=\mathring{R}+(e-\tfrac12|\bar v|^2)\mathrm{ Id}$ (cf.~formulas \eqref{D:1}\eqref{def:e}). Nevertheless, the expression \eqref{eq:dissipation}, which does not depend on the specific choice of $\bar p$ and $\tr R$, is well-defined.\\

%From now on let us assume in this section that there is a pair $(\zcurve,\varpi)$ satisfying the conditions stated in Proposition~\ref{p:solvability}. 
Our aim now is to calculate the initial dissipation measure in terms of $(\zcurve_0,\varpi_0)$.
In particular, since $e>\tfrac{1}{2}|\bar{v}|^2+|\mathring{R}|$, having defined $(\bar{v},\bar{p})$ by the \emph{ansatz} \eqref{ansatzomega} and Bernoulli's law \eqref{Blaw} our next aim is to minimize $|\mathring{R}|$ at time $t=0$ among all solutions of the boundary value problem \eqref{R}. Recall our notation $R^\sharp(t,s,\lambda)=R(t,\map_{\lambda}(t,s))$.

\begin{prop}\label{prop:|R|} In general, for any uniformly bounded solution $R$ of \eqref{R}, 
$$
|\mathring{R}^{\sharp(0)}|
\geq |\dev{\B}^{(0)}\cdot\partial_s\zcurve_0^\perp|.
$$
Equality is attained if and only if $\tr R^{\sharp(0)}=2\dev{\B}^{(0)}\cdot\partial_s\zcurve_0$, which can be achieved if, in the setting of Proposition \ref{p:solvability}, we have in addition
\begin{enumerate}[(a)]
\addtocounter{enumi}{3}
\item\label{prop:|R|:1} $q_2^{(0)}=c\{\B\}^{(0)}\cdot\partial_sz_0$;	
\item\label{prop:|R|:2} $\left(q_1|\partial_sz|^2-\{\B\}\cdot\vectW\right)^{(1)}=0$.
\end{enumerate}
\end{prop}

\begin{proof}
\noindent{\bf Step 1.}
First of all we claim that $|\mathring{R}^\sharp|$ at time $t=0$ is given by the formula
\begin{equation}\label{|R|:1}
|\mathring{R}^{\sharp(0)}|=|\dev{\B}^{(0)}-\tfrac{1}{2}\tr R^{\sharp(0)}\partial_s\zcurve_0|.
\end{equation} 
Recalling Proposition \ref{R:solution} let us decompose $R$ as
$$\mathring{R}=\mathrm{M}(\tfrac{1}{2}(\partial_{11}-\partial_{22})g-i\partial_{12}g+\F),
\quad\quad\tr R=\Delta g.$$
%where
%$$\mathbf{h}:=\tfrac{1}{2}(\partial_{11}-\partial_{22})g+i\partial_{12}g.$$
Therefore $|\mathring{R}|=|\tfrac{1}{2}(\partial_{11}-\partial_{22})g+i\partial_{12}g+\F^*|$.
In particular, since
$$(\tfrac{1}{2}(\partial_{11}-\partial_{22})g+i\partial_{12}g)^\sharp\partial_s\map^*
=\partial_s(\nabla g)^\sharp-\tfrac{1}{2}(\Delta g)^\sharp\partial_s\map,$$
we have
\begin{equation}\label{|R|:2}
|\mathring{R}^\sharp||\partial_s\map|
=|\partial_s(\nabla g)^\sharp+(\F^\sharp\partial_s\map)^*-\tfrac{1}{2}\tr R^\sharp\partial_s\map|.
\end{equation}
Solving the linear system \eqref{eq:derivG} and applying $\partial_{\lambda}G^{(0)}=0$ we obtain
\begin{equation}\label{nablaf}
(\nabla g)^\sharp
=\frac{1}{\partial_s\map\cdot\partial_s\zcurve_{0}}\left(\partial_sG\partial_s\zcurve_0+c^{-1}(\partial_{\lambda}G)^{(1}\partial_s\map^\perp\right).
\end{equation}
Hence, recalling the decomposition $G=L+H$ in \eqref{ansatzL} and the proof of Proposition~\ref{p:solvability}, we deduce
\begin{equation*}
\begin{split}
(\nabla g)^{\sharp (0)}
&=\left((\partial_s G)^{(0)}+ic^{-1}(\partial_\lambda G)^{(1)}\right)\partial_s\zcurve_0\\
&=\left(\partial_s l_0^{(0)}+ic^{-1} l_1^{(1)}\right)\partial_s\zcurve_0\\
&=\left(\partial_s\langle G\rangle^{(0)}+ic^{-1}\langle\partial_\lambda G\rangle^{(1)}\right)\partial_s\zcurve_0\\
&=\int_0^s\mean{\A}^{(0)}\dif s_1 +\mean{o}^{(0)}.	
\end{split}	
\end{equation*}
Therefore, \eqref{|R|:1} follows from
$$
(\partial_s(\nabla g)^\sharp+(\F^\sharp\partial_s\map)^*)^{(0)}
=\dev{\B}^{(0)}.
$$

\bigskip

\noindent{\bf Step 2.}
Next, we derive the formula
$$
\tr R^{\sharp(0)}
=\dev{\B}^{(0)}\cdot\partial_s\zcurve_0+c^{-1}q_2^{(0)}
+c^{-2}(6l_3\lambda+\partial_\lambda^2 H)^{(2)}.
$$
Let us consider the equivalent matrix $S=S(t,s,\lambda)$
\begin{equation}
S:=Q(\nabla^2g)^\sharp Q,
\end{equation}
given by the orthogonal change of basis $Q:=\mathrm{M}(-\partial_s\zcurve_0^*)$ ($Q^2=\mathrm{ Id}$), that is,
$$(\nabla^2g)^\sharp=QSQ
\quad\textrm{and}\quad
\tr R=\tr S.$$
Thus, using
$Q\partial_s\zcurve_0=(1,0)$, $Q\partial_s\zcurve_0^\perp=(0,1)$, 
\eqref{D2g} yields
\begin{align*}
(\partial_s^2G)^{(0)}&=s_{11}^{(0)}+(\nabla g)^{\sharp (0)}\cdot\partial_s^2\zcurve_0,\\
(\partial_\lambda^2G)^{(2)}&=c^2s_{22}^{(0)}.
\end{align*}
Hence, since
$$(\partial_s^2G)^{(0)}
=\partial_s((\nabla g)^{\sharp (0)}\cdot\partial_s\zcurve_0)
=\mean{\A}^{(0)}\cdot\partial_s\zcurve_0+(\nabla g)^{\sharp (0)}\cdot\partial_s^2\zcurve_0,$$
we have
$$\tr S^{(0)}
=\mean{\A}^{(0)}\cdot\partial_s\zcurve_0+c^{-2}(\partial_\lambda^2G)^{(2)},$$
with $G=L+H$ and
$$(\partial_\lambda^2L)^{(2)}
=2l_2^{(2)}+6l_3^{(2)}\lambda.$$
On the one hand, by \eqref{b1:1} we get
$$2l_2^{(2)}
=\dev{\partial_{\lambda}G}^{(2)}
=\left(\int_0^s\dev{\A}^{(1)}\dif s_1
+\dev{o}^{(1)}\right)\cdot\vectW
=cq_2^{(0)}-c^2(\F^{(0)}\partial_s\zcurve_0)^*\cdot\partial_s\zcurve_0,$$
with $(\F^{(0)}\partial_s\zcurve_0)^*=(\dev{\B}-\mean{\A})^{(0)}$.
On the other hand, by \eqref{jumpslope:2} we get
$$
2l_3^{(2)}=(\mean{\partial_{\lambda}G}-\dev{G})^{(2)}
=\int_0^s(\dev{\B}\cdot\vectW-q_1|\partial_s\zcurve|^2-t\tilde{d})^{(1)}\dif s_1.
$$
Recalling the average condition \ref{p:solvability:average} in Proposition \ref{p:solvability}, then $l_3^{(2)}=0$ if and only if $\tilde{d}^{(1)}=0$ and the condition \ref{prop:|R|:2} above holds.
This concludes the proof by taking $H=0$.
\end{proof}

%Under ..., notice
%$$L(t,s,\lambda)=\sum_{k=0}^3l_k^{(k}(t,s)(t\lambda)^{k}$$

\begin{prop}\label{prop:D0} Let us assume that there is $(\zcurve,\varpi)$ satisfying the conditions \ref{Bq}-\ref{prop:|R|:2} stated in Propositions~\ref{p:solvability} and \ref{prop:|R|}. 
Then, the dissipation measure \eqref{eq:dissipation} satisfies $D\in\mathcal{M}_c([0,T]\times\R^2)$ with $\sop D\subset\bar{\Omega}_{\mathrm{tur}}$. Moreover,
given $0<\varepsilon\leq\ell$ and $N=1$, \eqref{local:1} holds.
\end{prop}
\begin{proof} Recalling Proposition \ref{prop:dissipation} we have
\begin{subequations}
\label{dissipation:1}
\begin{align}
\left\langle\frac{D(t_2)-D(t_1)}{t_2-t_1},\psi\right\rangle
&=\sum_{\lambda=\pm 1}\dashint_{t_1}^{t_2}\int ([e]_\lambda\partial_t\map_\lambda+[(e+p)\bar{v}]_\lambda)\cdot\partial_s\map_\lambda^\perp\psi_\lambda\dif s\dif\tau\label{dissipation:1:1}\\
&-\dashint_{t_1}^{t_2}\int_{\Turzone(\tau)}\left(\partial_t(e-\tfrac{1}{2}|\bar{\velocity}|^2)-\bar v\cdot\Div\mathring{R}\right)\psi\dif x\dif\tau\label{dissipation:1:2}\\
&-\dashint_{t_1}^{t_2}\int_{\Turzone(\tau)}(v-\bar v)\cdot\nabla(e+p)\psi\dif x\dif\tau.\label{dissipation:1:3}
\end{align}
\end{subequations}
\indent Concerning \eqref{dissipation:1:2}, by applying $e-\tfrac{1}{2}|\bar{\velocity}|^2=|\mathring{R}|+e'$ \eqref{def:e} and $\Div R=0$ \eqref{R:1}, we deduce
$$\eqref{dissipation:1:2}
=-\dashint_{t_1}^{t_2}\int_{\Turzone(\tau)}\left(\partial_t(|\mathring{R}|+e')+\tfrac{1}{2}\bar v\cdot\nabla\tr R\right)\psi\dif x\dif\tau.$$
Hence, by Lemma~\ref{lemma:vRcontrol} below and imposing $\partial_te'\in L^\infty(\Turzone)$,
$$\eqref{dissipation:1:2}
\leq \mathcal{O}(t_2)\norma{\psi}{\infty}.$$
\indent Concerning \eqref{dissipation:1:3}, we can guarantee that (see \cite{Onadmissibility,Degraded}), for any fixed $0<\varepsilon\leq \ell$ and time-error function $\mathcal{T}\in C(]0,T];]0,1])$, these weak solutions satisfy
$$\left|\int_{\R^2}(\velocity-\bar{\velocity})\cdot\nabla(e+p)\psi_I\dif x\right|
\leq \mathcal{T}(t),$$
for every interval $I\subset\T$ with $|I|\geq\varepsilon$ and $t\in]0,T]$. \\
\indent Concerning \eqref{dissipation:1:1}, notice that
$$\left[(e+p)\bar{\velocity}\right]_{\lambda}\cdot\partial_s\map_{\lambda}^\perp
=(e+p)_{\lambda}^+\bar{\velocity}_{\lambda}^+\cdot\partial_s\map_{\lambda}^\perp
-(e+p)_{\lambda}^-\bar{\velocity}_{\lambda}^-\cdot\partial_s\map_{\lambda}^\perp
=[e+p]_{\lambda}\BRO_{\lambda}\cdot\partial_s\map_{\lambda}^\perp.$$
Hence, since (recall \eqref{eq:potjump}-\eqref{jump|u|})
\begin{align*}
[e]_{\lambda}
&=-\tfrac{1}{2}\varpi\Re_{\partial_s\map_{\lambda}}(\BRO_{\lambda})-\lambda(|\mathring{R}|+e')_{\lambda},\\
[e+p]_{\lambda}
&=-[\partial_t\phi]_{\lambda}-\tfrac{\lambda}{2}\tr R_{\lambda},
\end{align*}
we deduce
\begin{equation}
\eqref{dissipation:1:1}
=-\sum_{\lambda=\pm 1}\lambda\dashint_{t_1}^{t_2}\int_\T
(\BRO\cdot\mathring{R}\partial_s\map^\perp
+(|\mathring{R}|+e')\partial_t\map\cdot\partial_s\map^\perp)_{\lambda}\psi_{\lambda}\dif s\dif\tau.
\end{equation}
On the one hand, as we shall in Corollary \ref{cor:BRO0} below,
\begin{equation}\label{BRL:0}
\BRO_{\lambda}^{(0)}
=\BRO_0-\tfrac{\lambda}{4}\varpi_0\partial_s\zcurve_0,
\end{equation}
with $\BRO_{0}$ given in \eqref{BRO0}. Hence
\begin{equation}\label{BRL:0.1}
\mean{\BRO}^{(0)}=\BRO_0,
\quad\quad\dev{\BRO}^{(0)}=-\tfrac{1}{4}\varpi_0\partial_s\zcurve_0,
\end{equation}
and so
\begin{equation}\label{b:0}
\pm\B_{\pm}^{(0)}
=\dev{\B}^{(0)}
=-\tfrac{1}{2}\varpi_0(\vectW-\dev{\BRO}^{(0)})
=-\tfrac{1}{2}\varpi_0\left(ic+\tfrac{1}{4}\varpi_0\right)\partial_s\zcurve_0.
\end{equation}
Then, by splitting $\mathring{R}=R-\tfrac{1}{2}(\tr R)\mathrm{Id}$ and applying Proposition~\ref{prop:|R|}, we get
$$(\mathring{R}\partial_s\map^\perp)_{\lambda}^{(0)}
=i(\dev{\B}^{(0)}
-\Re_{\partial_s\zcurve_0}\dev{\B}^{(0)}\partial_s\zcurve_0)
=-\Im_{\partial_s\zcurve_0}\dev{\B}^{(0)}\partial_s\zcurve_0
=\tfrac{1}{2}c\varpi_0\partial_s\zcurve_0.$$
On the other hand, as we shall see in Section \ref{sec:proof:1}, Proposition \ref{p:solvability} requires $(\partial_{t}\zcurve\cdot\partial_s\zcurve^\perp)^{(0)}=\BRO_{0}\cdot\partial_s\zcurve_0^\perp$.
Thus, by imposing ${e}'=\mathcal{O}(t)$, we obtain
\begin{align*}
&(\BRO\cdot\mathring{R}\partial_s\map^\perp
+(|\mathring{R}|+e')\partial_t\map\cdot\partial_s\map^\perp)_{\lambda}^{(0)}\\
&=
\left(\BRO_0-\tfrac{\lambda}{4}\varpi_0\partial_s\zcurve_0\right)\cdot\left(\tfrac{1}{2}c\varpi_0\partial_s\zcurve_0\right)+\tfrac{1}{2}c|\varpi_0|(\BRO_{0}\cdot\partial_s\zcurve_0^\perp+\lambda c)\\
&=\tfrac{1}{2}c|\varpi_0|\left(\lambda\left(c-\tfrac{1}{4}|\varpi_0|\right)+B_0\right),
\end{align*}
where $B_0:=\BRO_0\cdot((\sgn\varpi_0+i)\partial_s\zcurve_0)$. This concludes the proof.
\end{proof}

\begin{lemma}\label{lemma:Wmax}
The functional $W_{\T}^{(N)}$ has a global maximum at $c_{\max}^{(N)}=\tfrac{1}{2}\bar{c}_N|\varpi_0|$.
\end{lemma}
\begin{proof}
Notice that $W\equiv W_{\T}^{(N)}$ satisfies
$$W(c+\psi)-W(c)
=dW_c(\psi)
-\bar{a}_N\int|\varpi_0|\psi^2,$$
where $dW_c$ is the Fr\'echet derivative of $W$ at $c$
$$dW_c(\psi)=\bar{a}_N\int|\varpi_0|(\bar{c}_N|\varpi_0|-2c)\psi.$$
This concludes the proof.
\end{proof}

\begin{lemma}\label{lemma:T3} Let $c$ given in \eqref{c}. Then, there is  $\epsilon(\varpi_0,\varepsilon,\delta)>0$ so that
$$W_I(c)
\geq\frac{1}{2}\delta(1-\delta)\bar{a}_N\bar{c}_N^2\dashint_{I}|\varpi_0|^3\dif s,$$ 
for every interval $I\subset\T$ with $|I|\geq\varepsilon$.
\end{lemma}
\begin{proof}
Let $c=\delta\bar{c}_Nf_{\epsilon}$ where we have abbreviated $f\equiv|\varpi_0|$ and $f_{\epsilon}\equiv f*\eta_\epsilon$. We split
\begin{align*}
&\dashint_{I}c|\varpi_0|(\bar{c}_N|\varpi_0|-c)\dif s
=\frac{\delta\bar{c}_N^2}{|I|}\int_{I}f_\epsilon f(f-\delta f_\epsilon)\dif s\\
&=\frac{\delta\bar{c}_N^2}{|I|}\left((1-\delta)\int_{I}f^3\dif s
+(1-\delta)\int_{I}(f_\epsilon-f)f^2\dif s
+\delta\int_{I}f_\epsilon f(f-f_\epsilon)\dif s\right)\\
&\geq\frac{\delta\bar{c}_N^2}{|I|}\left((1-\delta)\int_{I}f^3\dif s-\norma{f}{L^\infty(I)}^2\norma{f_\epsilon-f}{L^1(I)}\right).
\end{align*}
Finally, since
$$s_0\in\T\mapsto\int_{[s_0,s_0+\varepsilon]}f^3\dif s>0$$
is continuous, the lemma holds.
\end{proof}

\begin{lemma}\label{lemma:vRcontrol} The $L^\infty(\Turzone)$-norm of $\nabla_{(t,x)}\tr R$, $\nabla_{(t,x)}|\mathring{R}|$ and $\nabla(e+p)$ are controlled by $\norma{\zcurve}{2,\alpha}$ and $\norma{\varpi}{1,\alpha}$.
\end{lemma}
\begin{proof}
Concerning the first two terms $h=\tr R$ and $|\mathring{R}|$, since $\partial_{\lambda}h^{\sharp(0)}=0$ by Proposition \ref{prop:|R|}, we apply \eqref{nablaf} for $h$ (instead of $g$) to deduce that $\nabla h$ is bounded.\\
Hence, since
$\partial_th^\sharp
=(\partial_th)^\sharp+(\nabla h)^\sharp\cdot\partial_t\map$ 
and, by \eqref{|R|:2},
$$\partial_t|\mathring{R}|^\sharp
=\partial_t\left(\frac{|\partial_s(\nabla g)^\sharp+(\F^\sharp\partial_s\map)^*-\tfrac{1}{2}\tr R^\sharp\partial_s\map|}{|\partial_s\map|}\right),$$
the statement follows by using the (a.e.) inequality $|\partial_t|F||\leq|\partial_tF|$.\\
\indent For $\nabla(e+p)$,
recalling \eqref{phit}\eqref{Blaw} and \eqref{D:1}, it is enough to control
$$
\nabla(\tfrac12|\bar{\velocity}|^2+\bar{p})(t,x)^*
=\frac{1}{2}\sum_{\lambda=\pm 1}\frac{1}{2\pi i}\int_{\T}\frac{(\partial_{t}\tilde{\varpi}\partial_s\map_{\lambda}-\varpi\partial_t\map_{\lambda})(t,s)}{(x-\map_{\lambda}(t,s))^2}\dif s,
\quad x\in\Turzone(t).
$$	
Let us abbreviate here $f_{\lambda}\equiv(\partial_{t}\tilde{\varpi}\partial_s\map_\lambda-\varpi\partial_t\map_\lambda)$. Then, the Cauchy's integral formula yields
\begin{align*}
\int_{\T}\frac{f_{\lambda}(t,s+\xi)}{(x-\map_{\lambda}(t,s+\xi))^2}\dif\xi
&=\int_{\T}\frac{f_{\lambda}(t,s+\xi)-f_{\lambda}(t,s)}{(x-\map_{\lambda}(t,s+\xi))^2}\dif\xi\\
&-\frac{f_{\lambda}(t,s)}{\partial_s\map_{\lambda}(t,s)}\int_{\T}\frac{\partial_s\map_{\lambda}(t,s+\xi)-\partial_s\map_{\lambda}(t,s)}{(x-\map_{\lambda}(t,s+\xi))^2}\dif\xi.
\end{align*}
By adding and subtracting $\xi\partial_sf_{\lambda}(t,s)$ and also $\xi\partial_s^2\map_{\lambda}(t,s)$ we gain integrability (cf.~Lemma \ref{geometriclemma} below), while the remainder integral
$$\int_{\T}\frac{\xi}{(x-\map_{\lambda}(t,s+\xi))^2}\dif\xi$$
can be bounded easily by following the proof of Lemma \ref{lemma:index} below.
\end{proof}

\section{The Birkhoff-Rott operator}\label{sec:BRO}

In light of Proposition \ref{prop:D0}, there exists a subsolution adapted to $\Turzone$ provided $(\zcurve,\varpi)$ satisfies the conditions \ref{Bq}-\ref{prop:|R|:2} stated in Propositions~\ref{p:solvability} and \ref{prop:|R|}. In Section \ref{sec:proof} we construct such a pair $(\zcurve,\varpi)$ by setting its Taylor polynomial $(\zcurve,\varpi)^{n_0)}$ satisfying the pointwise conditions \ref{p:solvability:pointwise}-\ref{prop:|R|:2}, and the corresponding remainder $(\zcurve,\varpi)^{(n_0+1}$ satisfying the average conditions \ref{Bq}-\ref{p:solvability:average}. This requires to check that the Taylor decompositions $\BRO_{\lambda}=\BRO_{\lambda}^{n_0-1)}+t^{n_0}\BRO_{\lambda}^{(n_0}$ of the Birkhoff-Rott operators $\BRO_{\lambda}$ \eqref{BRO:0} are well-defined. This is the goal of this section. Although $n_0=3$ in Section \ref{sec:proof}, let us keep it general here.

\subsection{Geometric setup} Since $\zcurve_0$ is closed and $C^1$, it is simple and regular \eqref{z0} if and only if its \textbf{chord-arc} constant is bounded
\begin{equation}\label{CAC}
\CA{\zcurve_{0}}:=\sup_{s,\xi}\left|\triangle_\xi\zcurve_{0}(s)\right|^{-1}<\infty.
\end{equation}
In particular, $\CA{\zcurve_{0}}\geq|\partial_s\zcurve_0|^{-1}=1$ by the arc-length parametrization.\\  
\indent The chord-arc condition \eqref{CAC} is usually imposed when considering Birkhoff-Rott type operators (cf.~\cite[\S 1.1]{Wu06}) because it avoids self-intersections (simple) and peaks (regular). Moreover, it
gives a lower bound of the proximity of different points at $\zcurve_0$: $|\zcurve_0(s)-\zcurve_{0}(s')|\geq\CA{\zcurve_{0}}^{-1}|s-s'|$, thus measuring the singularity in $\BRO_{\lambda}$ at time $t=0$. For $t>0$, $\BRO_{\lambda}$ requires to compare different points at the boundary of the turbulence zone:
\begin{lemma}\label{geometriclemma} 
%Let $\vectV=ct\partial_s\zcurve_{0}^\perp$ \eqref{ansatz:vectV}. 
There exists $t_0(\norma{\zcurve}{1,\alpha},\norma{c}{1,\alpha},\CA{\zcurve_{0}})>0$ such that, for all $0\leq t\leq t_0$, the following ``equi chord-arc'' condition holds:
$$|\map_{\mu}(t,s+\xi)-\map_{\lambda}(t,s)|
\geq \tfrac{1}{4}\sqrt{\CA{\zcurve_{0}}^{-2}|\xi|^2+t^2c(s)^2(\mu-\lambda)^2},$$
for every $s,\xi\in\T$ and $\lambda,\mu\in\{-1,1\}$.
\end{lemma}
\begin{proof}
The case $t=0$ follows from \eqref{CAC}. Now let $0<t\leq t_0$ with $t_0$ to be determined, and let us fix some parameters $\beta,\gamma,\delta>0$ with $\tfrac{1}{4}=\gamma<\beta<\tfrac{1}{2}$ and
$$\delta=1-\left(1-\left(\frac{\gamma}{\beta}\right)^2\right)^2<1.$$
In particular, we can take $t_0$ so that
\begin{align}
|\zcurve(t,s+\xi)-\zcurve(t,s)|
&\geq|\zcurve_0(s+\xi)-\zcurve_0(s)|-t|\zcurve^{(1}(t,s+\xi)-\zcurve^{(1}(t,s)|
\geq 2\beta\CA{\zcurve_0}^{-1}|\xi|,\label{geometriclemma:2}\\
|\map_{\mu}(t,s+\xi)-\map_{\mu}(t,s)|
&\geq|\zcurve(t,s+\xi)-\zcurve(t,s)|-t|\vectW(s+\xi)-\vectW(s)|
\geq\beta\CA{\zcurve_0}^{-1}|\xi|.\label{geometriclemma:3}
\end{align}
We split the proof in two cases, depending on some $r>0$ to be determined.\\
\indent\textbf{Case $|\xi|\leq r$}: By writing
\begin{equation}\label{goodbad}
\map_{\mu}(t,s+\xi)-\map_{\lambda}(t,s)
=\underbrace{\map_{\mu}(t,s+\xi)-\map_{\mu}(t,s)}_{\equiv I}+(\mu-\lambda)\vectV(t,s),
\end{equation}
we split
\begin{equation}\label{geometriclemma:1}
|\map_{\mu}(t,s+\xi)-\map_{\lambda}(t,s)|^2
=|I|^2
+(\mu-\lambda)^2|\vectV(t,s)|^2
+2(\mu-\lambda)I\cdot\vectV(t,s).
\end{equation}
Let us analyse the third term. Since
$$
\left|\frac{\partial_s\map_{\mu}^\perp}{|\partial_s\map_{\mu}|}\cdot\frac{\vectV}{|\vectV|}\right|
=\left|\frac{\partial_s\map_{\mu}}{|\partial_s\map_{\mu}|}\cdot\partial_s\zcurve_{0}\right|
\geq 1-\mathcal{O}(t),
$$
and
$$
\left|\frac{I}{|I|}-\frac{\partial_s\map_{\mu}(t,s)}{|\partial_s\map_{\mu}(t,s)|}\right|
\leq \frac{2|\triangle_{\xi}^2\map_{\mu}(t,s)|}{|\partial_s\map_{\mu}(t,s)|}|\xi|
\leq \frac{2\snorma{\partial_s\map_{\mu}}{\alpha}}{|\partial_s\map_{\mu}(t,s)|}|\xi|^\alpha,
$$
we have
$$
\left|\frac{I\cdot\vectV^\perp}{|I||\vectV|}\right|
\geq\left|\frac{\partial_s\map_{\mu}^\perp}{|\partial_s\map_{\mu}|}\cdot\frac{\vectV}{|\vectV|}\right|-\left|\frac{I}{|I|}-\frac{\partial_s\map_{\mu}}{|\partial_s\map_{\mu}|}\right|
\geq 1-\mathcal{O}(t+|\xi|^\alpha).
$$
Thus, we can take $t_0,r$ so that
$|I\cdot\vectV^\perp|^2
\geq\delta|I|^2|\vectV|^2$,
and so
$$|I\cdot\vectV|^2
=|I|^2|\vectV|^2-|I\cdot\vectV^\perp|^2
\leq(1-\delta)|I|^2|\vectV|^2,$$
from which we get
$$2|\mu-\lambda||I\cdot\vectV|\leq\sqrt{1-\delta}(|I|^2+(\mu-\lambda)^2|\vectV|^2).$$
Hence, by applying this and \eqref{geometriclemma:3} into \eqref{geometriclemma:1}, we deduce
\begin{align*}
|\map_{\mu}(t,s+\xi)-\map_{\lambda}(t,s)|
&\geq\sqrt{1-\sqrt{1-\delta}}
\sqrt{|I|^2+(\mu-\lambda)^2|\vectV(t,s)|^2}\\
&\geq \underbrace{\beta\sqrt{1-\sqrt{1-\delta}}}_{=\gamma}
\sqrt{\CA{\zcurve_0}^{-2}|\xi|^2+(\mu-\lambda)^2|\vectV(t,s)|^2}.
\end{align*}
\indent\textbf{Case $|\xi|> r$}: Since \eqref{geometriclemma:2} yields
$$|\map_{\mu}(t,s+\xi)-\map_{\lambda}(t,s)|
\geq|\zcurve(t,s+\xi)-\zcurve(t,s)|-2t\norma{c}{\infty}
\geq 2(\beta\CA{\zcurve_0}^{-1}|\xi|-t\norma{c}{\infty}),$$
and
$$\CA{\zcurve_0}^{-2}|\xi|^2+t^2c(s)^2(\mu-\lambda)^2
\leq\CA{\zcurve_0}^{-2}|\xi|^2+4 t^2\norma{c}{\infty}^2,$$
it is enough to guarantee that
$$ 4(\beta\CA{\zcurve_0}^{-1}|\xi|-t\norma{c}{\infty})^2\geq
\gamma^2(\CA{\zcurve_0}^{-2}|\xi|^2+4t^2\norma{c}{\infty}^2).$$
This holds if and only if
$$(4\beta^2-\gamma^2)\CA{\zcurve_0}^{-2}|\xi|^2
\geq 4t\norma{c}{\infty}(2\beta\CA{\zcurve_0}^{-1}|\xi|
-t(1-\gamma^2)\norma{c}{\infty}),$$
which is satisfied by taking $t_0$ small enough.
\end{proof}

\subsection{Alternative expressions for $\bar{\velocity}$ and $\BRO_{\lambda}$} 
Let us abbreviate
$$
\zeta_{\lambda}:=\frac{\varpi_{\lambda}}{\partial_s\map_{\lambda}},
$$
and also $\zeta_0\equiv\zeta_{\lambda}^{(0)}=\frac{\varpi_0}{\partial_s\zcurve_0}$.
Given $x\notin\Gamma(t)$, observe that (recall Def.~\ref{defi:circulation})
$$\bar{\velocity}(t,x)^*=\frac{1}{2}\sum_{\mu=\pm 1}\left(\frac{1}{2\pi i}\int_{\T}\frac{\varpi_{\mu}(t,s')-w_{\mu}(t,x)\partial_s\map_{\mu}(t,s')}{x-\map_{\mu}(t,s')}\dif s'-w_{\mu}(t,x)\Ind_{\map_{\mu}(t)}(x)\right),$$
for any $w_{\mu}\in\C$. Setting $w_\mu=\zeta_{\mu}(t,s)$ we therefore deduce
\begin{equation}\label{v:Ind}
\begin{split}	
\bar{\velocity}(t,x)^*&=\frac{1}{2}\sum_{\mu=\pm 1}\bigg(\frac{\zeta_{\mu}(t,s)}{2\pi i}\int_{\T}\frac{\partial_s\map_{\mu}(t,s)-\partial_s\map_{\mu}(t,s')}{x-\map_{\mu}(t,s')}\dif s'\\
&-\frac{1}{2\pi i}\int_{\T}\frac{\varpi_{\mu}(t,s)-\varpi_{\mu}(t,s')}{x-\map_{\mu}(t,s')}\dif s'
-\zeta_{\mu}(t,s)\Ind_{\map_{\mu}(t)}(x)\bigg)
\end{split}
\end{equation}
valid for any $t\geq 0$, $x\notin\Gamma(t)$ and $s\in\T$.

In particular, for $t=0$ we have
$$\velocity_0(x)^*
=\frac{\zeta_0(s)}{2\pi i}\int_{\T}\frac{\partial_s\zcurve_0(s)-\partial_s\zcurve_0(s')}{x-\zcurve_0(s')}\dif s'
-\frac{1}{2\pi i}\int_{\T}\frac{\varpi_0(s)-\varpi_0(s')}{x-\zcurve_0(s')}\dif s'
-\zeta_0(s)\Ind_{\zcurve_0}(x).$$
By considering nontangential limits $x=x_{\pm\varepsilon}=\zcurve_0(s)\pm \varepsilon\partial_s\zcurve_0^\perp(s)$ as $\varepsilon\downarrow 0$ we deduce that $\velocity_0$ is bounded in a neighbourhood of $\Gamma_0$, and consequently $\velocity_0=\mathcal{O}((1+|x|)^{-1})$ in terms of $\norma{\zcurve_0}{1,\alpha}$, $\norma{\varpi_0}{0,\alpha}$ and $\CA{\zcurve_{0}}$. Furthermore,
$$\velocity_0^{\pm}
=\BRO_0\mp\tfrac{1}{2}\zeta_0^*,$$
where $\BRO_0\equiv\BRO(\zcurve_0,\varpi_0)$ is
\begin{align}
\BRO_0(s)^*
:=&\,\frac{1}{2\pi i}\PV\!\!\int_{\T}\frac{\varpi_0(s')}{\zcurve_0(s)-\zcurve_0(s')}\dif s'\nonumber\\
=&\,\frac{\zeta_0(s)}{2\pi i}\int_{\T}\frac{\partial_s\zcurve_0(s')-\partial_s\zcurve_0(s)}{\zcurve_0(s')-\zcurve_0(s)}\dif s'
-\frac{1}{2\pi i}\int_{\T}\frac{\varpi_0(s')-\varpi_0(s)}{\zcurve_0(s')-\zcurve_0(s)}\dif s'
-\frac{1}{2}\zeta_0(s).\label{BRO0}
\end{align}

Similarly, for $t>0$ we consider $x=x_{\lambda,\pm\varepsilon}=\map_{\lambda}(t,s)\pm\varepsilon\partial_s\map_{\lambda}^\perp(t,s)$ to deduce that $\bar{\velocity}(t)$ is bounded in a neighbourhood of $\Gamma(t)$, and consequently $\bar{\velocity}(t)=\mathcal{O}((1+|x|)^{-1})$ in terms of $\norma{\zcurve(t)}{1,\alpha}$, $\norma{\varpi(t)}{0,\alpha}$ and $\CA{\zcurve_{0}}$. Furthermore, 
\begin{equation}\label{e:vpm}
\bar{\velocity}_{\lambda}^{\pm}
=\BRO_{\lambda}\mp\tfrac{1}{4}\zeta_{\lambda}^{*},
\end{equation}
with
\begin{align*}
&\BRO_{\lambda}(t,s)^*
=\frac{1}{2}\sum_{\mu=\pm 1}\frac{1}{2\pi i}\PV\!\int_{\T}\frac{\varpi_{\mu}(t,s')}{\map_{\lambda}(t,s)-\map_{\mu}(t,s')}\dif s'\\
&=\frac{1}{2}\sum_{\mu=\pm 1}\left(\frac{\zeta_{\mu}(t,s)}{2\pi i}\int_{\T}\frac{\partial_s\map_{\mu}(t,s')-\partial_s\map_{\mu}(t,s)}{\map_{\mu}(t,s')-\map_{\lambda}(t,s)}\dif s'
-\frac{1}{2\pi i}\int_{\T}\frac{\varpi_{\mu}(t,s')-\varpi_{\mu}(t,s)}{\map_{\mu}(t,s')-\map_{\lambda}(t,s)}\dif s'
-\theta_{\lambda,\mu}\zeta_{\mu}(t,s)\right),
\end{align*}
where, for $\lambda,\mu\in\{-1,1\}$, we have
$$
\theta_{\lambda,\mu}
:=\frac{1+\sgn(\lambda-\mu)}{2}=\begin{cases}\frac12&\lambda=\mu,\\ 1& (\lambda,\mu)=(1,-1),\\ 0&(\lambda,\mu)=(-1,1).\end{cases}
$$

\bigskip

Analogously to \cite[\S3]{Piecewise}, we consider the  weighted Hilbert transform on $C^{0,\alpha}(\T)$: 
\begin{equation}\label{e:weightedH}
\TT{\Phi}f(s):=\frac{1}{2\pi i}\int_{\T}\triangle_\xi f(s)\Phi(s,\xi)\dif\xi,
\end{equation}
for the space of weights
$L^{\infty}(\T;C^{k,\alpha}(\T))$
with the norm
$$
\Norma{\Phi}{k,\alpha}:=\underset{\xi\in\T}{\mathrm{ess\,sup}}\norma{\Phi(\cdot,\xi)}{k,\alpha}.
$$
Recall that $\T=\R/\ell\Z$. Then, w.l.o.g.~in the following we will use the parameter range $\xi\in [-\ell/2,\ell/2]$ in \eqref{e:weightedH}. For ease of notation let us assume that $\ell=1$. We start with the following simple result for boundedness of $\TT{\Phi}$ on H\"older spaces, following \cite{Piecewise}: 

\begin{thm}\label{thm:generalkernel} For all $0\leq \alpha'<\alpha $ and $k\in\N_0$ there exists $C(k,\alpha)>0$ such that
$$\norma{\TT{\Phi}f}{k,\alpha'}\leq \frac{C}{\alpha-\alpha'}\Norma{\Phi}{k,\alpha'}\norma{f}{k,\alpha},$$
for every $\Phi\in L^{\infty}(\T;C^{k,\alpha'}(\T))$ and $f\in C^{k,\alpha}(\T)$.
\end{thm}

\begin{proof}
Let us start with $k=0$. On the one hand,
$$
\norma{\TT{\Phi}f}{\infty}\leq\frac{2^{-\alpha}}{\pi\alpha}\underset{\xi\in\T}{\mathrm{ess\,sup}}\norma{\Phi(\cdot,\xi)}{\infty}\snorma{f}{\alpha}.
$$
On the other hand, we split
\begin{align*}
\TT{\Phi}f(s)-\TT{\Phi}f(s')&=
\frac{1}{2\pi i}\int_{\T}\triangle_\xi f(s)\left(\Phi(s,\xi)-\Phi(s',\xi)\right)\dif\xi &=:I\\
&+\frac{1}{2\pi i}\int_{\T}\left(\triangle_\xi f(s)-\triangle_\xi f(s')\right)\Phi(s',\xi)\dif\xi &=:J.
\end{align*}
For $I$, simply
$$
|I|\leq\frac{2^{-\alpha}}{\pi\alpha}\underset{\xi\in\T}{\mathrm{ess\,sup}}\snorma{\Phi(\cdot,\xi)}{\alpha'}\snorma{f}{\alpha}|s-s'|^{\alpha'}.
$$
For $J$, consider $\beta=\frac{\alpha'}{\alpha}\in(0,1)$. By applying $|a+b|^\beta\leq 2^\beta(|a|^\beta+|b|^\beta)$ we obtain
\begin{align*}
&\left|\triangle_\xi f(s)-\triangle_\xi f(s')\right|
=\left|\triangle_\xi f(s)-\triangle_\xi f(s')\right|^{1-\beta}
\left|\triangle_\xi f(s)-\triangle_\xi f(s')\right|^\beta\\
&\leq 
2^{1-\beta}\left(\left|\triangle_\xi f(s)\right|^{1-\beta}+\left|\triangle_\xi f(s')\right|^{1-\beta}\right)2^\beta|\xi|^{-\beta}
\left(\left|f(s)-f(s')\right|^\beta
+\left|f(s+\xi)-f(s'+\xi)\right|^\beta\right)\\
&\leq 8\snorma{f}{\alpha}|\xi|^{(1-\beta)\alpha-1}|s-s'|^{\beta\alpha},
\end{align*}
and consequently
$$
|J|\leq\frac{2^{3-(1-\beta)\alpha}}{\pi(1-\beta)\alpha}
\underset{\xi\in\T}{\mathrm{ess\,sup}}\norma{\Phi(\cdot,\xi)}{\infty}
\snorma{f}{\alpha}|s-s'|^{\beta\alpha}.
$$
For $k\geq 1$ the result follows by applying the Leibniz rule.
\end{proof}
\begin{Rem} The operator $\TT{\Phi}$ is bounded for $\alpha'=\alpha$ by imposing additional conditions on $\Phi$ (see \cite[\S3]{Piecewise}) but we have presented this version for simplicity.
\end{Rem}

In particular, we will focus on the following kernels
\begin{align*}
\Phi_{\lambda,\mu}(t,s,\xi)
&:=\frac{\xi}{\map_{\mu}(t,s+\xi)-\map_{\lambda}(t,s)},\\
\Phi_{\lambda,\mu,n_0}(t,s,\xi)
&:=\frac{\xi}{(\map_{\mu}(t,s+\xi)-\map_{\lambda}(t,s))^{n_0-1)}},
 \end{align*}
with $\Phi_0\equiv\Phi_{\lambda,\mu}^{(0)}$ and the remainder
\begin{align*}
\Theta_{\lambda,\mu,n_0}(t,s,\xi)
&:=t^{-n_0}(\Phi_{\lambda,\mu}-\Phi_{\lambda,\mu,n_0})(t,s,\xi)\\
&=-\frac{\xi(\map_{\mu}(t,s+\xi)-\map_{\lambda}(t,s))^{(n_0}}{(\map_{\mu}(t,s+\xi)-\map_{\lambda}(t,s))(\map_{\mu}(t,s+\xi)-\map_{\lambda}(t,s))^{n_0-1)}}.
\end{align*}
With them we can express the Birkhoff-Rott operators $\BRO_{\lambda}$ as follows. At time $t=0$ we have
\begin{equation}
\BRO_0^*=\zeta_0\TT{\Phi_0}(\partial_s\zcurve_0)-\TT{\Phi_0}(\varpi_0)-\tfrac{1}{2}\zeta_0,
\end{equation}
and for $t>0$ we have
\begin{equation}\label{e:BROexpand1}
\BRO_{\lambda}^{*}
=\frac{1}{2}\sum_{\mu=\pm 1}(\zeta_{\mu}\TT{\Phi_{\lambda,\mu}}(\partial_s\map_{\mu})-\TT{\Phi_{\lambda,\mu}}(\varpi_{\mu})-\theta_{\lambda,\mu}\zeta_{\mu}).
\end{equation}
Notice that Lemma \ref{geometriclemma} implies
\begin{equation}\label{eq:Phi0}
\norma{\Phi_{\lambda,\mu}}{\infty}
\leq\frac{4|\xi|}{\sqrt{\CA{\zcurve_{0}}^{-2}|\xi|^2+t^2c(s)^2(\mu-\lambda)^2}}
\leq 4\CA{\zcurve_0},
\end{equation}
and, recalling \eqref{goodbad}, also
\begin{align}
|\Phi_{\lambda,\mu}(t,s,\xi)-\Phi_{\lambda,\mu}(t,s',\xi)|
&=\left|\frac{\xi(\xi\triangle_\xi(\map_{\mu}(t,s)-\map_{\mu}(t,s'))
+(\mu-\lambda)(\vectV(t,s)-\vectV(t,s')))}{(\map_{\mu}(s+\xi,t)-\map_{\lambda}(s,t))(\map_{\mu}(s'+\xi,t)-\map_{\lambda}(s',t))}\right|\nonumber\\
&\leq 4^2\CA{\zcurve_0}(\CA{\zcurve_0}\snorma{\partial_s\map_{\mu}}{\alpha}+|\mu-\lambda|c(s)^{-1}\snorma{\vectW}{\alpha})|s-s'|^{\alpha}.\label{eq:Psi1}
\end{align}
Therefore, $\Norma{\Phi_{\lambda,\mu}}{0,\alpha}$ is bounded. Similar estimates apply to $\Norma{\Phi_{\lambda,\mu,n_0}}{0,\alpha}$ and $\Norma{\Theta_{\lambda,\mu,n_0}}{0,\alpha}$.
In light of Theorem \ref{thm:generalkernel} this proves:

\begin{cor}\label{Phi:bdd}
For any fixed $0\leq\alpha'<\alpha<1$ and any $\lambda,\mu\in\{-1,1\}$, the operators $\TT{\Phi_{\lambda,\mu}}$, $\TT{\Phi_{\lambda,\mu,n_0}}$, $\TT{\Theta_{\lambda,\mu,n_0}}$ are bounded operators from $C^{0,\alpha}$ to $C^{0,\alpha'}$, with operator norm bounded in terms of $\CA{\zcurve_{0}}$, $\norma{1/c}{\infty}$, the $C^{0,\alpha}$-norm of $\varpi$ and the $C^{1,\alpha}$-norm of $\zcurve$ and $\vectW$. 
\end{cor} 

Based on this corollary we can consider expansions in $t$ as follows. 
By considering the Taylor decompositions $f=f^{n_0-1)}+t^{n_0}f^{(n_0}$ of $f=\partial_s\map_{\mu},\varpi_{\mu}$, we split
\begin{align*}
\TT{\Phi_{\lambda,\mu}}f&=\TT{\Phi_{\lambda,\mu,n_0}}f^{n_0-1)}
+t^{n_0}(\TT{\Phi_{\lambda,\mu,n_0}}f^{(n_0}+\TT{\Theta_{\lambda,\mu,n_0}}f)\\
&=\TT{\Phi_{\lambda,\mu,n_0}}f^{n_0-1)}
+\mathcal{O}(t^{n_0}),
\end{align*}
where the $\mathcal{O}(t^{n_0})$ is understood in the $C^{0,\alpha'}$-norm. 
Applying this to \eqref{e:BROexpand1} we obtain
for any $0\leq n\leq n_0-1$
\begin{align*}
\BRO_{\lambda}^{*(n)}
&=\frac{1}{2}\sum_{\mu=\pm 1}\bigg(\sum_{n_1+n_2+n_3=n}\zeta_{\lambda}^{(n_3)}(\TT{\Phi_{\lambda,\mu,n_0}}(\partial_s\map_{\mu}^{(n_2)}))^{(n_1)}\\
&-\sum_{n_1+n_2=n}(\TT{\Phi_{\lambda,\mu,n_0}}(\varpi_{\mu}^{(n_2)}))^{(n_1)}-\theta_{\lambda,\mu}\zeta_{\mu}^{(n)}\bigg).
\end{align*}
Therefore, $\BRO_{\lambda}^{(n)}$ depends on
$$
f^{(n_2)},\,(\TT{\Phi_{\lambda,\mu,n_0}}f^{(n_2)})^{(n_1)},\quad \textrm{ for $f=\partial_s\zcurve,\partial_s\vectV,\varpi$ and $0\leq n_1+n_2\leq n$.}
$$
In particular, since 
$
(\TT{\Phi_{\lambda,\mu}}f)^{(0)}
%=\TT{\Phi_{\lambda,\mu}^{(0)}}f
=\TT{\Phi_{0}}f
$
for $f\in C^{0,\alpha}(\T)$ and
$$
\frac{1}{2}\sum_{\mu=\pm 1}\theta_{\lambda,\mu}
=\frac{1}{2}+\frac{\lambda}{4},
$$
 we deduce the following corollary.
\begin{cor}\label{cor:BRO0}
Assume $\partial_s\zcurve,\varpi\in C^1(0,T;C^{\alpha}(\T))$ for some $\alpha>0$. Then, 
$$
\BRO_{\lambda}^{(0)}
=\BRO_0-\tfrac{\lambda}{4}\varpi_0\partial_s\zcurve_0,
$$
with $\BRO_{0}$ given in \eqref{BRO0}. 
\end{cor}

\subsection{The operator $\TT{\Phi_{\lambda,\mu}}$}
In this section we analyse $\TT{\Phi_{\lambda,\mu,n_0}}$. For ease of notation let us abbreviate from now on $\Phi_{\lambda,\mu,n_0}$ by $\Phi_{\lambda,\mu}$ and $\map_{\lambda}^{n_0-1)}$ by $\map_{\lambda}$. For any $0\leq n_1\leq n_0-1$,
Fa\`{a} di Bruno's formula yields
\begin{equation}\label{Faa}
\partial_t^{n_1}\Phi_{\lambda,\mu}
=\sum_{b\in\pi_{n_1}}F_b\Phi_{\lambda,\mu}^{|b|+1}\xi^{-|b|}
\prod_{m=1}^{n_1}(\partial_t^m(\map_{\mu}(t,s+\xi)-\map_{\lambda}(t,s)))^{b_m},
\end{equation}
with $\pi_{n_1}:=\{b\in\N_0^{n_1}\,:\,b_1+2b_2+\cdots+ {n_1}b_{n_1}={n_1}\}$, $|b|:=b_1+\cdots b_{n_1}$ and $F_b$ the combinatorial constant
$$F_b:=(-1)^{|b|}\frac{|b|!{n_1}!}{b_1!1!^{b_1}\cdots b_{n_1}!{n_1}!^{b_{n_1}}}.$$
Using \eqref{goodbad} and the binomial theorem we get
\begin{align*}
\prod_{m=1}^{n_1}(\partial_t^m(\map_{\mu}(t,s+\xi)-\map_{\lambda}(t,s)))^{b_m}
&=\prod_{m=1}^{n_1}\sum_{a_m=0}^{b_m}\binom{b_m}{a_m}((\mu-\lambda)\partial_t^m\vectV)^{a_m}(\xi\partial_t^m\triangle_{\xi}\map_{\mu})^{b_m-a_m}\\
&=\sum_{a\leq b}(\mu-\lambda)^{|a|}\xi^{|b-a|}\prod_{m=1}^{n_1}\binom{b_m}{a_m}(\partial_t^m\vectV)^{a_m}(\partial_t^m\triangle_{\xi}\map_{\mu})^{b_m-a_m}.
\end{align*}
Thus, by using \eqref{eq:Phi0}, the dominated convergence theorem yields
\begin{equation}\label{Faa:1}
\partial_t^{n_1}\TT{\Phi_{\lambda,\mu}}f
=\TT{\partial_t^{n_1}\Phi_{\lambda,\mu}}f
=\sum_{b\in\pi_{n_1}}F_b\sum_{a\leq b}(\mu-\lambda)^{|a|}V_{a,b}I_{a,b},
\quad t>0,
\end{equation}
where
$$
V_{a,b}:=\prod_{m=1}^{n_1}\binom{b_m}{a_m}(\partial_t^m\vectV)^{a_m},
\quad\quad
I_{a,b}:=\int_{\T}\triangle_\xi f\Phi_{\lambda,\mu}^{|b|+1}\xi^{-|a|}\prod_{m=1}^{n_1}(\partial_t^m\triangle_\xi\map_{\mu})^{b_m-a_m}\dif\xi.
$$
Although for our $\vectV(t,s)=t\vectW(s)$ simply $V_{a,b}$ equals to $\binom{b_1}{a_1}(\vectW)^{a_1}$ for $a=(a_1,0,\ldots,0)$ and vanishes otherwise, it requires the same effort to keep it general here, and thus may be helpful if different growth rates appear in other problems.

\smallskip

Let us analyse the cases $\mu=\lambda$ and $\mu\neq\lambda$ separately.
As we shall see, while for $\mu=\lambda$ the expression \eqref{Faa:1} holds at $t=0$ too, for $\mu\neq\lambda$ this is only true for $n_1=0$.

\subsubsection{Case $\mu=\lambda$}

\begin{prop}\label{prop:goodkernel} 
For all $0\leq\alpha'<\alpha$ and $f\in C^{k,\alpha}(\T)$ we have
\begin{equation}\label{goodkernel2}
\norma{\TT{\Phi_{\lambda,\lambda}}f}{C_t^{n_1}C_s^{k,\alpha'}}
\leq\frac{C}{\alpha-\alpha'}\norma{f}{k,\alpha},
\end{equation}
where $C$ depends on $\CA{\zcurve_{0}}$ and the $C^{k+1,\alpha}$-norm of $\zcurve^{(n)}$, $\vectV^{(n)}$ for $0\leq n\leq n_0-1$. In particular,
\begin{equation}\label{goodkernel3}
\norma{(\TT{\Phi_{\lambda,\lambda}}f)^{(n_1)}}{k,\alpha'}\leq\frac{C_0}{\alpha-\alpha'}\norma{f}{k,\alpha},
\end{equation}
where $C_0$ depends on $\CA{\zcurve_{0}}$ and the $C^{k+1,\alpha}$-norm of $\zcurve^{(n)}$, $\vectV^{(n)}$ for $0\leq n\leq n_1$. Moreover,
\begin{equation}\label{goodkernel1}
(\TT{\Phi_{\lambda,\lambda}}f)^{(n_1)}=\TT{\Phi_{\lambda,\lambda}^{(n_1)}}f.
\end{equation}
\end{prop}
\begin{proof}
Since $\mu=\lambda$ the expression \eqref{Faa:1} reads as
\begin{equation}\label{Faa:good}
\partial_t^{n_1}\TT{\Phi_{\lambda,\lambda}}f
=\sum_{b\in\pi_{n_1}}F_bI_{0,b}.
\quad t>0.
\end{equation}
Hence, for $k=0$, \eqref{goodkernel2} follows from Corollary~\ref{Phi:bdd}. Moreover, \eqref{Faa:good} holds at $t=0$ by the dominated convergence theorem, and thus \eqref{goodkernel3}\eqref{goodkernel1} follow from Corollary~\ref{Phi:bdd} too. For $k\geq 1$, the same holds by applying the Leibniz rule.
\end{proof}

\subsubsection{Case $\mu\neq\lambda$}
As we mentioned, in contrast to $\Phi_{\lambda,\lambda}$, the expression \eqref{Faa:1} holds at $t=0$ only for ${n_1}=0$, whilst for ${n_1}\geq 1$ we only have
$$(\TT{\Phi_{\lambda,\mu}}f)^{(n_1)}
=\tfrac{1}{{n_1}!}\lim_{t\downarrow 0}\TT{\partial_t^{n_1}\Phi_{\lambda,\mu}(t)}f,$$
because $\TT{\Phi_{\lambda,\mu}^{(n_1)}}f$ is not well-defined. 
Let us prove a lemma which will be helpful. In particular, we consider the auxiliary kernel
$$\Psi_{\lambda,\mu}(t,s,\xi):=\frac{t}{\map_{\mu}(t,s+\xi)-\map_{\lambda}(t,s)},$$
which, by Lemma \ref{geometriclemma}, satisfies
$$\norma{\Psi_{\lambda,\mu}}{\infty}
\leq 2c(s)^{-1},$$
and, recalling \eqref{goodbad},
$$|\Psi_{\lambda,\mu}(t,s,\xi)
-\Psi_{\lambda,\mu}(t,s',\xi)|
\leq 8c(s)^{-1}(\CA{\zcurve_0}\snorma{\partial_s\map_{\mu}}{\alpha}+c(s)^{-1}\snorma{\vectW}{\alpha})|s-s'|^\alpha.$$
Therefore, $\Norma{\Psi_{\lambda,\mu}}{0,\alpha}$ is bounded.

\begin{lemma}\label{lemma:index} 
For every $(a,b)\in\N_0^2$ with $a\leq b$ consider 
$$\mathcal{C}_{\lambda,\mu}^{a,b}(t,s)
:=\int_{\T}\frac{\xi^a\dif\xi}{(\map_{\mu}(t,s+\xi)-\map_{\lambda}(t,s))^{b+1}}
=\int_{\T}\Phi_{\lambda,\mu}^{b+1}\xi^{-(b-a+1)}\dif\xi
,\quad t>0.$$
Then,
\begin{equation}
\norma{\mathcal{C}_{\lambda,\mu}^{a,b}}{C_tC_s^{0,\alpha'}}
\leq\frac{C}{\alpha-\alpha'},
\end{equation}
where $C$ depends on $\CA{\zcurve_{0}}$, $\norma{1/c}{\infty}$ and the $C^{(b-a-n)_++1,\alpha}$-norm of $\zcurve^{(n)},\vectV^{(n)}$ for $0\leq n\leq n_0-1$. In particular,
\begin{equation}
\norma{(\mathcal{C}_{\lambda,\mu}^{a,b})^{(0)}}{k,\alpha'}
\leq\frac{C_0}{\alpha-\alpha'},
\end{equation}
where $C_0$ depends on $\CA{\zcurve_{0}}$ and the $C^{k+b-a+1,\alpha}$-norm of $\zcurve_0$.
\end{lemma}
\begin{proof} 
We shall follow the double induction scheme on $(a,b)$:
$$\begin{array}{cccccccccccc}
b-a=0: & (0,0) & \rightarrow & (1,1) & \rightarrow & (2,2) & \rightarrow & \cdots & \rightarrow & (b,b) & \rightarrow & \cdots \\
& & & \downarrow & & \downarrow & &  & & \downarrow & & \\
b-a=1: & & & (0,1) & \rightarrow & (1,2) & \rightarrow & \cdots & \rightarrow & (b-1,b) & \rightarrow & \cdots \\
& & & & & \downarrow & & & & \downarrow & & \\
b-a=2: & & & & & (0,2) & \rightarrow & \cdots & \rightarrow & (b-2,b) & \rightarrow & \cdots \\
\\
\cdots & & & & & & & \cdots & & & & \cdots
\end{array}$$
where better regularity is required at every drop row.\\
\indent We claim that the following identity holds
\begin{equation}\label{IH}
\mathcal{C}_{\lambda,\mu}^{a,b}
=\frac{2\pi i\theta_{\lambda,\mu}B_{a}^{a,b}}{\partial_s\map_{\mu}^{a+1}}\delta_{a,b}
-\sum_{d=0}^{a}\frac{B_{d}^{a,b}}{\partial_s\map_{\mu}^{d+1}}\left(\frac{X_{\lambda,\mu}^{a-d,b-d}}{b-d}+\int_{\T}\triangle_\xi\partial_s\map_{\mu}\Phi_{\lambda,\mu}^{b+1-d}\xi^{-(b-a)}\dif\xi\right),
\end{equation}
where
\begin{equation*}
B_{d}^{a,b}:=\binom{a}{d}\binom{b}{d}^{-1},\quad 
X_{\lambda,\mu}^{a,b}(t,s)
:=\frac{2^{-a}(1-(-1)^a)}{(\map_{\mu}(t,s+\tfrac{1}{2})-\map_{\lambda}(t,s))^{b}},\quad 
\delta_{a,b}:=\begin{cases}1,&a=b,\\ 0,&a\neq b.\end{cases}
\end{equation*}
We define the auxiliary functions
$$\tilde{\mathcal{C}}_{\lambda,\mu}^{a,b}(t,s):=\int_{\T}\frac{\xi^a\partial_s\map_{\mu}(t,s+\xi)}{(\map_{\mu}(t,s+\xi)-\map_{\lambda}(t,s))^{b+1}}\dif\xi.$$
We can then write
\begin{equation}\label{index:1}
\mathcal{C}_{\lambda,\mu}^{a,b}
=\frac{1}{\partial_s\map_{\mu}}\left(\tilde{\mathcal{C}}_{\lambda,\mu}^{a,b}-\int_{\T}\triangle_\xi\partial_s\map_{\mu}\Phi_{\lambda,\mu}^{b+1}\xi^{-(b-a)}\dif\xi\right).
\end{equation}
In particular, the case $a=0$ in \eqref{IH} follows from \eqref{index:1}, because $X_{\lambda,\mu}^{0,b}=0$ and Cauchy's integral formula yields (for $\mu\neq\lambda$)
$$\tilde{\mathcal{C}}_{\lambda,\mu}^{0,b}
=\int_{\map_{\mu}}\frac{\dif x}{(x-\map_{\lambda}(t,s))^{b+1}}
=2\pi i\theta_{\lambda,\mu}\delta_{0,b}.$$
For $a\geq 1$, an integration by parts yields 
\begin{align*}
a\mathcal{C}_{\lambda,\mu}^{a-1,b-1}
&=\int_{\T}\frac{a\xi^{a-1}\dif\xi}{(\map_{\mu}(t,s+\xi)-\map_{\lambda}(t,s))^{b}}\\
&=\frac{\xi^a}{(\map_{\mu}(t,s+\xi)-\map_{\lambda}(t,s))^{b}}\Big|_{\xi=-\tfrac{1}{2}}^{\xi=+\tfrac{1}{2}}
+b\int_{\T}\frac{\xi^a\partial_s\map_{\mu}(t,s+\xi)}{(\map_{\mu}(t,s+\xi)-\map_{\lambda}(t,s))^{b+1}}\dif\xi,
\end{align*}
that is,
$$a\mathcal{C}_{\lambda,\mu}^{a-1,b-1}=X_{\lambda,\mu}^{a,b}
+b\tilde{\mathcal{C}}_{\lambda,\mu}^{a,b}.$$
Hence, \eqref{index:1} reads as
$$
\mathcal{C}_{\lambda,\mu}^{a,b}
=\frac{1}{\partial_s\map_{\mu}}\left(\frac{a}{b}\mathcal{C}_{\lambda,\mu}^{a-1,b-1}-\frac{X_{\lambda,\mu}^{a,b}}{b}-\int_{\T}\triangle_\xi\partial_s\map_{\mu}\Phi_{\lambda,\mu}^{b+1}\xi^{-(b-a)}\dif\xi\right),
$$
which allows to prove \eqref{IH} by induction.\\ 
\indent In light of the identity \eqref{IH}, to prove the result it is enough to control its last term. Let us assume w.l.o.g.~that $d=0$ for simplicity. Recalling \S\ref{sec:Notation}\ref{Notation:Taylor}, we split it as
$$\int_{\T}\triangle_\xi\partial_s\map_{\mu}\Phi_{\lambda,\mu}^{b+1}\xi^{-(b-a)}\dif\xi
=\sum_{n=0}^{n_0-1}\int_{\T}\triangle_\xi\partial_s\map_{\mu}^{(n)}\Psi_{\lambda,\mu}^{n}\Phi_{\lambda,\mu}^{b+1-n}\xi^{n-(b-a)}\dif\xi.$$
Thus, the terms with $n\geq b-a$ are bounded. In particular, the case $a=b$ is done. For the terms with $n<b-a$, we split 
\begin{align*}
&\int_{\T}\triangle_\xi\partial_s\map_{\mu}^{(n)}\Psi_{\lambda,\mu}^{n}\Phi_{\lambda,\mu}^{b+1-n}\xi^{n-(b-a)}\dif\xi\\
&=t^n\sum_{j=1}^{(b-a)-n}\frac{\partial_s^{j+1}\map_{\mu }^{(n)}}{j!}\mathcal{C}_{\lambda,\mu}^{a+j,b}
+\int_{\T}\triangle_\xi^{(b-a)-n+1}\partial_s\map_{\mu}^{(n)}\Psi_{\lambda,\mu}^n\Phi_{\lambda,\mu}^{b+1-n}\dif\xi.
\end{align*}
The last term above is bounded because $\map_{\mu}^{(n)}$ belongs to $C^{b-a-n+1,\alpha}$, and the others by induction hypothesis. 
This concludes the proof for $k=0$. The case $k\geq1$ follows by applying the Leibniz rule at time $t=0$.
\end{proof}

\begin{prop}\label{prop:badop}
For all $0\leq\alpha'<\alpha$ and $f\in C^{n_1+k,\alpha}(\T)$ we have
\begin{equation}\label{badkernel2}
\norma{\TT{\Phi_{\lambda,\mu}}f}{C_t^{n_1}C_s^{0,\alpha'}}
\leq\frac{C}{\alpha-\alpha'}\norma{f}{n_1,\alpha},
\end{equation}
where $C$ depends on $\CA{\zcurve_{0}}$, $\norma{1/c}{\infty}$ and the $C^{1\vee(n_1-n),\alpha}$-norm of $\zcurve^{(n)}$, $\vectV^{(n)}$ for $0\leq n\leq n_0-1$. In particular,
\begin{equation}\label{badkernel3}
\norma{(\TT{\Phi_{\lambda,\mu}}f)^{(n_1)}}{k,\alpha'}\leq\frac{C_0}{\alpha-\alpha'}\norma{f}{n_1+k,\alpha},
\end{equation}
where $C_0$ depends on $\CA{\zcurve_{0}}$ and the $C^{k+1\vee(n_1-n),\alpha}$-norm of $\zcurve^{(n)}$, $\vectV^{(n)}$ for $0\leq n\leq n_1$. 
\end{prop}
\begin{proof} 
Given $a\leq b\in\pi_{n_1}$, let us show that $I_{a,b}$ in \eqref{Faa:1} is bounded. If $a=0$ we are done as in Proposition~\ref{prop:goodkernel}.  
Otherwise, similarly to Lemma~\ref{lemma:index}, we split it as
$$
I_{a,b}=\sum_{j=1}^{|a|}\frac{\partial_s^j f}{j!}J_{a,b,j}+J_{a,b},
$$
where
\begin{align*}
J_{a,b,j}&:=\int_{\T}\Phi_{\lambda,\mu}^{|b|+1}\xi^{j-(|a|+1)}\prod_{m=1}^{n_1}(\partial_t^m\triangle_\xi\map_{\mu})^{b_m-a_m}\dif\xi,\\
J_{a,b}&:=\int_{\T}
\triangle_\xi^{|a|+1}f
\Phi_{\lambda,\mu}^{|b|+1}\prod_{m=1}^{n_1}(\partial_t^m\triangle_\xi\map_{\mu})^{b_m-a_m}\dif\xi.
\end{align*}
In order to bound $J_{a,b}$ we need $f$ to be in $C^{|a|,\alpha}$ and notice that the largest $|a|$ is for $a_1=b_1=n_1$.\\ 
Fixed $(a,b)$, the most singular term of $J_{a,b,j}$ is for $j=1$. Let us analize it. If $a=b$ this is $J_{b,b,1}=\mathcal{C}_{\lambda,\mu}^{1,|b|}$.
Otherwise, let $m$ be the first index for which $a_m<b_m$. As we did for $f$, let us split $J_{a,b,1}$ as
$$\partial_t^m\triangle_\xi\map_{\mu}
=\sum_{n=m}^{n_0-1}\frac{n!}{(n-m)!}t^{n-m}\triangle_\xi\map_{\mu}^{(n)}.$$
Thus, the terms with $n\geq |a|+m$ are bounded. For the terms with $n<m+|a|$ we split
$$\triangle_\xi\map_{\mu}^{(n)}
=\sum_{d=1}^{|a|+m-n}\frac{\partial_s^d\map_{\mu}^{(n)}}{d!}\xi^{d-1}
+\xi^{|a|+m-n}\triangle_\xi^{|a|+m-n+1}\map_{\mu}^{(n)}.$$
Then, since $b_m\geq 1$ implies $|a|<|b|\leq n_1-(m-1)$,
i.e. $|a|+m\leq n_1$, it is enough to impose that $\map_{\mu}^{(n)}$ belongs to $ C^{n_1-n,\alpha}$. We repeat this $|b-a|$ times, where terms $\mathcal{C}$ as in Lemma \ref{lemma:index} appear and the worst is again $\mathcal{C}_{\lambda,\mu}^{1,|b|}$ with $|b|\leq n_1$. This concludes the proof for $k=0$. The case $k\geq1$ follows by applying the Leibniz rule at time $t=0$.
\end{proof}

\section{Proof of the main results}\label{sec:proof}

As we mentioned in Remark~\ref{Rem:BReq}, it can be seen that the velocity field \eqref{velocity:1}, with $p$ determined by the Bernoulli's law, is a weak solution to (IE) if and only if 
\begin{equation}\label{eq:classical}
\B:=\partial_t\tilde{\varpi}\partial_s\zcurve-\varpi(\partial_t\zcurve-\BRO)=0,
\end{equation}
or equivalently \eqref{BReq}. Moreover, the jump in $p$ across $\Gamma$ vanishes in this case:
\begin{equation}\label{[p]:classical}
[p]=\Re_{\partial_{s}\zcurve}\B=0.
\end{equation}
\indent In our construction, the Reynolds stress $R$ introduces a relaxation whereby \eqref{eq:classical} is regularized. As we saw in Section \ref{sec:subsolution}, under our choice of $\bar{\omega}$ \eqref{ansatzomega}, the construction of $R$ leading to a subsolution adapted to $\Turzone$ is subordinated to determine $(\zcurve,\varpi)$ up to some order $t^{n_0}$. This was the key observation in \cite{Piecewise} for the unstable Muskat problem.
In light of Propositions \ref{p:solvability} and \ref{prop:|R|}, we have to take $(\zcurve,\varpi)$
satisfying three pointwise conditions \ref{p:solvability:pointwise}-\ref{prop:|R|:2}, which can be expressed compactly as
\begin{equation}\label{pointwise}
(\mean{\B}-t\partial_s(\q\partial_s\zcurve))^{2)}=0,
\end{equation}
coupled with two average conditions \ref{Bq}-\ref{p:solvability:average}
\begin{equation}\label{averagecond}
\begin{split}
\int\mean{\B}^{(3}&=0,\\
\int(q_1|\partial_s\zcurve|-c\Im_{\partial_s\zcurve_0}\dev{\B})^{(2}&=0,
\end{split}
\end{equation}
where
\begin{subequations}
\label{b}
\begin{align}
\mean{\B}
&=\tfrac{1}{2}(\partial_t\tilde{\varpi}\partial_s\zcurve-\varpi(\partial_t\zcurve-\mean{\BRO})),\label{b:1}\\
\dev{\B}
&=\tfrac{1}{2}(\partial_t\tilde{\varpi}\partial_s\vectV-\varpi(\vectW-\dev{\BRO})),\label{b:2}
\end{align}
\end{subequations}
and $q=q_1+iq_2$ satisfying
\begin{equation}\label{q}
(q_1|\partial_s\zcurve|-c\Im_{\partial_s\zcurve_0}\dev{\B})^{1)}=0,
\quad\quad
q_2^{(0)}=c\Re_{\partial_s\zcurve_0}\dev{\B}^{(0)}.
\end{equation}

\smallskip

Let $(\zcurve_0,\varpi_0)$ as in \eqref{initial} for some $k_0\geq 0$ and $\alpha>0$ big enough.
We define recursively $(\zcurve,\tilde{\varpi})$ by means of its Taylor decomposition 
\begin{align*}
\zcurve(t,s)&:=\sum_{n=0}^{n_0}t^n\zcurve^{(n)}(s)+t^{n_0+1}\zcurve^{(n_0+1}(t,s),\\
\tilde{\varpi}(t,s)&:=\sum_{n=1}^{n_0}t^n\tilde{\varpi}^{(n)}(s)+t^{n_0+1}\tilde{\varpi}^{(n_0+1}(t,s),
\end{align*}
starting from $(\zcurve,\tilde{\varpi})^{(0)}=(\zcurve_0,0)$, namely the term of order $n=0,1,2$ in \eqref{pointwise} determines $(\zcurve,\tilde{\varpi})^{(n+1)}$, and so $n_0=3$ is enough.

\subsection{Choice of $(\zcurve,\tilde{\varpi})^{(1)}$}\label{sec:proof:1} The zero-order term of \eqref{pointwise} reads as
\begin{equation}\label{pointwise:0}
\mean{\B}^{(0)}=0.
\end{equation}
Since \eqref{b:1} yields
$$\mean{\B}^{(0)}
=\tfrac{1}{2}(\tilde{\varpi}^{(1)}\partial_s\zcurve_0-\varpi_0(\zcurve^{(1)}-\mean{\BRO}^{(0)})),$$
\eqref{pointwise:0} is equivalent to 
$$
\varpi_0(\zcurve^{(1)}-\mean{\BRO}^{(0)})\cdot\partial_s\zcurve_{0}^\perp=0,\quad\quad
\tilde{\varpi}^{(1)}=\varpi_0(\zcurve^{(1)}-\mean{\BRO}^{(0)})\cdot\partial_s\zcurve_{0}.
$$
Thus, since $\mean{\BRO}^{(0)}=\BRO_0$, it is enough to set
\begin{equation}\label{zw:1}
\zcurve^{(1)}=\BRO_0,
\quad\quad
\tilde{\varpi}^{(1)}=0.
\end{equation}
In light of Section~\ref{sec:BRO}, we have $\zcurve^{(1)}\in C^{k_0,\alpha_1}(\T;\R^2)$ for any $0<\alpha_1<\alpha$.

\begin{Rem}
Notice that \eqref{pointwise:0} can be understood as that \eqref{eq:classical} must hold at $t=0$. In particular \eqref{[p]:classical} holds at $t=0$ in the sense that there is not jump of $p$ across $\Gamma_0$ 
$$[p^{(0)}]
%=([p]_++[p]_-)^{(0)}
=2\Re_{\partial_s\zcurve_0}\mean{\B}^{(0)}=0.$$
\end{Rem}

\subsection{Choice of $(\zcurve,\tilde{\varpi})^{(2)}$} The first-order term of \eqref{pointwise} reads as 
\begin{equation}\label{pointwise:1}
\mean{\B}^{(1)}=\partial_s(\q^{(0)}\partial_s\zcurve_0).
\end{equation}
On the one hand, \eqref{b:1} and \eqref{zw:1} yield
$$\mean{\B}^{(1)}
=\tfrac{1}{2}(2\tilde{\varpi}^{(2)}\partial_s\zcurve_0-\varpi_0(2\zcurve^{(2)}-\mean{\BRO}^{(1)})).$$
On the other hand, since $|\partial_s\zcurve_0|=1$ implies 
$\partial_s^2\zcurve_0=\kappa_0\partial_s\zcurve_0^\perp$ with $\kappa_0:=\partial_s^2\zcurve_0\cdot\partial_s\zcurve_0^\perp$ $\equiv$ (signed) curvature of $\zcurve_0$, we have
$$\partial_s(\q^{(0)}\partial_s\zcurve_0)
=(\partial_s\q^{(0)}+i\kappa_0\q^{(0)})\partial_s\zcurve_0.$$
In particular, since $\dev{\BRO}^{(0)}=-\tfrac{1}{4}\varpi_0\partial_s\zcurve_0$ and \eqref{b:2} yield
$$\dev{\B}^{(0)}
=-\tfrac{1}{2}\varpi_0(\vectW-\dev{\BRO}^{(0)})
=-\tfrac{1}{2}\varpi_0(\tfrac{1}{4}\varpi_0+ic)\partial_s\zcurve_0,$$
\eqref{q} implies
$$\q^{(0)}
=ic\dev{\B}^{(0)*}\partial_s\zcurve_0
=-\tfrac{1}{2}c\varpi_0(c+\tfrac{1}{4}i\varpi_0).$$
Therefore, \eqref{pointwise:1} is equivalent to
\begin{align*}
\varpi_0(2\zcurve^{(2)}-\mean{\BRO}^{(1)})\cdot\partial_s\zcurve_{0}^\perp
&=\tfrac{1}{4}\partial_s(c\varpi_0^2)+\kappa_0c^2\varpi_0,\\
2\tilde{\varpi}^{(2)}-\varpi_0(2\zcurve^{(2)}-\mean{\BRO}^{(1)})\cdot\partial_s\zcurve_0
&=\tfrac{1}{4}\kappa_0c\varpi_0^2-\partial_s(c^2\varpi_0).
\end{align*}
Thus, it is enough to set
\begin{align*}
2\zcurve^{(2)}
&=\mean{\BRO}^{(1)}
+(\tfrac{1}{2}c\partial_s\varpi_0+\tfrac{1}{4}\varpi_0\partial_sc+\kappa_0c^2)\partial_s\zcurve_0^\perp,\\
2\tilde{\varpi}^{(2)}
&=\tfrac{1}{4}c\kappa_0\varpi_0^2-\partial_s(c^2\varpi_0).
\end{align*}
On the one hand, $\tilde{\varpi}^{(2)}\in C^{k_0-1,\alpha}(\T;\R)$. On the other hand, in light of Section \ref{sec:BRO}, we have $\zcurve^{(2)}\in C^{k_0-1,\alpha_2}(\T;\R^2)$ for any $0<\alpha_2<\alpha_1$.

\begin{Rem} 
For $|\varpi_0|\gg0$ 
one may set also $\tilde{\varpi}^{(2)}=0$ by taking \begin{equation}\label{z/w}
2\zcurve^{(2)}=
\mean{\BRO}^{(1)}
-\frac{2}{\varpi_0}\partial_s(\q^{(0)}\partial_s\zcurve_0).
\end{equation}
Furthermore, we may set $\tilde{\varpi}=0$.
For the Birkhoff-Rott equations \eqref{BReq} this can be done by taking $r=0$, which can be understood as fixing the parametrization that keeps $\varpi$ constant in time. Moreover, in this case one may assume also that $\varpi_0$ is constant in $s$ by choosing $\zcurve_0$ properly (not necessarily arc-length). However, for mixed sign vorticities the choice \eqref{z/w} is singular. In spite of this, since $(\partial_s(\q^{(0)}\partial_s\zcurve_0))\cdot\partial_s\zcurve_0^\perp\sim\varpi_0$, it is not necessary to divide by $\varpi_0$ by taking $\tilde{\varpi}^{(2)}$ as above.
\end{Rem}

\subsection{Choice of $(\zcurve,\tilde{\varpi})^{(3)}$}
The second-order term of \eqref{pointwise} reads as
\begin{equation}\label{pointwise:2}
\mean{\B}^{(2)}=\partial_s(\q\partial_s\zcurve)^{(1)}.
\end{equation}
On the one hand, \eqref{b:1} and \eqref{zw:1} yield
$$\mean{\B}^{(2)}
=\tfrac{1}{2}(3\tilde{\varpi}^{(3)}\partial_s\zcurve_0+2\tilde{\varpi}^{(2)}\partial_s\zcurve^{(1)}-\varpi_0(3\zcurve^{(3)}-\mean{\BRO}^{(2)})).$$
On the other hand, we split
$$
\partial_s(\q\partial_s\zcurve)^{(1)}
=
\underbrace{\partial_s(\q^{(0)}\partial_s\zcurve^{(1)}+q_1^{(1)}\partial_s\zcurve_0)}_{\equiv\tilde{\q}}
+\partial_s(i\q_2^{(1)}\partial_s\zcurve_0)
=\tilde{\q}+(i\partial_sq_2^{(1)}-\kappa_0q_2^{(1)})\partial_s\zcurve_0.$$
In particular, since \eqref{b:2} and \eqref{zw:1} yield $$\dev{\B}^{(1)}=\tfrac{1}{2}\varpi_0\dev{\BRO}^{(1)},$$
\eqref{q} implies
$$\q_1^{(1)}
=-(\q_1^{(0)}\partial_s\zcurve^{(1)}+ic\dev{\B}^{(1)})\cdot\partial_s\zcurve_0
=\tfrac{1}{2}c\varpi_0(c\partial_s\zcurve-i\dev{\BRO})^{(1)}\cdot\partial_s\zcurve_0.$$
Since $q_2^{(1)}$ is free, \eqref{pointwise:2} is equivalent to solve
\begin{align*}
\int(\mean{\B}^{(2)}-\tilde{\q})\cdot\partial_s\zcurve_0^\perp
&=0.\\
(\mean{\B}^{(2)}-\tilde{\q})\cdot\partial_s\zcurve_0
&=-\kappa_0\underbrace{\int_0^s(\mean{\B}^{(2)}-\tilde{\q})\cdot\partial_s\zcurve_0^\perp\dif s_1}_{=q_2^{(1)}}.
\end{align*}
Thus, it is enough to set
\begin{align*}
3\zcurve^{(3)}&=\mean{\BRO}^{(2)}+H\partial_s\zcurve_0^\perp,\\
3\tilde{\varpi}^{(3)}
&=-h\cdot\partial_s\zcurve_0
-\kappa_0\int_0^s(h\cdot\partial_s\zcurve_0^\perp-H\varpi_0)\dif s_1,
\end{align*}
with
$$
h
\equiv
2(\tilde{\varpi}^{(2)}\partial_s\zcurve^{(1)}
-\tilde{\q}),
\quad
H\equiv \frac{\int h\cdot\partial_s\zcurve_0^\perp}{\int\varpi_0^2}\varpi_0.
$$
In light of Section \ref{sec:BRO}, it follows that $\tilde{\varpi}^{(3)}\in C^{k_0-2,\alpha_3}(\T;\R)$ and $\zcurve^{(3)}\in C^{k_0-2,\alpha_3}(\T;\R^2)$ for any $0<\alpha_3<\alpha_2$. Therefore, Lemma \ref{lemma:vRcontrol} requires $k_0\geq 4$.

\subsection{Choice of the remainder}
The average conditions \eqref{averagecond} read as 
\begin{equation}\label{averagecond2}
\begin{split}
\int\mean{\B}^{(3}&=0,\\
\int(\mean{\B}\cdot\zcurve+ct\dev{\B}\cdot\partial_s\zcurve_0^\perp)^{(3}&=0.
\end{split}
\end{equation}
We declare $\tilde{\varpi}^{(n_0+1}=0$ and
$$
\zcurve^{(n_0+1}(t,s)
:=t^{-(n_0+1)}\int_0^t\tau^{n_0}Z(\tau,s)\dif\tau
=\int_0^1\tau^{n_0}Z(\tau t,s)\dif\tau,
$$
for some
$Z$ to be determined. Since $\varpi_0\neq 0$ we can take a cutoff function $\psi_0\in C^\infty(\T;\R_+)$ vanishing on $\{s\in\T\,:\,|\varpi_0(s)|\leq\tfrac{1}{2}\norma{\varpi_0}{\infty}\}$ and with $\int\psi_0=1$. We declare
\begin{equation}\label{ansatzZ}
Z:=\frac{1}{\varpi_{0}}(\gamma\psi_0-\delta\partial_s(\psi_0\partial_s\zcurve_{0})),
\end{equation}
with $(\gamma,\delta)\in\R^2\times\R=\R^3$ a time dependent vector, to be determined.
Let us split
\begin{align*}
\mean{\B}^{(3}
&=\tfrac{1}{2}(\varGamma-\varpi_{0} Z),\\
(\mean{\B}\cdot\zcurve+ct\dev{\B}\cdot\partial_s\zcurve_0^\perp)^{(3}
&=\tfrac{1}{2}((\varGamma-\varpi_{0}Z)\cdot\zcurve+\varLambda),
\end{align*}
where we have abbreviated
\begin{align*}
\varGamma&\equiv\partial_s\tilde{\varpi}Z+(\partial_t\tilde{\varpi}\partial_s\zcurve
-\varpi(\partial_t\zcurve^{3)}-\mean{\BRO}))^{(3},\\
\varLambda&\equiv
2(\mean{\B}^{2)}\cdot\zcurve+ct\dev{\B}\cdot\partial_s\zcurve_0^\perp)^{(3}.
\end{align*}
Therefore, \eqref{averagecond2} reads as
\begin{equation}\label{averagecond3}
\gamma(t)=\int\varGamma(t),
\quad\quad
\delta(t)=\int\varDelta(t),
\end{equation}
with $$\varDelta\equiv\frac{\varLambda+(\varGamma-\gamma\psi_0)\cdot\zcurve}{\displaystyle\int\psi_0\partial_s\zcurve_{0}\cdot\partial_s\zcurve}.$$
Hence, since \eqref{averagecond3} is an implicit equation $(\gamma,\delta)=F(\gamma,\delta)$ with $F\in C^{0,1}(\R^3;\R^3)$ and
$$|F(\gamma_1,\delta_1)-F(\gamma_0,\delta_0)|\leq\mathcal{O}(t)|(\gamma_1,\delta_1)-(\gamma_0,\delta_0)|,$$
the existence (and uniqueness) of $(\gamma,\delta)$ (and so $Z\in C_tC_s^{k_0-1}$) follows from the Banach fixed point theorem, namely the Lipschitz constant, and so the time of existence $t_1>0$, depends on $\norma{\zcurve_0}{k_0+1,\alpha}$,  $\norma{\varpi_0}{k_0,\alpha}$, $\norma{1/c}{\infty}$ and $\CA{\zcurve_{0}}$.\\

Once we have fixed $(\zcurve,\tilde{\varpi})$ and $c$ \eqref{c}, we define the turbulence zone $\Turzone(t)$ and the vorticity $\bar{\omega}(t)$ by means of the map $\map(t)$ \eqref{def:turzone} and \eqref{ansatzomega} respectively, for all $0<t\leq T$ smaller than $t_0$ in Lemma~\ref{geometriclemma} and $t_1$ above. Secondly, we define the velocity $\bar{\velocity}=\BSO(\bar{\omega})$, the pressure $\bar{p}$ by means of the Bernoulli's law \eqref{Blaw}, and the Reynolds stress $R$ as in Proposition~\ref{R:solution}. Then, Proposition \ref{p:solvability} guarantees that $R$ is uniformly bounded and so Theorem \ref{thm:hprinciple} applies. Finally, for the associated weak solutions, Proposition \ref{prop:D0} yields Theorems \ref{thm:global}-\ref{thm:local}.

\section{Piecewise harmonic subsolutions}\label{sec:piecewise}

Following \cite[\S5]{Piecewise}, we generalize the previous construction from Sections \ref{sec:subsolution}-\ref{sec:proof} to the case of subsolutions $\bar{\velocity}$ whose vorticity $\bar{\omega}$ is concentrated on $2N$ curves for $N\geq 1$. 

\subsection{Geometric setup} 
In light of Example \ref{Ex:flat}, it seems suitable to consider the grid
$\Lambda:=\{\lambda_j\,:\,1\leq|j|\leq N\}$ of $[-1,1]$ 
given by
$$
\lambda_j=\sgn j\tfrac{2|j|-1}{|\Lambda|-1}.
$$
Observe that $0<\lambda_1<\cdots<\lambda_N=1$ and $\lambda_{-j}=-\lambda_j$ for $1\leq j\leq N$. 
Given $\lambda\in\Lambda_+:=\{\lambda_j\,:\,1\leq j\leq N\}$, at each $t>0$ we define $\Turzone^{\lambda}(t)$ as the annular region in $\R^2$ whose boundary is
$$\partial\Turzone^{\lambda}(t)=\Gamma_{-\lambda}(t)\cup\Gamma_{\lambda}(t),$$
with $\Gamma_{\pm\lambda}(t):=\map_{\pm\lambda}(t,\T)$ parametrized by the map
$$\map_{\pm\lambda}(t,s):=\zcurve_{\lambda}(t,s)\pm\lambda tc(s)\partial_s\zcurve_0(s)^\perp,$$
where $\zcurve_{\lambda}$ is an evolution of $\zcurve_{0}$ to be determined. We note that, following Lemma \ref{geometriclemma}, there is not intersection of different $\Gamma_{\lambda}(t)$ for short times because we shall take $\zcurve_{\lambda}^{(1)}=\BRO_{0}$ independently of $\lambda$ (cf.~\S\ref{sec:proof:1}). The turbulence zone is then $\Turzone:=\Turzone^{\lambda_N}\supset\cdots\supset\Turzone^{\lambda_1}$. We denote also $\Gamma:=\bigcup\Gamma_{\lambda}$.

\subsection{The velocity} 
The \textit{ansatz} \eqref{ansatzomega} is generalized here by
$$\bar{\omega}(t)
:=\frac{1}{|\Lambda|}\sum_{\lambda\in\Lambda}\map_{\lambda}(t)^\sharp(\varpi_{\lambda}(t)\dif s),$$
with
$\varpi_{\lambda}=\varpi_0+\partial_s\tilde{\varpi}_{\lambda}$
and $(\zcurve,\tilde{\varpi})_{-\lambda}=(\zcurve,\tilde{\varpi})_{\lambda}$ for every $\lambda\in\Lambda_+$. More precisely, we define $(\zcurve,\tilde{\varpi})_{\lambda}$ by means of its Taylor decomposition
$$
\zcurve_{\lambda}(t,s):=\sum_{k=0}^{n_0}t^n\zcurve_{\lambda}^{(n)}(s)+t^{n_0+1}\zcurve_{\lambda}^{(n_0+1}(t,s),\quad\quad
\tilde{\varpi}_{\lambda}(t,s):=\sum_{n=0}^{n_0}t^n\tilde{\varpi}_{\lambda}^{(n)}(s),
$$
with
$$
\zcurve_{\lambda}^{(n_0+1}(t,s)
:=t^{-(n_0+1)}\int_0^t\tau^{n_0}Z_{\lambda}(\tau,s)\dif\tau
=\int_0^1\tau^{n_0}Z_{\lambda}(\tau t,s)\dif\tau,
$$
and
$Z_{\lambda}$ to be determined. \\
\indent Thus, for $t>0$ the velocity is given by $\bar{\velocity}(t):=\BSO(\bar{\omega}(t))$, 
$$\bar{\velocity}(t,x)^*
=\frac{1}{|\Lambda|}\sum_{\lambda\in\Lambda}
\frac{1}{2\pi i}\int_{\T}\frac{\varpi_{\lambda}(t,s)}{x-\map_{\lambda}(t,s)}\dif s,
\quad x\notin\Gamma(t).$$
This $\bar{\velocity}(t)$ is bounded, anti-holomorphic outside $\Gamma(t)$ but with tangential discontinuities across $\Gamma(t)$. Indeed, these limits $\bar{\velocity}_{\lambda}^{\pm}(t,s)$  are
$$\bar{\velocity}_{\lambda}^{\pm}
=\BRO_{\lambda}\mp\tfrac{1}{2|\Lambda|}\zeta_{\lambda}^*,$$
where $\zeta_{\lambda}\equiv\frac{\varpi_{\lambda}}{\partial_s\map_{\lambda}}$ and $\BRO_{\lambda}\equiv\BRO_{\lambda}(\zcurve,\varpi)$ are the Birkhoff-Rott type operators
$$\BRO_{\lambda}(t,s)^*
=\frac{1}{|\Lambda|}\sum_{\mu\in\Lambda}\frac{1}{2\pi i}\PV\!\int_{\T}\frac{\varpi_{\mu}(t,s')}{\map_{\lambda}(t,s)-\map_{\mu}(t,s')}\dif s'.$$
Notice that the $\PV$ is not necessary for $\mu\neq\lambda$ when $t>0$. Therefore,
$$
\M{\bar{\velocity}}{\lambda}=\BRO_{\lambda},
\quad\quad
\J{\bar{\velocity}}{\lambda}=-\tfrac{1}{|\Lambda|}\zeta_{\lambda}^*.
$$
This $\BRO_{\lambda}$ can be written as
$$\BRO_{\lambda}^{*}
=\frac{1}{|\Lambda|}\sum_{\mu\in\Lambda}(\zeta_{\mu}\TT{\Phi_{\lambda,\mu}}(\partial_s\map_{\mu})-\TT{\Phi_{\lambda,\mu}}(\varpi_{\mu})-\theta_{\lambda,\mu}\zeta_{\mu}),$$
with $\theta_{\lambda,\mu}:=\frac{1+\sgn(\lambda-\mu)}{2}$.
In particular, for $\lambda=\lambda_j$,
it is straightforward to check that
$$\frac{1}{|\Lambda|}\sum_{\mu\in\Lambda}\theta_{\lambda,\mu}
=\frac{1}{2}+\sgn{j}\frac{2|j|-1}{2|\Lambda|}
=\frac{1}{2}+\lambda\bar{c}_N,$$
from which we generalize Corollary \ref{cor:BRO0}:
\begin{equation}\label{BRON}
\BRO_{\lambda}^{(0)}=\BRO_{0}-\lambda\bar{c}_N\zeta_0.
\end{equation}
\subsubsection{Helmholtz decomposition of $\bar{\velocity}$} Analogously to Section \ref{sec:Helmholtz}, $\bar{\velocity}$ can be written as
$$\bar{\velocity}=\nabla\phi+\circulation K_{x_0}^*,$$
where
$$\circulation(t,x)=\frac{1}{|\Lambda|}\sum_{\lambda\in\Lambda}(1-\Ind_{\map_{\lambda}(t)}(x))\int\varpi_0,$$
and $\phi$ is the (piecewise) harmonic function
$$\phi(t,x)=
\frac{\Re}{|\Lambda|}
\sum_{\lambda\in\Lambda}\frac{1}{2\pi i}
\int_{\T}\varpi_{\lambda}(t,s)(\mathrm{L}_{\map_{\lambda}(t,s)}(x)-(1-\Ind_{\map_{\lambda}(t)}(x))\Log(x-x_0))\dif s
+O(t),$$
for $x\notin\Gamma(t)$, where $O$ can be chosen in such a way that
$$\partial_t\phi(t,x)
=\frac{\Re}{|\Lambda|}
\sum_{\lambda\in\Lambda}\frac{1}{2\pi i}\int_{\T}\frac{(\partial_s\map\partial_{t}\tilde{\varpi}-\partial_t\map\varpi)_{\lambda}(t,s)}{x-\map_{\lambda}(t,s)}\dif s,$$
and so
$$[\partial_t\phi]_{\lambda}
=\tfrac{1}{|\Lambda|}
(\varpi\Re_{\partial_s\map}(\partial_{t}\map)-\partial_{t}\tilde{\varpi})_{\lambda}.$$

\subsection{The pressure} We define $\bar{p}$ outside $\Gamma$ by means of the Bernoulli's law
$$\bar{p}:=-\partial_t\phi-\tfrac{1}{2}|\bar{\velocity}|^2
\quad\textrm{outside }\Gamma.$$
Thus, analogously to Proposition~\ref{p:[p]}, we have
$$[\bar{p}]_{\lambda}=\tfrac{1}{|\Lambda|}
(\partial_{t}\tilde{\varpi}-\varpi\Re_{\partial_s\map}(\partial_{t}\map-\BRO))_{\lambda}.$$

\subsection{The Reynolds stress} We define $R$ as 
$$
R:=\sum_{\lambda\in\Lambda_+}R^{\lambda}\car{\Turzone^{\lambda}}.
$$
Analogously to Proposition~\ref{prop:consmom}, each  $R^{\lambda}$ must satisfy
\begin{align*}
\Div R^{\lambda}&=0 \hspace{1.50cm} \textrm{in }\Turzone^{\lambda},\\
\pm(R^{\lambda}\partial_s\map^\perp)_{\pm\lambda}&=i\B_{\pm\lambda} \hspace{1.0cm} \textrm{on }\Gamma_{\pm\lambda},
\end{align*}
where each $\B_{\lambda}$ is the boundary condition
$$\B_{\lambda}
=\tfrac{1}{|\Lambda|}(\partial_t\tilde{\varpi}\partial_s\map
-\varpi(\partial_t\map-\BRO))_{\lambda}.$$
On the one hand, by \eqref{BRON} and taking $(\zcurve,\tilde{\varpi})_{\lambda}^{(1)}=(\BRO_{0},0)$ as in Section \ref{sec:proof:1}, we have
$$
\mean{\B}_{\lambda}^{(0)}=0,\quad
\dev{\B}_{\lambda}^{(0)}
=-\tfrac{|\lambda|}{|\Lambda|}\varpi_0(\bar{c}_N\varpi_0+ic)\partial_s\zcurve_0,$$
where $\mean{\B}_{\lambda}:=\tfrac{1}{2}(b_{\lambda}+b_{-\lambda})$ and $\dev{\B}_{\lambda}:=\tfrac{1}{2}(b_{\lambda}-b_{-\lambda})$ for $\lambda\in\Lambda_+$.\\
On the other hand, analogously to Proposition~\ref{prop:|R|}, since
$$|\mathring{R}^{\lambda(0)}|
=|\dev{\B}_{\lambda}^{(0)}-\tfrac{1}{2}\tr R^{\lambda(0)}\partial_s\zcurve_0|,$$
by taking $H^{\lambda}=l_3^{\lambda}=0$ and $q_2^{\lambda(0)}=c|\lambda|\dev{\B}_{\lambda}^{(0)}\cdot\partial_s\zcurve_0$, this is minimized by $\tr R^{\lambda(0)}=2\dev{\B}_\lambda^{(0)}\cdot\partial_s\zcurve_0$:
$$|\mathring{R}^{\lambda(0)}|
=|\dev{\B}_{\lambda}^{(0)}\cdot\partial_s\zcurve_0^\perp|
=\tfrac{|\lambda|}{|\Lambda|}c|\varpi_0|.$$
Finally, since
$$(\mathring{R}\partial_s\map^\perp)_{\lambda}^{(0)}
=i(\dev{\B}^{(0)}
-\Re_{\partial_s\zcurve_0}\dev{\B}_\lambda^{(0)}\partial_s\zcurve_0)
=-\Im_{\partial_s\zcurve_0}\dev{\B}_\lambda^{(0)}\partial_s\zcurve_0
=\tfrac{|\lambda|}{|\Lambda|}c\varpi_0\partial_s\zcurve_0,$$
we obtain
\begin{align*}
&(\BRO\mathring{R}\partial_s\map^\perp
+(|\mathring{R}|+e')\partial_t\map\cdot\partial_s\map^\perp)_{\lambda}^{(0)}\\
&=
(\BRO_{0}-\lambda\bar{c}_N\varpi_0\partial_s\zcurve_0)\cdot(\tfrac{|\lambda|}{|\Lambda|}c\varpi_0\partial_s\zcurve_0)+\tfrac{|\lambda|}{|\Lambda|}c|\varpi_0|(\BRO_0\cdot\partial_s\zcurve_0^\perp+\lambda c)\\
&=\tfrac{|\lambda|}{|\Lambda|}c|\varpi_0|(\lambda(c-\bar{c}_N|\varpi_0|)
+B_0),
\end{align*}
where $B_0:=\BRO_0\cdot((\sgn\varpi_0+i)\partial_s\zcurve_0)$.
The rest follows analogously to the case $N=1$.

\appendix

\section{Infinite energy lemmas}
In this section we prove two lemmas for (bounded) weak solutions which may not have finite kinetic energy $E(t)=\tfrac{1}{2}\norma{\velocity(t)}{2}^2$.

\begin{lemma}\label{De} 
Let $D$ be the dissipation measure in Proposition \ref{prop:D0}.
For any $M>0$ let $\psi_M\in C_c^\infty(\R^2;\R_+)$ be a radial function with $\car{B_{M}}\leq\psi_M\leq\car{B_{M+1}}$ and $\norma{\nabla\psi_M}{\infty}\leq 2$. Then,
$$\langle D(t_2)-D(t_1),\car{}\rangle
=\lim_{M\rightarrow\infty}\langle  D(t_2)-D(t_1),\psi_M\rangle
=\int_{\R^2}(e(t_1)-e(t_2))\dif x.$$
\end{lemma}
\begin{proof} Since $\psi_M$ does not depend on $t$, Definition \ref{defi:dissipation} reads as
\begin{equation}\label{De:1}
\langle D(t_2)-D(t_1),\psi_M\rangle
=\int_{t_1}^{t_2}\int_{\R^2} \left(e+p\right)\velocity\cdot\nabla\psi_M\dif x\dif\tau
-\int_{\R^2}(e(t_2)-e(t_1))\psi_M\dif x.
\end{equation}
Notice that
$$\int_{\R^2}||\bar{\velocity}(t_2)|^2-|\bar{\velocity}(t_1)|^2|\dif x
=\int_{\R^2}|\underbrace{(\bar{\velocity}(t_2)+\bar{\velocity}(t_1))}_{\sim(1+|x|)^{-1}}\cdot\underbrace{(\bar{\velocity}(t_2)-\bar{\velocity}(t_1))}_{\sim(1+|x|)^{-2}}|\dif x<\infty,$$
where we have applied \eqref{BSO:decay}. Hence, the dominated convergence theorem allows to pass to the limit in the second term of \eqref{De:1}. Finally, since $|\nabla\psi_M|\leq 2\car{B_{M+1}\setminus B_M}$ and $\velocity,(e+p)\sim (1+|x|)^{-1}$, we have
$$\int_{\R^2} |(e+p)\velocity\cdot\nabla\psi_M|\dif x
\leq 2\int_{B_{M+1}\setminus B_M}|(e+p)\velocity|\dif x
\lesssim
\log\left(\frac{M+1}{M}\right)\rightarrow 0$$
as $M\rightarrow\infty$.
\end{proof}

\begin{lemma}[Weak-strong uniqueness principle]\label{lemma:WSU} Assume there is a strong solution $(\mathbf{v},\mathbf{p})$  to $\mathrm{(IE)}$ satisfying $(\nabla_{\mathrm{sym}}\mathbf{v})^-\in L_t^1L^{\infty}$ (recall \S\ref{sec:Notation}\ref{Notation:matrix}). Then, if $(\velocity,p)$ is an admissible weak solution to $\mathrm{(IE)}$ with $\velocity_0=\mathbf{v}_0$ and
\begin{equation}\label{WSU:cond}
\int_0^t\int_{B_{M+1}\setminus B_M}|\velocity-\mathbf{v}|^{a}|p-\mathbf{p}|^{b}\dif x\dif\tau\underset{M\rightarrow\infty}{\longrightarrow}0,
\end{equation}
for $(a,b)=(2,0)$ and $(1,1)$, necessarily $(\velocity,p)=(\mathbf{v},\mathbf{p})$. 
\end{lemma}
\begin{proof} Let $\psi\in C^1(\R^3;[0,1])$ be a test function with $\psi\equiv 1$ on $\sop D$. %Then, by hypothesis we have $\langle D,\psi\rangle\geq 0$. 
Let us consider the error
$$F_{\psi}(t):=\frac{1}{2}\int_{\R^2}|\velocity(t)-\mathbf{v}(t)|^2\psi(t)\dif x.$$
Hence, we deduce
\begin{align}
F_{\psi}
&=\int_{\R^2}e\psi\dif x
+\int_{\R^2}\mathbf{e}\psi\dif x
-\int_{\R^2}\velocity\cdot\mathbf{v}\psi\dif x\nonumber\\
&\leq\int_{0}^{t}\int_{\R^2}\left( e\partial_t\psi+\left(e+p\right)\velocity\cdot\nabla\psi\right)\dif x\dif\tau
+\int_{\R^2}e_0\psi_0\dif x\label{WS:1}\\
&+\int_{0}^{t}\int_{\R^2}\left( \mathbf{e}\partial_t\psi+\left(\mathbf{e}+\mathbf{p}\right)\mathbf{v}\cdot\nabla\psi\right)\dif x\dif\tau
+\int_{\R^2}\mathbf{e}_0\psi_0\dif x\label{WS:2}\\
&-\int_0^t\int_{\R^2}(\velocity\cdot\partial_t(\mathbf{v}\psi)
+\velocity\otimes\velocity :\nabla(\mathbf{v}\psi)+p\,\Div(\mathbf{v}\psi))\dif x\dif\tau
-\int_{\R^2}\velocity_0\cdot\mathbf{v}_0\psi_0\dif x,\label{WS:3}
\end{align}
where we have applied $\langle D,\psi\rangle\geq 0$ in \eqref{WS:1}, $\langle\mathbf{D},\psi\rangle= 0$ in \eqref{WS:2} and Definition~\ref{Euler:weak} for $(\velocity,p)$ tested with $\mathbf{v}\psi$ in \eqref{WS:3}. Since $\tfrac{1}{2}\velocity_0\cdot\mathbf{v}_0=\mathbf{e}_0=e_0$, the last terms in \eqref{WS:1}-\eqref{WS:3} cancel each other out. 
It can be checked that the above inequality can be written as
$F_\psi\leq I_\psi+J_\psi$ where
\begin{align*}
I_{\psi}&=-\int_0^t\int_{\R^2}(\velocity\cdot\partial_t\mathbf{v}
+\velocity\otimes\velocity :\nabla\mathbf{v})\psi\dif x\dif\tau
+\int_0^t\int_{\R^2}(\mathbf{p}\velocity+\mathbf{e}(\mathbf{v}-\velocity))\cdot\nabla\psi\dif x\dif\tau,\\
J_{\psi}&=\frac{1}{2}\int_0^t\int_{\R^2}|\velocity-\mathbf{v}|^2D_t\psi\dif x\dif\tau+\int_0^t\int_{\R^2}(p-\mathbf{p})(\velocity-\mathbf{v})\cdot\nabla\psi\dif x\dif\tau,
\end{align*}
being $D_t\equiv\partial_t+\velocity\cdot\nabla$ the material derivative.
For $I_\psi$, since $(\mathbf{v},\mathbf{p})$ is a strong solution  we have
$$-\int_0^t\int_{\R^2}\velocity\cdot\partial_t\mathbf{v}\psi\dif x\dif\tau
=\int_0^t\int_{\R^2}\velocity\cdot\Div(\mathbf{v}\otimes\mathbf{v})\psi\dif x\dif\tau
+\int_0^t\int_{\R^2}\velocity\cdot\nabla\mathbf{p}\psi\dif x\dif\tau.$$
Hence, by applying %($\nabla\mathbf{v}=\partial_i\mathbf{v}_j$)
$$
\velocity\cdot\Div(\mathbf{v}\otimes\mathbf{v})
=\mathbf{v}\cdot\nabla\mathbf{v}\velocity,\quad\quad
\velocity\otimes\velocity :\nabla\mathbf{v}=\velocity\cdot\nabla\mathbf{v}\velocity,
$$
we get
$$\int_0^t\int_{\R^2}(\velocity\cdot\Div(\mathbf{v}\otimes\mathbf{v})
-\velocity\otimes\velocity :\nabla\mathbf{v})\psi\dif x\dif\tau
=\int_0^t\int_{\R^2}(\mathbf{v}-\velocity)\cdot\nabla\mathbf{v}\velocity\psi\dif x\dif\tau.$$
Thus, since $\velocity,\mathbf{v}\in L^\infty_\sigma$ implies
\begin{align*}
0&=\int_{\R^2}\velocity\cdot\nabla(\mathbf{p}\psi)\dif x
=\int_{\R^2}\velocity\cdot\nabla\mathbf{p}\psi\dif x+\int_{\R^2}\mathbf{p}\velocity\cdot\nabla\psi\dif x,\\
0&=\int_{\R^2}(\mathbf{v}-\velocity)\cdot\nabla(\mathbf{e}\psi)\dif x
=\int_{\R^2}\mathbf{e}(\mathbf{v}-\velocity)\cdot\nabla\psi\dif x+\int_{\R^2}(\mathbf{v}-\velocity)\cdot\nabla\mathbf{v}\mathbf{v}\psi\dif x,
\end{align*}
we have
$$I_\psi
=-\int_0^t\int_{\R^2}(\mathbf{v}-\velocity)\cdot\nabla_{\mathrm{sym}}\mathbf{v}(\mathbf{v}-\velocity)\psi\dif x\dif\tau.$$
In particular
$$I_\psi
\leq 2\int_0^t\norma{(\nabla_{\mathrm{sym}}\mathbf{v})^-}{\infty}F_{\psi}\dif\tau.$$
\indent Now let us take $\psi_M$ as in Lemma~\ref{De}. Thus, by applying the decay hypothesis \eqref{WSU:cond}, for every $\varepsilon>0$ there is $M_0(\varepsilon)>0$ so that
$$J_{\psi_M}
\leq\varepsilon\quad\textrm{for all }M\geq M_0.$$
Plugging all together, Gr\"{o}nwall's inequality yields
$$F_{\psi_M}(t)\leq\varepsilon\mathrm{exp}\left(2\int_0^t\norma{(\nabla_{\mathrm{sym}}\mathbf{v})^-}{\infty}\dif\tau\right)
\quad\textrm{for all }M\geq M_0.$$
Finally, the statement follows by taking $\limsup_{M\uparrow\infty}$ above and making $\varepsilon\downarrow 0$ after.
\end{proof}

\section{Simulations}\label{sec:simulations}

In this section we provide some numerical simulations with the aim of illustrating how these solutions may look like. To this end we consider the classical vortex-blob regularization (\cite{caflischlowengrub,Kra86}) which consists of desingularazing the Cauchy kernel $K(x)=\frac{1}{2\pi i x}$ by introducing a parameter $\delta>0$ in the denominator
$$K_{\delta}(x)
:=K(x)\frac{|x|^2}{|x|^2+\delta^2}
=\left(\frac{1}{2\pi}\frac{x^\perp}{|x|^2+\delta^2}\right)^*.$$
Let
$h$ $\equiv$ time step and $S$ $\equiv$ grid of $\T$. Here we take $\delta=0.002$, $h=0.025$ and $|S|=20000$ points. Thus, we consider the discrete $\delta$-BR equation
\begin{equation}\label{BRdelta}
\frac{\zcurve(t+h,s)-\zcurve(t,s)}{h}
=\frac{\ell}{2\pi|S|}\sum_{s'\in S}\frac{(\zcurve(t,s)-\zcurve(t,s'))^\perp}{|\zcurve(t,s)-\zcurve(t,s')|^2+\delta^2}\varpi(t,s'),
\quad s\in S,
\end{equation}
which yields a recurrence for $t=0$, $h$, $2h$, $3h$, $\ldots$ starting from $\zcurve_0$.
For simplicity we shall focus on the circle $\zcurve_0(s)=e^{is}$ ($\ell=2\pi$),
for different vortex sheet strengths $\varpi(t)=\varpi_0$. To simulate the Kelvin-Helmholtz instability we consider a tiny perturbation of $\zcurve_{0}$
$$\zcurve_{0,\gamma}=\zcurve_0-\gamma\partial_s\zcurve_0^\perp,
%=(1+\gamma)\zcurve_{0},
$$
with $\gamma(s)=\epsilon\sin(ks)$. Here we take $\epsilon=0.001$ (perturbation amplitude) and $k=30$ (perturbation frequency).
For times
$t=0$, $1.25$ and $2.5$ (from left to right) we plot the macroscopic vector field
$$\bar{\velocity}(t,x)=\frac{\ell}{2\pi |S||\Lambda|}\sum_{\lambda\in\Lambda}\sum_{s'\in S}\frac{(x-\map_{\lambda}(t,s'))^\perp}{|x-\map_{\lambda}(t,s')|^2+\delta^2}\varpi(t,s'),$$
the Kelvin-Helmholtz curve $\zcurve_{\gamma}(t)$ (light blue) given by \eqref{BRdelta} starting from $\zcurve_{0,\gamma}$, and the boundary of the turbulence zone $\map_{\pm}(t)=\zcurve(t)\pm ct\partial_s\zcurve_0^\perp$ (dark blue) with 
$c(s)=\beta(|\varpi_0|*\eta_\epsilon)(s)$
and, for simplicity, $\zcurve(t)$ given by \eqref{BRdelta} starting from $\zcurve_{0}$, coupled with the points where $\varpi_0$ vanishes (red). Here we take $|\Lambda|=10$ and $\epsilon=\ell/20$. In the pictures below $\beta=\tfrac{1}{8}$. However, for short times we have observed that $\beta=\tfrac{1}{4}$ may fit better as $\delta$ decreases. Although we would have liked to explore the scope of this view point in more detail, we have thought appropriate to present this simple approach here and leave possible improvements for future works.

\begin{figure}[h!]
\centering
\subfigure{\label{fig:00}\includegraphics[width=0.32\textwidth]{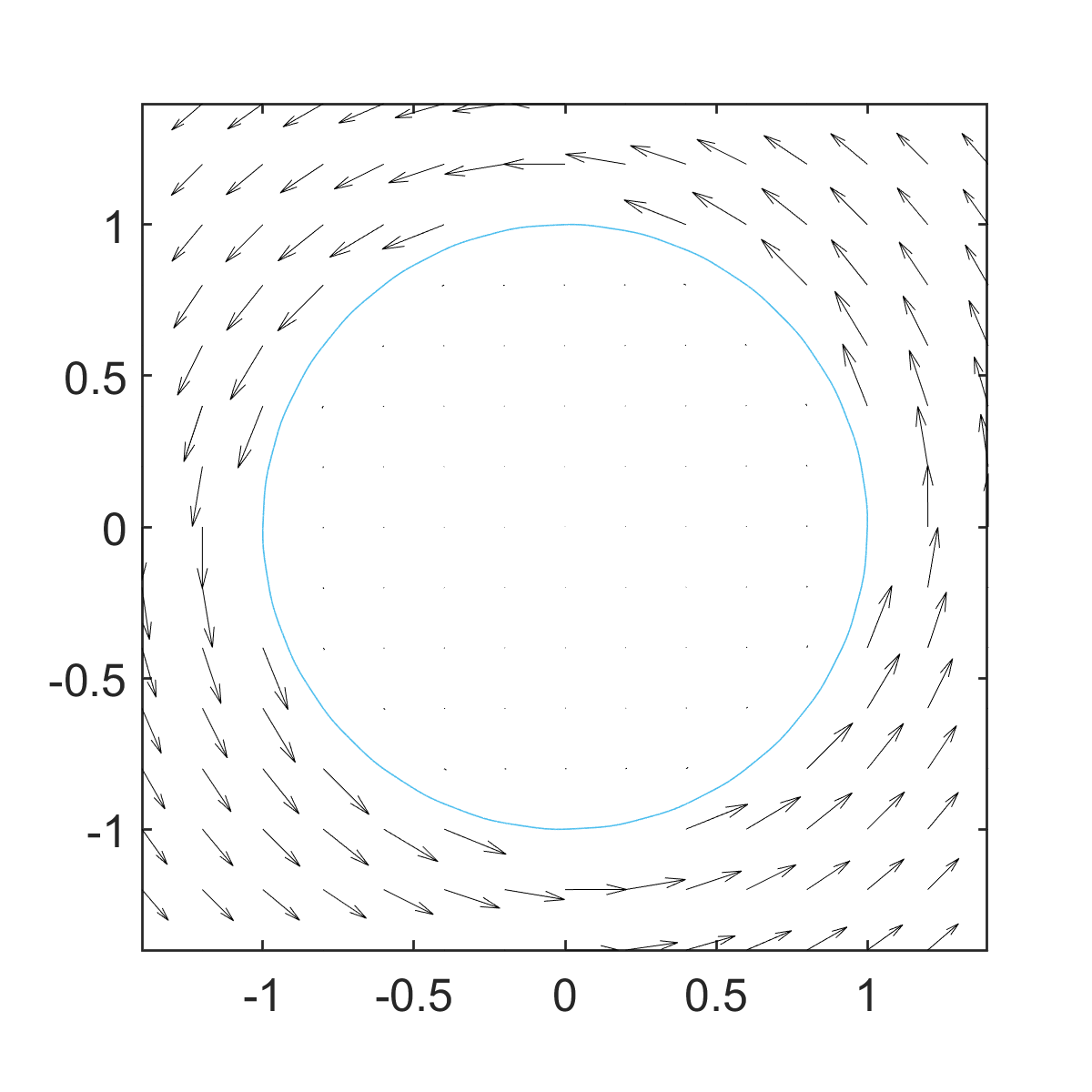}}  
\subfigure{\label{fig:01}\includegraphics[width=0.32\textwidth]{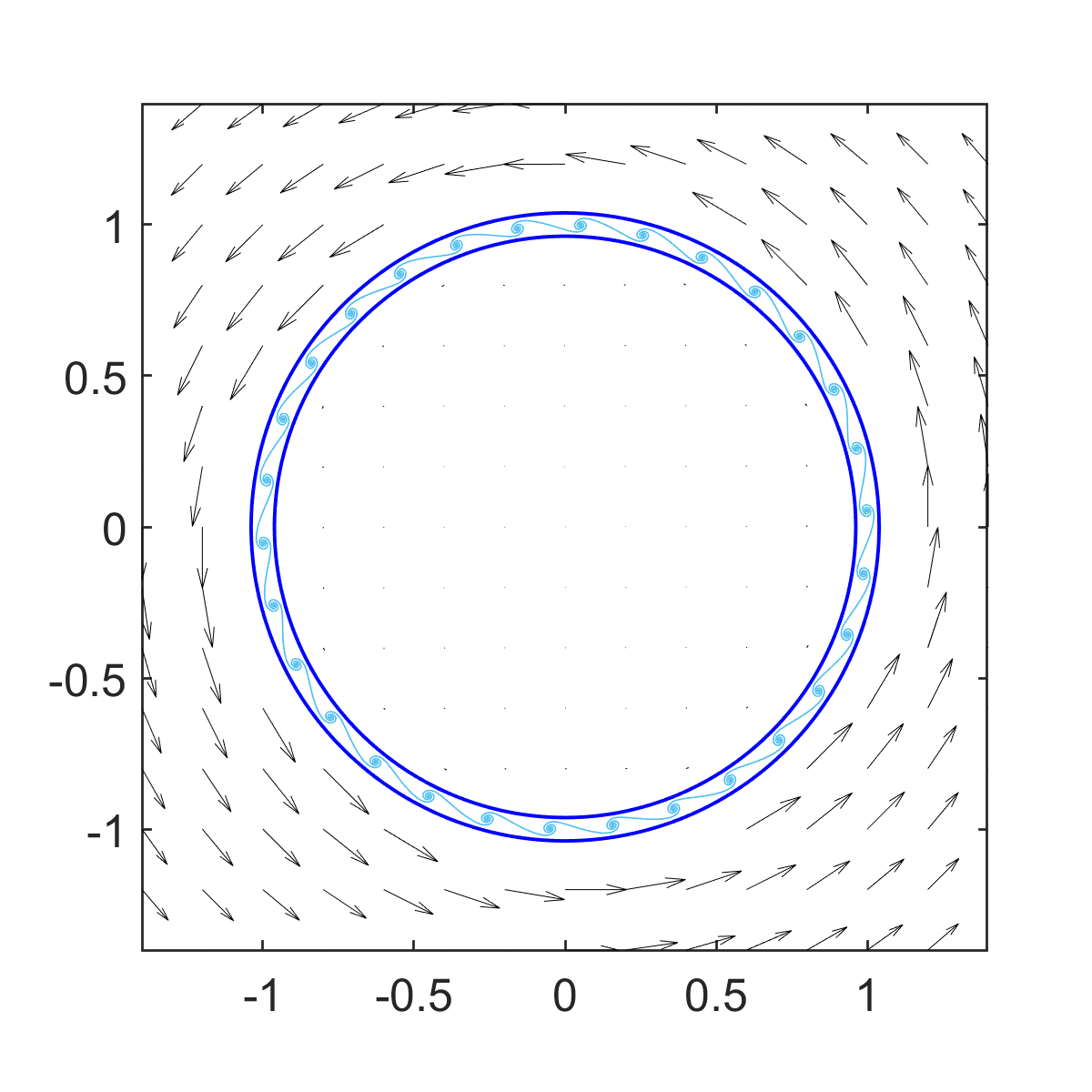}}   
\subfigure{\label{fig:02}\includegraphics[width=0.32\textwidth]{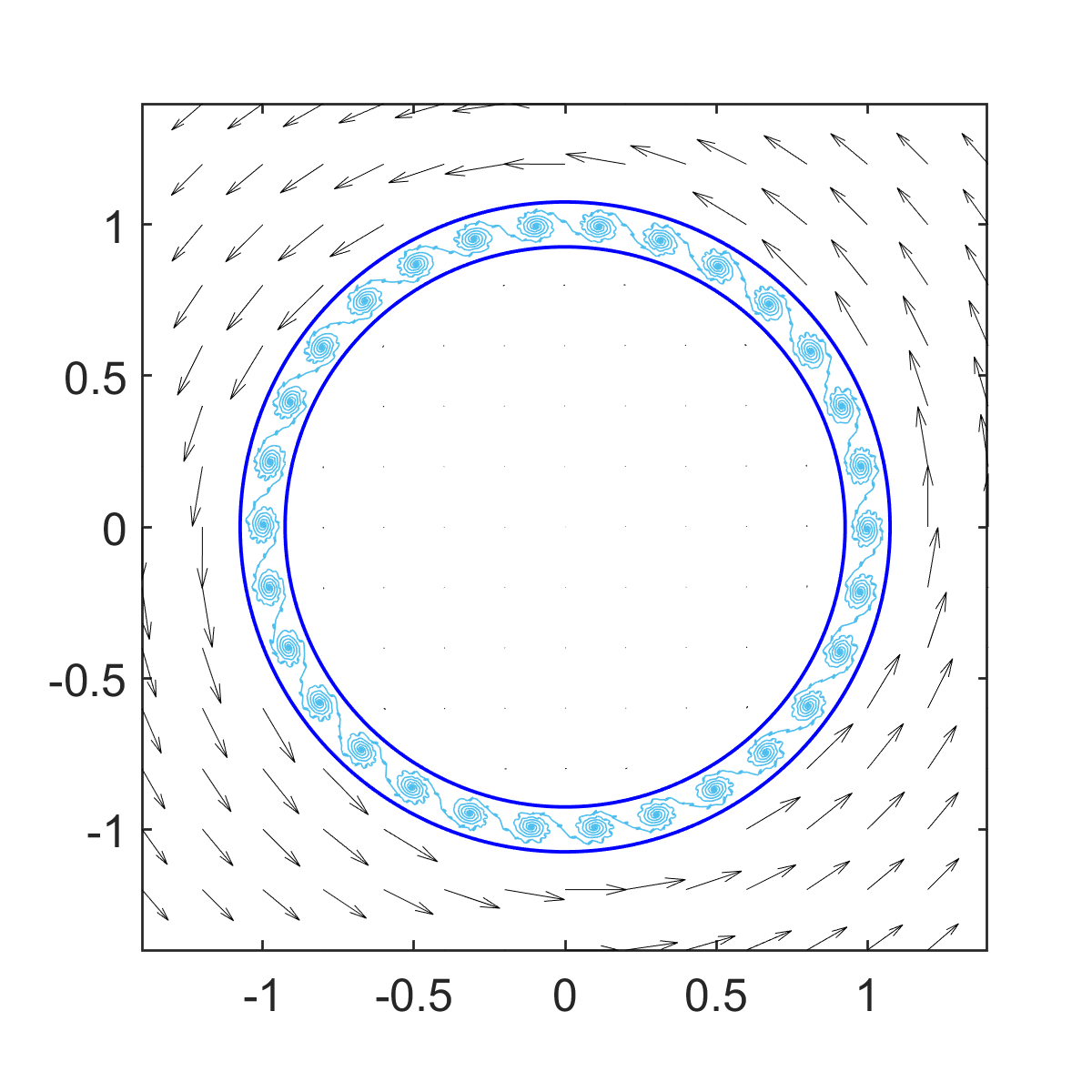}} 
\subfigure{\label{fig:20}\includegraphics[width=0.32\textwidth]{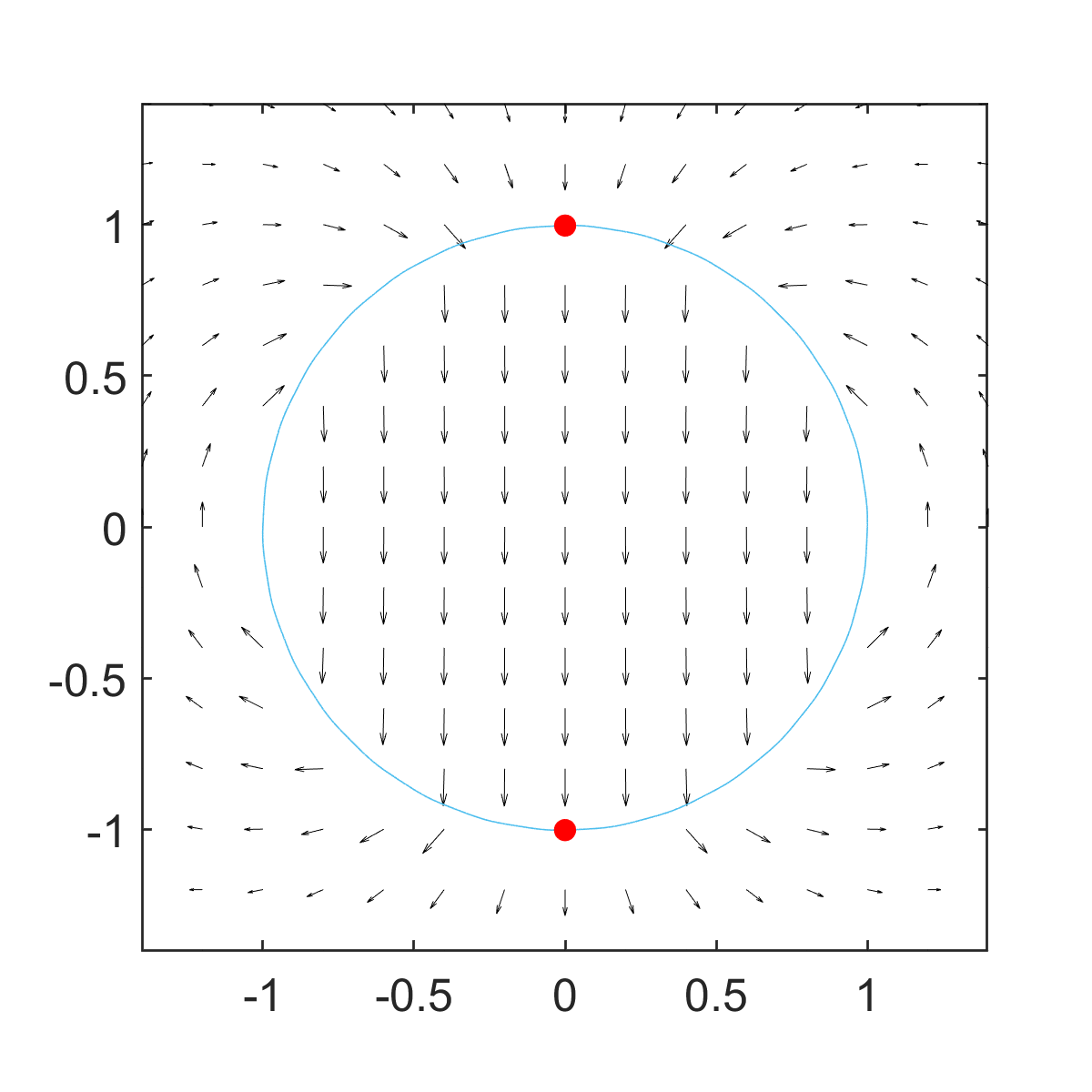}}  
\subfigure{\label{fig:21}\includegraphics[width=0.32\textwidth]{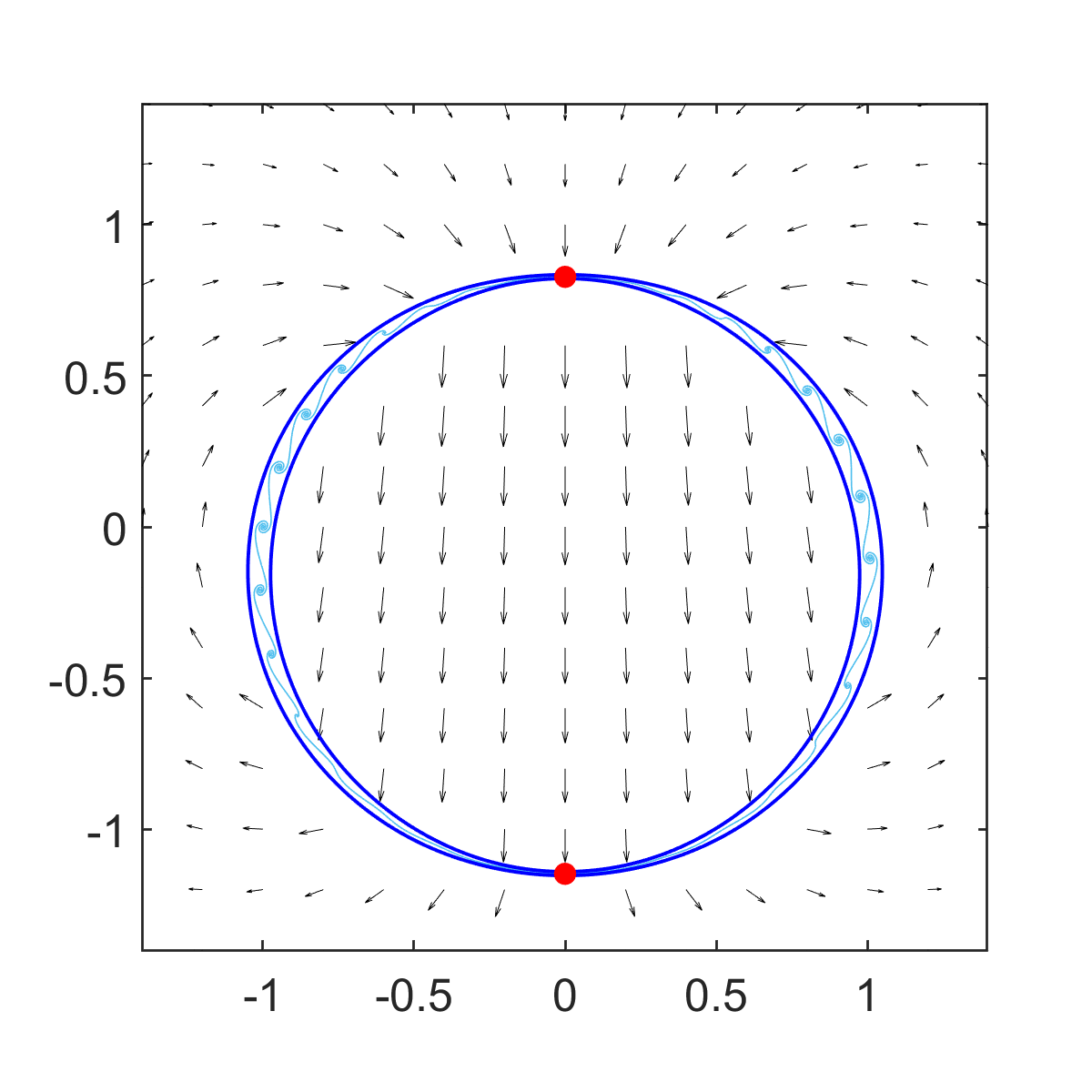}}   
\subfigure{\label{fig:22}\includegraphics[width=0.32\textwidth]{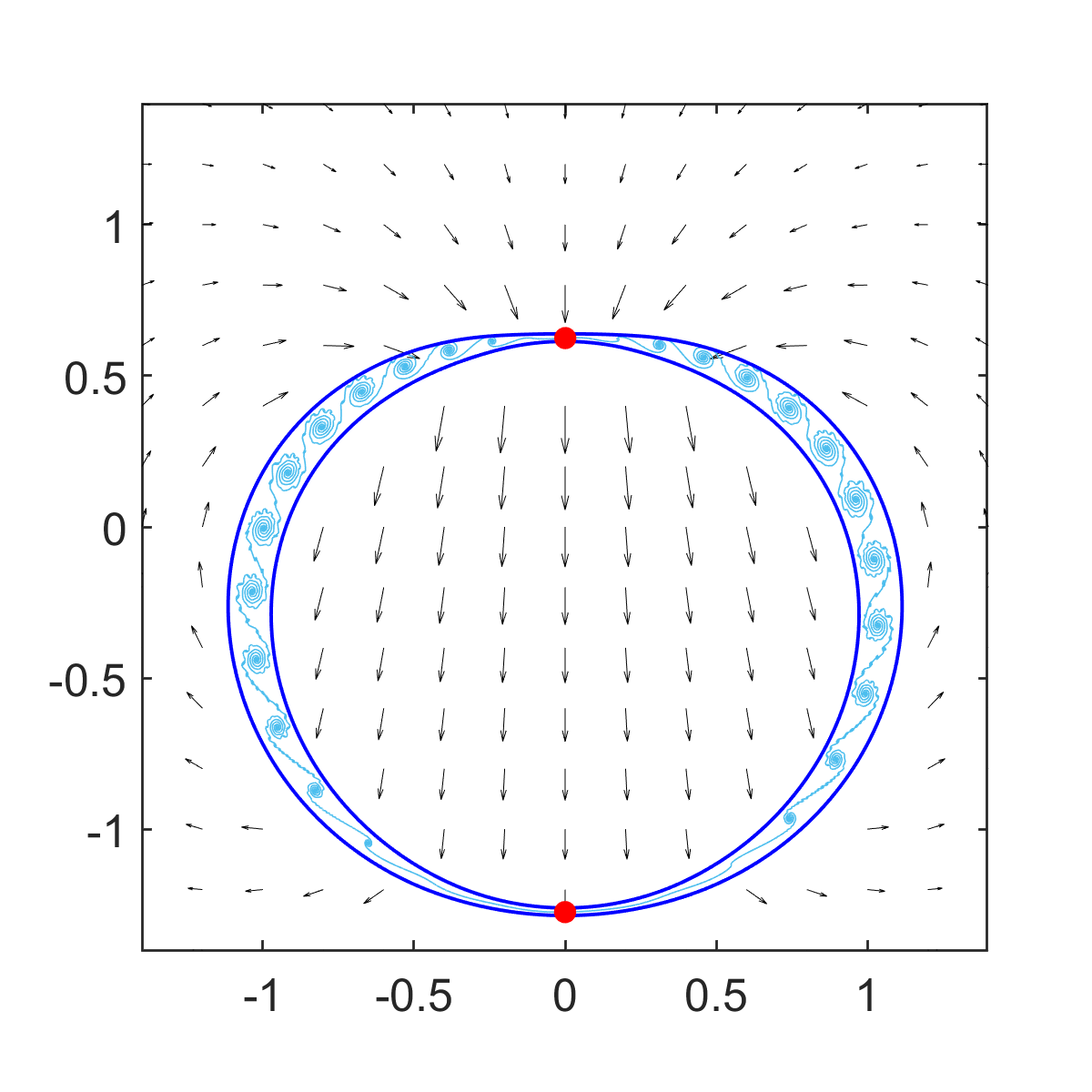}} 
\subfigure{\label{fig:40}\includegraphics[width=0.32\textwidth]{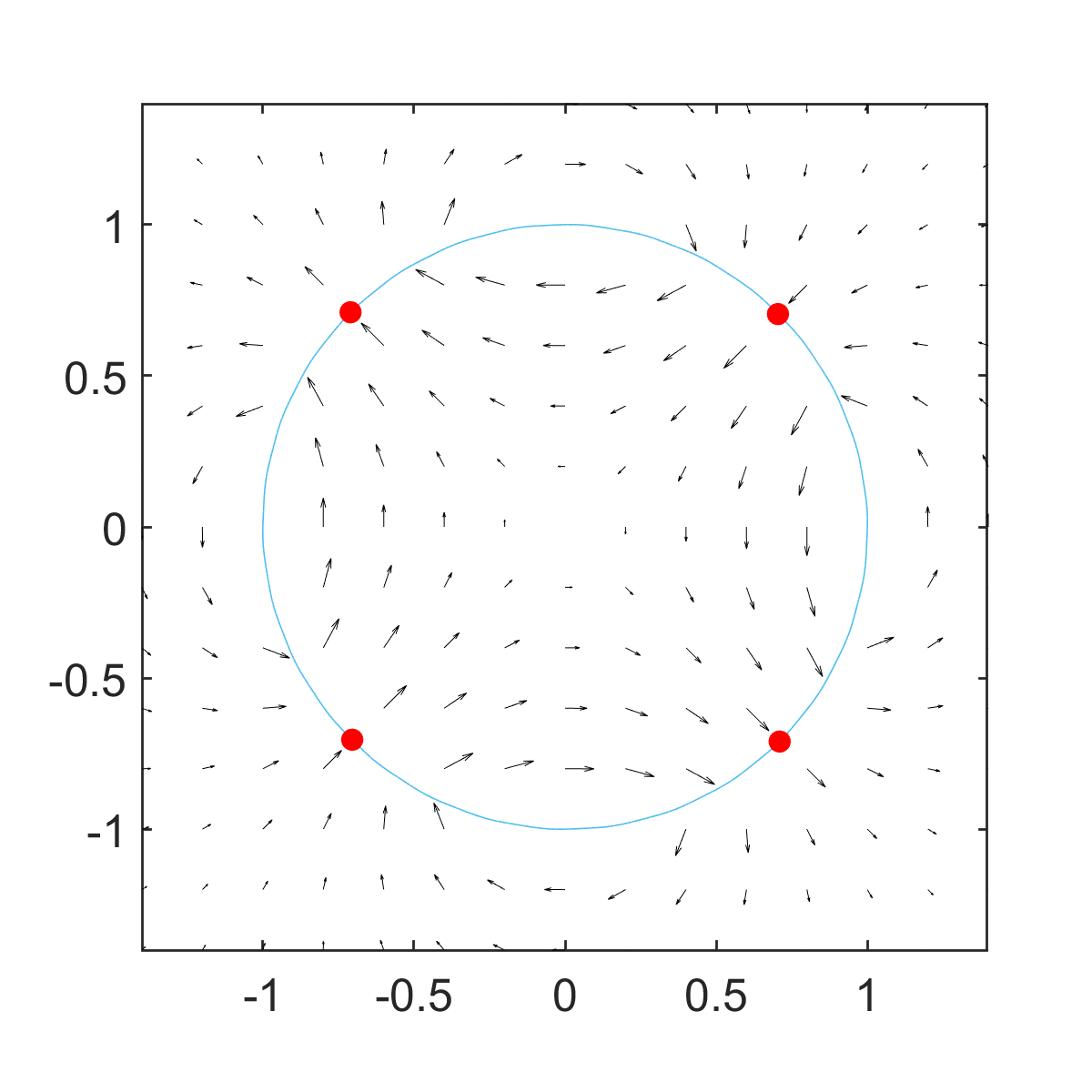}}  
\subfigure{\label{fig:41}\includegraphics[width=0.32\textwidth]{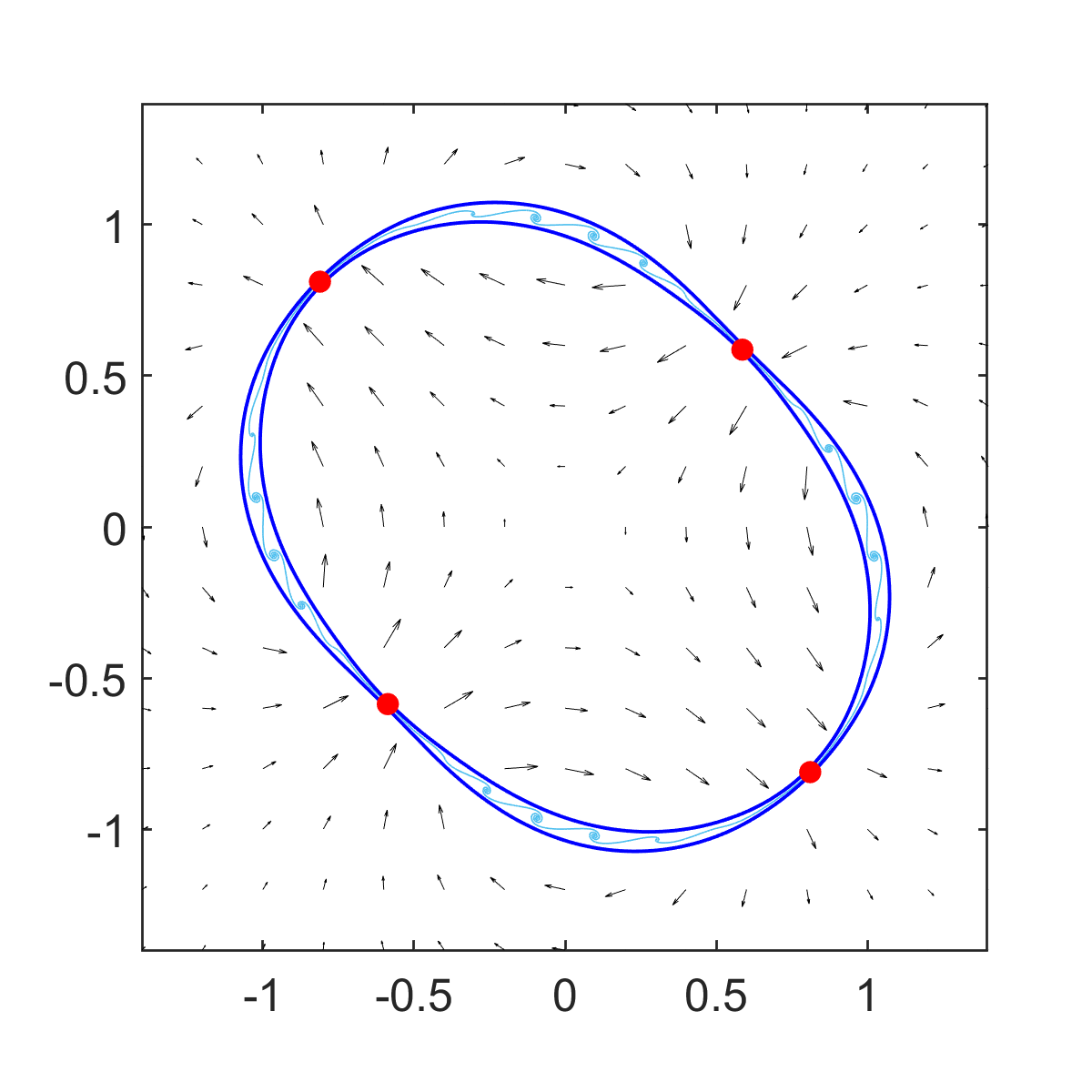}}   
\subfigure{\label{fig:42}\includegraphics[width=0.32\textwidth]{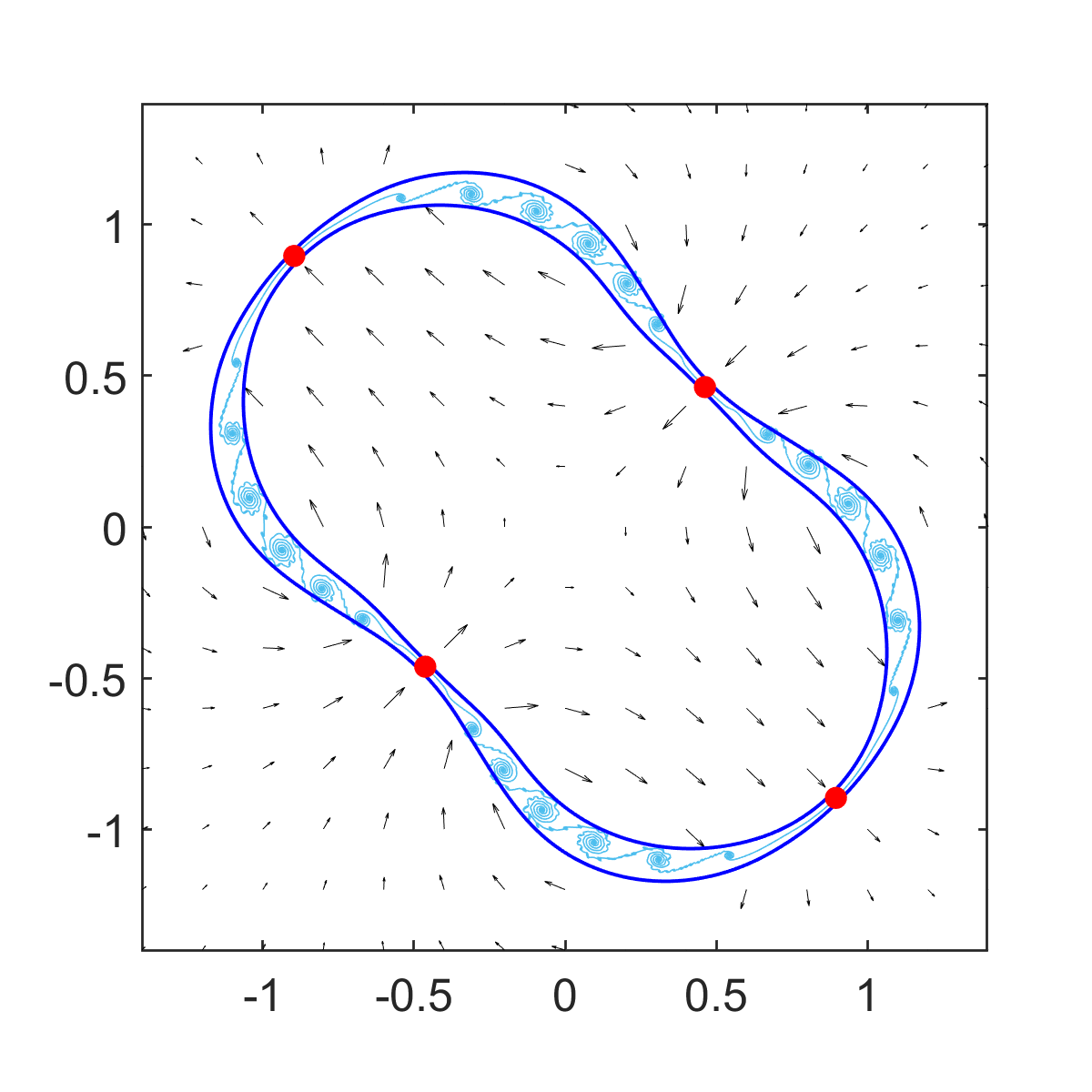}} 
\caption{From top to bottom, $\varpi_0(s)=\tfrac{1}{4}$, $\tfrac{1}{4}\cos(s)$ and $\tfrac{1}{4}\cos(2s)$.}              
\end{figure}

%\vfill\newpage
\centering

\end{document}